\documentclass[12pt]{elsarticle}
\usepackage{amsmath}
\usepackage{amssymb}
\usepackage{amsthm}
\usepackage{verbatim}
\usepackage{subfigure}
\usepackage{latexsym}
\usepackage{lmodern}
\usepackage{array}
\usepackage{graphicx}
\usepackage{float}
\usepackage{multirow}
\usepackage{listings}
\usepackage{color}
\usepackage{xcolor}
\usepackage{bm}
\usepackage{url}
\usepackage[english]{babel}
\usepackage{hyperref}
\usepackage{algorithm}
\usepackage{algorithmic}
\usepackage{multicol,balance}
\usepackage{times}
\usepackage{cases}
\usepackage{titlesec}
\usepackage[title]{appendix}
\usepackage{mdwlist}
\usepackage{makecell}
\usepackage{booktabs}
\usepackage{mathabx}
\usepackage{epstopdf}
\usepackage{framed}
\usepackage{tikz}
\usetikzlibrary{backgrounds}
\usetikzlibrary{patterns,patterns.meta}
\allowdisplaybreaks[4]

\definecolor{bblue}{RGB}{204,248,255}
\definecolor{rred}{RGB}{255, 219, 213}
\definecolor{oran}{RGB}{255, 187, 153}

\newcommand{\cI}{\mathcal{I}}

\newcommand{\R}{\mathbb{R}}

\newcommand*\xbar[1]{%
	\hbox{%
		\vbox{%
			\hrule height 0.5pt 
			\kern0.5ex
			\hbox{%
				\kern-0.1em
				\ensuremath{#1}%
				\kern-0.1em
			}%
		}%
	}%
}

\newtheorem{lemma}{Lemma}[section]

\newtheorem{prop}{Proposition}[section]

\newtheorem{theorem}{Theorem}[section]
\newtheorem{remark}{Remark}[section]

\newtheorem{definition}{Definition}[section]

\DeclareMathOperator{\supp}{supp}

\numberwithin{equation}{section}

\normalsize \makeatletter

\newcommand*\bbar[1]{%
	\hbox{%
		\vbox{%
			\hrule height 0.5pt 
			\kern0.365ex
			\hbox{%
				\kern-0.1em
				\ensuremath{#1}%
				\kern-0.1em
			}%
		}%
	}%
} 

\begin{document}
	
	\begin{frontmatter}
		\title
		{
			Weak-strong uniqueness of the full coupled Navier-Stokes and Q-tensor system in dimension three
			\footnote{
				The authors are partially supported by NSFC-W2531006, NSFC-12250710674, NSFC-11831003, NSFC-12031012 and the Institute of Modern Analysis-A Frontier Research Center of Shanghai.
				The first author also gratefully acknowledges the support of the Scientific Research Foundation for Advanced Talent of Nantong University under Grant 135424639079 and the Jiangsu Innovation and Entrepreneurship Talent Program under Grant JSSCBS20250352.}
		}

		\author[rvt1]{Fan Yang}
		\ead{fanyang-m@ntu.edu.cn}
		
		\author[rvt2]{Junjie Zhou\corref{corresponding}}
		\ead{2023zjj@sjtu.edu.cn}	
		
		\cortext[corresponding]{Corresponding author.}
		
		\address[rvt1]{School of Mathematics and Statistics, Nantong University, Nantong, 226019, China}	
		
		\address[rvt2]{School of Mathematical Sciences, Shanghai Jiao Tong University, Shanghai, 200240, China}
		
		
		\begin{abstract}
			In this paper, we study the weak-strong uniqueness for the Leray-Hopf type weak solutions to the Beris-Edwards model of nematic liquid crystals in $\R^3$ with an arbitrary parameter $\xi\in\R$, which measures the ratio of tumbling and alignment effects caused by the flow.
			This result is obtained by proposing a new uniqueness criterion in terms of $(\Delta Q,\nabla u)$ with regularity $L_t^qL_x^p$ for $\frac{2}{q}+\frac{3}{p}=\frac{3}{2}$ and $2\leq p\leq 6$, which enables us to deal with the additional nonlinear difficulties arising from the parameter $\xi$.
			Compared with the known results, our finding reveals that the criterion of weak-strong uniqueness for $\xi\ne 0$ is a sub-regime of the one for the corotational case. The associated regularity assumption rises with the nonlinearity of the model.
            Moreover, we establish the global well-posedness of this model for small initial data in $H^s$-framework.
			
%
			
		\end{abstract}
		
	\end{frontmatter}
	%
	
	%
	
	\section{Introduction}
	In recent decades, the modeling and analysis of the hydrodynamics of nematic liquid crystals have attracted much attention.
	In general, according to the different ways to describe the orientation of liquid crystal molecules,
	the Ericksen-Leslie model \cite{Eri1962,Eri1976,Les1968} and the Beris-Edwards model \cite{EB1994}
	have been proposed to model nematic liquid crystals.
	Among these models,
	the Ericksen-Leslie model is a system that couples a forced Navier-Stokes equation for the underlying fluid velocity field with an evolution equation for the unit-vector field $d\in\mathbb{S}^{d-1}$ that represents the macroscopic molecular orientations.
	However, the vector field is not the best choice for describing the defect structures
	which are important physical phenomena in liquid crystals \cite{defect}. In fact, the most comprehensive model is the Beris-Edwards model in which a symmetric, traceless $3\times 3$ matrix $Q$ is used to describe the orientation and degree of ordering of the liquid crystals \cite{Q-tensor-book}.
	This has the advantage that the hydrodynamics of topological defects in liquid crystals can be included.
	For more background and discussion on various liquid crystal systems, we refer interested readers to \cite{BFO-book-2017,Frank1958,hiber2018,Oseen1933,Virga2012,Virga2018,Zarnescu2021,WZZ2021} and the references therein.

	In this paper, we shall be concerned with the following incompressible Beris-Edwards model in $\R^3$:
	\begin{equation}\label{eq1.1}
		\begin{cases}
			\partial_t Q + u\cdot\nabla Q - S(\nabla u,Q) =\Gamma H,\\
			\partial_t u + u\cdot\nabla u = \mu\Delta u-\nabla p + \nabla\cdot \left( \tau+\sigma \right),\\
			\nabla\cdot u = 0, \\
			(Q,u)(t,x)\mid_{t=0}=(Q_0,u_0),\hspace{1.4cm} x\in \R^3,\quad t>0,
		\end{cases}
	\end{equation}
	where $u \in \R^3$ is the flow velocity,
	$Q\in \mathbb{M}^{3\times 3}$ denotes the nematic tensor order parameter and is often referred to as "$Q$-tensor", $p$ is the usual pressure, $\mu>0$ represents the viscosity coefficient, and $\Gamma>0$ is a collective rotational diffusion constant. Meanwhile, $S(\nabla u,Q)$, $H$, $\tau$ are symmetric tensors and $\sigma$ is a skew-symmetric tensor given by
	\begin{align}
		S(\nabla u, Q):=&\xi D\left( Q+\dfrac{1}{3}\mathbb{I}_3 \right) + \xi \left( Q+\dfrac{1}{3}\mathbb{I}_3 \right)D - 2\xi \left( Q+\dfrac{1}{3}\mathbb{I}_3 \right){\rm Tr}(Q\nabla u) + \Omega Q- Q\Omega,\\
        H:=&H[Q]= L\Delta Q-aQ+b\left[Q^2-\dfrac{{\rm Tr}(Q^2)}{3}\mathbb{I}_3\right]-cQ{\rm Tr}(Q^2),\\
		\tau:=&-\xi \left( Q+\dfrac{1}{3}\mathbb{I}_3 \right)H - \xi H\left( Q+\dfrac{1}{3}\mathbb{I}_3 \right) + 2\xi \left( Q+\dfrac{1}{3}\mathbb{I}_3 \right){\rm Tr}(QH) - L\nabla Q \odot\nabla Q,\\
		\sigma:=&QH-HQ=Q\Delta Q-\Delta QQ,
	\end{align}
	and $L, c>0$ and $a, b, \xi\in \R$ are constants, $\mathbb{I}_3$ denotes the $3\times 3$ identity matrix, and $D:=\frac{1}{2}\left( \nabla u + \nabla u^{T} \right)$ and $\Omega:=\frac{1}{2}\left( \nabla u - \nabla u^{T} \right)$ are the symmetric and antisymmetric part of the strain tensor with $(\nabla u)_{\alpha\beta}=\partial_\beta u_\alpha$. Moreover, $\nabla Q \odot\nabla Q$ is a $3\times 3$ matrix with its component $(\nabla Q \odot\nabla Q)_{\alpha\beta}$ given by $\partial_\beta Q_{\gamma\delta}\partial_\alpha Q_{\gamma\delta}$.
	In what follows, we use a partial Einstein summation convention, i.e., the
	repeated indices are summed over.
	
	More specifically, we remark that $Q=0$ corresponds to the isotropic phase of the nematic liquid crystals.
	If $Q$ has two equal non-zero eigenvalues, it is known as uniaxial liquid crystals, and $Q$ admits three different eigenvalues corresponds to the case of biaxial liquid crystals. In addition, $S(\nabla u,Q)$ appears because the order parameter $Q$ can be both rotated and stretched by flow gradients, and $\xi$ represents the ratio of the tumbling and aligning effects on the molecules exerted by a shear flow. Moreover, $\xi=0$ corresponds to the corotational case, which describes the molecules only tumble in a shear flow but do not align.
	Furthermore, the molecular tensor $H$ is related to the derivative of the Landau-de Gennes free energy $\mathcal{F}$ by
	\[
	H=-\frac{\delta \mathcal{F}}{\delta Q}+ \frac{1}{3}\mathbb{I}_3{\rm Tr}\frac{\delta \mathcal{F}}{\delta Q}
	\]
	and
	\[
	\mathcal{F}=\int \left( \dfrac{L}{2}|\nabla Q|^2  + \dfrac{a}{2} {\rm Tr}(Q^2) - \dfrac{b}{3} {\rm Tr}(Q^3) + \dfrac{c}{4}| {\rm Tr}(Q^2) |^2 \right) dA,
	\]
	see also \cite{DOY-2001,Q-tensor-book}.
	
	The Beris-Edwards models for liquid crystals have been studied extensively in the literature.
	The first result regarding well-posedness for the system \eqref{eq1.1}
	was proposed by Paicu and Zarnescu \cite{Qtensor12}. In particular, they proved the existence of global weak solutions in $\R^d$ $(d=2,3)$ for the case $\xi=0$, as well as
	higher global regularity and weak-strong uniqueness for dimension two. Later on, they generalized their results to the case of sufficiently small $\xi$ \cite{Qtensor11}. In fact, as was shown in \cite{2dwithoutxi}, this smallness assumption on $\xi$ is unnecessary for the existence of regular solutions in two dimensional periodic case. 
	Moreover, a uniqueness criterion of weak solutions was made in \cite{weak-Q-tensor}.
	Recently, by employing the maximal regularities of the Stokes operator and parabolic operator in Besov spaces, Xiao \cite{Q-unique-3d} studied the existence of global strong solution in a smooth and bounded domain $\mathcal{O}\subset\R^3$ with small initial data for $\xi =0$, as well as weak-strong uniqueness. Very recently, the strong well-posedness for general $\xi\in \R$ was obtained by Hieber, Hussein and Wrona in \cite{xi-all}.
	For the case of a bounded domain with homogeneous/inhomogeneous boundary data, the corresponding well-posedness results can be found in \cite{ables2014,ables2016}.
	In addition,
	Feireisl, Rocca, Schimperna and Zarnescu \cite{nonisothermal1,nonisothermal2} also derived a nonisothermal variant of the Beris-Edwards model with the Ball-Majumdar potential \cite{potential}. Then, they established the existence of global weak solutions for this system under periodic boundary conditions. For the coupled compressible Navier-Stokes and $Q$-tensor system, Wang, Xu and Yu studied the existence and long-time dynamics of global weak solutions \cite{Q-tensor-comp}.
	Furthermore, many
	progress have been shown concerning well-posedness and decay properties of 
	some modified Beris-Edwards models in terms of its free energy or stress tensor; refer to \cite{2dlowSobolev,A-Z2016-2d,Du-wang-20,Weak-t-regularity,Hang-Ding-15,Liu-dcdsb,LC-decay-22,LC-decay-19,Q-tensor-Torus} and references therein.
	
	The aim of this paper is to give a rather complete understanding of
	weak-strong uniqueness
	for the system \eqref{eq1.1} with an arbitrary ratio of tumbling and alignment effects $\xi\in\R$ in $\R^3$.
	First, let us come back for a moment to the setting of the incompressible Navier-Stokes equations for the motivation of this work. It is well-known that there exist Leray-Hopf weak solutions, which were established by Leray \cite{Leary1934} and Hopf \cite{hopf1950}; their uniqueness is still unknown in general. However, we know that if $u$ is a Leray-Hopf weak solution which further satisfies
	\begin{equation}\label{eq:LPS}
		u\in L_t^qL_x^p \quad \text{for some $q\leq\infty$ and $p\geq d$ with $\dfrac{2}{q}+ \dfrac{d}{p}\leq 1$},
	\end{equation}
	then all Leray-Hopf weak solution with the same initial data must coincide (see \cite{kozono1996,Ladyzhenskaya,prodi1959,serrin1963}).
	This condition \eqref{eq:LPS} is known as the Ladyzhenskaya-Prodi-Serrin condition, and such uniqueness criteria for weak solutions are referred to as "weak-strong uniqueness".
	Since Ladyzhenskaya-Prodi-Serrin condition also implies the regularity properties of Leray-Hopf weak solutions, we refer the interested readers to \cite{L3infity,serrin1962,Struwe1988} and the references therein.
	On the other hand, in the context of mathematical analysis for very weak solutions with regularity $L_t^qL_x^p$, it has been proved that uniqueness holds for $\frac{2}{q}+ \frac{d}{p}\leq 1$ with $p\geq d$; see \cite{Fabes1972,giga1986,kato1984,lions2001}. Only very recently, many attempts have been made to investigate the non-uniqueness problem of very weak solutions in the class $L_t^qL_x^{p}$ with $\frac{2}{q}+ \frac{d}{p}> 1$. The first result in this direction was proved by Buckmaster and Vicol \cite{buck2019}. In particular, using the convex integration method, they established the non-uniqueness of very weak solutions to the Navier-Stokes equations in the class $C_tL_x^{2+}$. Moreover, in \cite{nonunique2022,non-2d}, Cheskidov and Luo showed the non-uniqueness both in the class $L_t^qL_x^{\infty}$ ($1\leq q<2$, $d\geq 2$) and $C_tL_x^p$ ($1\leq p<2$), which indicate the sharp non-uniqueness near two endpoints of Ladyzhenskaya-Prodi-Serrin condition. We also refer the readers to \cite{convex_ns-review,review20,Luo_H1/200,Miao-Nie-Ye-1,Miao-Nie-Ye-2,Miao-Zhao-25} for more results,
	and \cite{BLR-ill-2018,Luo-Titi,Li-MHD-1,Li-MHD-2,FNSE-Leray} for some recent progress on other hydrodynamic models.
	
	
	When $\xi=0$, with the aid of Littlewood-Paley theory and logarithmic embedding inequality,
	it has been proved in \cite{Qtensor12} that system \eqref{eq1.1} has solutions with higher regularity in $\R^2$, i.e.,
	\begin{equation}
		Q\in L^\infty\left( 0,T;H^{s+1} \right)\cap L^2\left( 0,T;H^{s+2} \right),\quad u\in L^\infty\left( 0,T;H^s \right)\cap L^2\left( 0,T;H^{s+1} \right),
	\end{equation}
	where the initial data $(Q_0,u_0)\in H^{s+1}\times H^s$ for $s>0$. Then, they obtained the associated weak-strong uniqueness. Note that the logarithmic embedding inequality fails in dimension three. Thus, we cannot get a similar result because of the difficulty in addressing the higher order energy estimates that used to close the a priori estimate. However, a uniqueness criterion similar to that of Serrin for Navier-Stokes equations has been made in \cite{weak-Q-tensor}, that is,
	\begin{equation}\label{eq1.8}
		(\Delta Q, \nabla u )\in L^q(0,T;L^p)\quad\text{with}\quad \dfrac{2}{q}+\dfrac{3}{p}=2,\quad \forall p\in [2,3].
	\end{equation}
	This result implies the weak-strong uniqueness of \eqref{eq1.1} holds in a smooth bounded domain $\mathcal{O}\subset \R^3$ for the case $\xi=0$. In addition, we mention that the existence of strong solutions and weak-strong uniqueness were also studied in \cite{Q-unique-3d} via the $L^p$-$L^q$ maximal regularity.
	
	In the setting of $\xi\ne 0$, the Beris-Edwards model exhibits the additional tumbling and alignment effects of molecules caused by the flow.
	Note that the molecules of liquid crystals can align in a shear flow. Therefore,
	in order to understanding
	their complex dynamics behavior, it is worth investigating the model with a general $\xi\in\R$.
	However, in this case, the system \eqref{eq1.1} becomes much more complicated compared with that for the corotational case. As seen in \eqref{eq1.1},
	several strongly nonlinear terms appear and create significant challenges in the mathematical analysis of this model.
	Although the authors of \cite{Qtensor11} proved the existence of global weak solutions in two and three dimensional cases when $\xi$ is sufficiently small, weak-strong uniqueness was only established in $\R^2$ for the same reasoning as in \cite{Qtensor12}.
	In fact, one can build on the work in \cite{xi-all} to show weak-strong uniqueness of system \eqref{eq1.1} in smooth bounded domains
	for the case $\xi\in\R$ (see also \cite{Q-unique-3d}).
	Here we are more interested in whether uniqueness criterion \eqref{eq1.8} for the simplified model can be achieved for arbirtray
	$\xi\in\R$ as well.
	Moreover, we emphasize that previous works in this direction are all limited to the situation of $\xi=0$. See also \cite{E-L-weak-strong} for the simplified Ericksen-Leslie system and \cite{weak-strong-3d} for the active liquid crystals.
	
	In our recent work \cite{weak-strong-3d}, weak-strong uniqueness for the incompressible active liquid crystals was obtained in $\R^3$ under condition \eqref{eq1.8}. Formally, in the case of constant concentration of active particles, the analyzed model can be regarded as a modified system \eqref{eq1.1} by adding some phenomenological active terms and assume that $\xi$ is zero.
	Interestingly, one can see that the uniqueness criterion is the same as in \cite{weak-Q-tensor}, even its structure is more complex.
	It appears that these simplified active terms
	do not impose any further regularity assumption on the weak-strong uniqueness. But for the models where active concentration can be changed,
	or in the situation of $\xi\ne0$, it may not be tenable;
	see Figure~\ref{fig1} below for a comparison.
	
	Let us note that our study is in line with the work \cite{serrin1963} of Serrin concerning weak-strong uniqueness of the incompressible Navier-Stokes equations.
	It is motivated by his approach that the weak solution specified herein should satisfy certain energy inequality.
	Indeed, as in \cite{weak-strong-3d},
	the definition of Leray-Hopf type weak solutions to the incompressible Beris-Edwards model for general values of $\xi\in\R$ can be similarly stated as follows:

%

	\begin{definition}[Leray-Hopf type weak solution]\label{def1}
		A pair $(Q,u)$ is called a Leray-Hopf type weak solution to system \eqref{eq1.1} with the initial data $(Q_0,u_0)\in H^1(\R^3, S_0^3)\times L_\sigma^2(\R^3)$, if the following conditions hold:
		\begin{enumerate}[(i).]
			\item $Q\in L_{\rm loc}^\infty(\R_+;H^1(\R^3)) \cap L_{\rm loc}^2(\R_+;H^2(\R^3))$ and $u\in L_{\rm loc}^\infty(\R_+;L^2(\R^3)) \cap L_{\rm loc}^2(\R_+;H^1(\R^3))$;
			\item For any compacted supported $\phi\in C^\infty([0,\infty)\times \R^3;\mathcal{S}_0^3)$ and $\psi\in C^\infty([0,\infty)\times \R^3;\R^3)$ with $\nabla\cdot \psi=0$, we have
			\begin{equation}\label{eq1.2}
				\begin{split}
					\int_0^\infty \left( Q:\partial_t\phi \right)&+\Gamma L (\Delta Q:\phi) +(Q:u\cdot\nabla \phi) -2\xi \left( \left[Q +\dfrac{\mathbb{I}_3}{3}\right] {\rm Tr}(Q\nabla u):\phi \right) dt\\
					&+ \int_0^\infty\left( (\xi  D+\Omega)\left[Q +\dfrac{\mathbb{I}_3}{3}\right] :\phi \right) + \left( \left[Q +\dfrac{\mathbb{I}_3}{3}\right] (\xi D-\Omega):\phi \right)dt\\
					=& \Gamma\int_0^\infty \left( aQ-b\left[Q^2-\dfrac{{\rm Tr}(Q^2)}{3}\mathbb{I}_3\right]+cQ{\rm Tr}(Q^2) :\phi \right)dt -(Q_0(x):\phi(0,x))
				\end{split}
			\end{equation}
			and
			\begin{equation}\label{eq1.3}
				\begin{split}
					&\int_0^\infty \left( u,\partial_t\psi \right)+(u,u\cdot\nabla \psi)-\mu (\nabla u,\nabla\psi)dt +(u_0(x),\psi(0,x)) \\
					=& -L\int_0^\infty \left( \nabla Q \odot\nabla Q - Q\Delta Q+\Delta QQ :\nabla\psi \right)dt  + 2\xi\int_0^\infty\left(  \left[Q +\dfrac{\mathbb{I}_3}{3}\right]{\rm Tr}(QH): \nabla\psi \right)dt \\
					&-\xi\int_0^\infty\left( H\left[Q +\dfrac{\mathbb{I}_3}{3}\right]+  \left[Q +\dfrac{\mathbb{I}_3}{3}\right]H: \nabla\psi \right)dt .
				\end{split}
			\end{equation}
			\item Let $T\in (0,+\infty)$, then $(Q,u)$ satisfies the following energy inequality for all $t\in [0,T]:$
			\begin{equation}\label{eq1.4}
				\begin{split}
					\|u\|_{L^2}^2& + \|Q\|_{L^2}^2+ L \|\nabla Q\|_{L^2}^2 + 2\mu \int_{0}^{t} \|\nabla u\|_{L^2}^2 ds\\
					& +2a\Gamma  \int_{0}^{t} \|Q\|_{L^2}^2ds +2(a+1)\Gamma L\int_{0}^{t} \|\nabla Q\|_{L^2}^2ds + 2\Gamma L^{2} \int_0^t \|\Delta Q\|_{L^2}^2
					ds \\
					\leq & \|u_0\|_{L^2}^2 +\|Q_0\|_{L^2}^2 + L \|\nabla Q_0\|_{L^2}^2  + 4(1-a)\xi \int_0^t \left( Q\left[Q + \dfrac{\mathbb{I}_3}{3}\right]:\nabla u\right) -\left(  Q {\rm Tr}(Q\nabla u):Q\right)ds\\
					&  +2\Gamma\int_0^t \left( b\left[Q^2-\dfrac{{\rm Tr}(Q^2)}{3}\mathbb{I}_3\right]-cQ{\rm Tr}(Q^2) :Q - L \Delta Q \right)ds\\
					&  +4\xi\int_0^t \left( \left[Q + \dfrac{\mathbb{I}_3}{3}\right]\left\lbrace b\left[Q^2-\dfrac{{\rm Tr}(Q^2)}{3}\mathbb{I}_3\right]-cQ{\rm Tr}(Q^2)\right\rbrace :\nabla u \right)ds\\
					& - 4\xi\int_0^t \left( b\left[Q^2-\dfrac{{\rm Tr}(Q^2)}{3}\mathbb{I}_3\right]-cQ{\rm Tr}(Q^2) :Q {\rm Tr}(Q\nabla u) \right)ds.
				\end{split}
			\end{equation}
		\end{enumerate}
	\end{definition}
	\begin{remark}
			As was shown in \cite{Qtensor11}, there exist a global weak solution of system \eqref{eq1.1} subject to the initial data $(Q_0,u_0)\in H^1 \times L_\sigma^2$ for small $\xi$, which satisfies the above conditions (i) and (ii).
			However, the question of whether such a weak solution
			satisfies the energy inequality \eqref{eq1.4} remains unclear.
			The same as in \cite{weak-strong-3d},
			it seems that the regularity provided by weak solutions is not enough to justify the convergence property of some nonlinear coupling terms; for instance,
			\[
			\begin{split}
				&\lim_{\epsilon\rightarrow 0}\int_0^t \left( Q^\epsilon\Omega^\epsilon -\Omega^\epsilon Q^\epsilon:\Delta Q^\epsilon \right)- \left(Q^\epsilon\Delta Q^\epsilon-\Delta Q^\epsilon Q^\epsilon:\nabla u^\epsilon\right)ds\\
				=&\int_0^t \left( Q\Omega -\Omega Q:\Delta Q \right)- \left(Q\Delta Q-\Delta Q Q:\nabla u\right)ds,
			\end{split}
			\]
			which thus vanishes according to certain cancellation rules.
			Nevertheless, it is reasonable to consider Leray-Hopf type weak solutions for system \eqref{eq1.1}, because one can show that the energy inequality \eqref{eq1.4} becomes a strict equality under an additional condition similar to \eqref{eq1.8}. Together with Theorem~\ref{thm3.1} below, it is obvious that regular solutions established for the initial data $(Q_0,u_0)\in H^3 \times H^2$ satisfy the energy inequality. From this point, we believe that such a Leray-Hopf type weak solution can be derived for initial data with an additional assumption. However, our original interest lies in understanding the existence of such weak solutions to system \eqref{eq1.1} with $(Q_0,u_0)\in H^1 \times L_\sigma^2$, and we hence leave it for future work.
			
	\end{remark}
	
	Our first main result concerns the energy equality of system \eqref{eq1.1}, which is essential for studying the weak-strong uniqueness via the method introduced by Serrin \cite{serrin1963}. More precisely, we have
	\begin{theorem}\label{thm1.1}
		For any $\xi\ne0$, let $(Q,u)$ and $(R,v)$ be two Leray-Hopf type weak solutions to system \eqref{eq1.1} with the same initial data $(Q_0,u_0)\in H^1(\R^3, S_0^3)\times L_\sigma^2(\R^3)$. Assume that for some $2\leq p\leq 6$,
		\begin{equation}\label{eq:1.5}
			\Delta Q\in L^q(0,T;L^p(\R^3)) \quad\text{and}\quad \nabla u\in L^q(0,T;L^p(\R^3)), \quad \dfrac{2}{q}+\dfrac{3}{p}=\dfrac{3}{2}.
		\end{equation}
		Then, for all $t\in [0,T]$, we have
		\begin{equation}\label{eq:1.6}
			\begin{split}
				\left(u(t),v(t)\right)=& \|u_0\|_{L^2}^2 -2\mu \int_{0}^{t} (\nabla u:\nabla v) ds + \int_{0}^{t} (u,u \cdot \nabla v) + (v,v \cdot \nabla u) ds\\
				& +L\int_0^t \left( \nabla R \odot\nabla R - R\Delta R+\Delta RR :\nabla u\right) + \left( \nabla Q \odot\nabla Q - Q\Delta Q+\Delta QQ :\nabla v\right) ds \\
				&+\xi\int_0^t\left( H[R]\left[R +\dfrac{\mathbb{I}_3}{3}\right]+  \left[R +\dfrac{\mathbb{I}_3}{3}\right]H[R]-2  R{\rm Tr}(RH[R]): \nabla u \right)ds \\
				& +\xi\int_0^t\left( H[Q]\left[Q +\dfrac{\mathbb{I}_3}{3}\right]+  \left[Q +\dfrac{\mathbb{I}_3}{3}\right]H[Q]  - 2 Q{\rm Tr}\left(QH[Q]\right): \nabla v \right)ds \\
			\end{split}
		\end{equation}
		and
		\begin{equation}\label{eq:1.7}
			\begin{split}
				(Q(t):R(t)) + L &(\nabla Q(t):\nabla R(t))\\
				=& \|Q_0\|_{L^2}^2 + L \|\nabla Q_0\|_{L^2}^2 - 2a\Gamma  \int_{0}^{t} (Q:R )ds  - 2(a+1)\Gamma L\int_{0}^{t}  (\nabla R:\nabla Q)ds\\
				& - 2\Gamma L^{2} \int_0^t(\Delta R:\Delta Q)
				ds +\dfrac{2\xi}{3}\int_0^t ( D:R - L \Delta R) + ( \widebar{D}:Q - L \Delta Q)ds\\
				&+ \int_{0}^{t} (Q:u \cdot \nabla R)+ (R:v \cdot \nabla Q) +L (v \cdot \nabla R:\Delta Q) +L (u \cdot \nabla Q:\Delta R) ds \\
				&+ \int_0^t \left(\xi ( DQ+QD ) + ( \Omega Q-Q\Omega)  -2\xi  Q {\rm Tr}(Q\nabla u):R - L \Delta R\right) ds\\
				&+ \int_0^t \left(\xi ( \widebar{D}R+R\widebar{D} ) + ( \widebar{\Omega} R-R\widebar{\Omega})  -2\xi   R {\rm Tr}(R\nabla v):Q - L \Delta Q \right) ds\\
				&  +\Gamma\int_0^t \left( b\left[Q^2-\dfrac{{\rm Tr}(Q^2)}{3}\mathbb{I}_3\right]-cQ{\rm Tr}(Q^2) :R  - L \Delta R  \right)ds\\
				& + \Gamma\int_0^t \left( b\left[R^2-\dfrac{{\rm Tr}(R^2)}{3}\mathbb{I}_3\right]-cR{\rm Tr}(R^2) :Q - L \Delta Q \right)ds.
			\end{split}
		\end{equation}
		Moreover, the following energy equality holds for all $t\in [0,T]:$
		\begin{equation}\label{energy_eq}
			\begin{split}
				\|u\|_{L^2}^2& + \|Q\|_{L^2}^2+ L \|\nabla Q\|_{L^2}^2 + 2\mu \int_{0}^{t} \|\nabla u\|_{L^2}^2 ds\\
				& +2a\Gamma  \int_{0}^{t} \|Q\|_{L^2}^2ds +2(a+1)\Gamma L\int_{0}^{t} \|\nabla Q\|_{L^2}^2ds + 2\Gamma L^{2} \int_0^t \|\Delta Q\|_{L^2}^2
				ds \\
				= & \|u_0\|_{L^2}^2 +\|Q_0\|_{L^2}^2 + L \|\nabla Q_0\|_{L^2}^2 + 4(1-a)\xi \int_0^t \left( Q\left[Q + \dfrac{\mathbb{I}_3}{3}\right]:\nabla u\right)-\left(  Q {\rm Tr}(Q\nabla u):Q\right)ds \\
				&  +2\Gamma\int_0^t \left( b\left[Q^2-\dfrac{{\rm Tr}(Q^2)}{3}\mathbb{I}_3\right]-cQ{\rm Tr}(Q^2) :Q - L \Delta Q \right)ds\\
				&  +4\xi\int_0^t \left( \left[Q + \dfrac{\mathbb{I}_3}{3}\right]\left\lbrace b\left[Q^2-\dfrac{{\rm Tr}(Q^2)}{3}\mathbb{I}_3\right]-cQ{\rm Tr}(Q^2)\right\rbrace :\nabla u \right)ds\\
				& - 4\xi\int_0^t \left( b\left[Q^2-\dfrac{{\rm Tr}(Q^2)}{3}\mathbb{I}_3\right]-cQ{\rm Tr}(Q^2) :Q {\rm Tr}(Q\nabla u) \right)ds.
			\end{split}
		\end{equation}
	\end{theorem}
	\begin{remark}
		We list a few remarks here concerning Theorem~\ref{thm1.1}.
		\begin{enumerate}[(1).]
			\item This result implies that if a Leray-Hopf type weak solution of \eqref{eq1.1} satisfies the condition \eqref{eq:1.5}, then energy equality holds. In particular, it holds for all weak solutions constructed in \cite{Qtensor11} under the condition \eqref{eq:1.5}; see the proof of Theorem~\ref{thm1.1} in Section~\ref{Sec.3} below.
			\item In contrast to the regularity condition of energy equality for the case $\xi=0$,
			we impose a slightly stronger condition on $(\Delta Q,\nabla u)$, which is basically used to overcome the analytical difficulties generated by the non-zero parameter $\xi$.
			\item
			In view of recent studies on energy equality for the incompressible Navier-Stokes equations (see \cite{shinbrot1974}), one may expect to get a weaker condition such that the above energy equality holds. Indeed, we would like to point out that the condition \eqref{eq:1.5} is taken to ensure uniqueness properties of weak solutions for system \eqref{eq1.1}, which is not optimal for energy equality in general.
		\end{enumerate}
	\end{remark}
	In order to show weak-strong uniqueness of \eqref{eq1.1},
	we introduce the following notion of strong solution, which is stronger in the sense that a weak solution has some additional regularity.
	\begin{definition}[Strong solution]\label{def1.2}
		For any $\xi\ne 0$ and $T>0$, let $(Q,u)$ be a weak solution to system \eqref{eq1.1} with initial data $(Q_0,u_0)\in H^1(\R^3, S_0^3)\times L_\sigma^2(\R^3)$ such that condition $(i)$ and $(ii)$ hold. If $(Q,u)$ further belongs to the class
		\[
		\Delta Q\in L^q(0,T;L^p(\R^3)) \quad\text{and}\quad \nabla u\in L^q(0,T;L^p(\R^3)),\quad \dfrac{2}{q}+\dfrac{3}{p}\leq \dfrac{3}{2},\,\, p\geq 2,\,\, q\geq \frac{4}{3},
		\]
		then we call it a strong solution of system \eqref{eq1.1}.
	\end{definition}
	Now, we are in a position to state our second main result regarding the  weak-strong uniqueness for the Leray-Hopf type weak solutions of system \eqref{eq1.1} for arbirtray $\xi\ne 0$.
	\begin{theorem}\label{thm1.2}
		For any $\xi\ne0$, let $(Q,u)$ and $(R,v)$ be two Leray-Hopf type weak solutions to system \eqref{eq1.1} with the same initial data $(Q_0,u_0)\in H^1(\R^3, S_0^3)\times L_\sigma^2(\R^3)$. 
		For some $2\leq p,q\leq +\infty$, if $(Q,u)$ satisfies
		\begin{equation}\label{w-k-xi}
		\Delta Q\in L^q(0,T;L^p(\R^3)) \quad\text{and}\quad \nabla u\in L^q(0,T;L^p(\R^3)),\quad \dfrac{2}{q}+\dfrac{3}{p}\leq \dfrac{3}{2},
		\end{equation}
		then $(Q,u)\equiv (R,v)$ on $[0,T]$.
	\end{theorem}
\begin{remark}
	Note that, once $\xi\ne 0$, the uniqueness criteria imposed on $(\Delta Q,\nabla u)$ in Theorem~\ref{thm1.2}
	differs from that of the standard Navier-Stokes equations \cite{nabla u} and the corotational case \cite{weak-Q-tensor}.
	Technically, this is determined by the highest order nonlinear terms in \eqref{eq1.1}.
	However, it is unknown whether weak-strong uniqueness holds in any regime other than that illustrated in Figure \ref{fig1}. ‌In addition, one may expect the original Serrin-type condition for $(\nabla Q,u)$ instead of $(\Delta Q, \nabla u)$.
	For these problems, we shall pursue them in a forthcoming paper.
	
\end{remark}

     \begin{figure}[h]
    	\centering
    	\begin{tikzpicture}[scale=2]
    		\draw[->] (0,0) -- (4.8,0) node[above] {\footnotesize time scaling $1/q$}; 
    		\draw[->] (0,0) -- (0,2.5) node[above] {\footnotesize space scaling $1/p$};
    		\draw[->, blue] (1.68,1.54) to (1.5,1.34) node[right] {};
    		\node at (-0.05,-0.25) {\footnotesize$L_{t,x}^\infty$};
    		\node at (2.01,1.56) {~~~~~\textcolor{blue}{\footnotesize$\frac{2}{q}+\frac{3}{p}=2$}};
    		\draw[->, red] (2.54,0.48) to (2.35,0.35) node[right] {};
    		\node at (2.75,0.48) {~~~~~~~~~~\textcolor{red}{\footnotesize$\frac{2}{q}+\frac{3}{p}=\frac{3}{2}$}};
    		\tikzset{knop/.style={circle,minimum size=2,inner sep=0pt,draw,fill=white}}
    		\tikzset{knopp/.style={circle,minimum size=2,inner sep=0pt,draw,fill=red}}
    		\tikzset{knopb/.style={circle,minimum size=2,inner sep=0pt,draw,fill=bblue}}
    		\tikzset{knopv/.style={circle,minimum size=2,inner sep=0pt,draw,fill=violet}}
    		\node[knop] (p1) at (0,1) {} ;
    		\node[knop] (p2) at (2,0) {} ;
    		\node[knop] (p3) at (0,1.5) {} ;
    		\node[knopp] (p31) at (0,1.5) {} ;
    		\node[knop] (p4) at (0,2) {} ;
    		\node[knop] (p5) at (4,0) {} ;
    		\node[knop] (p6) at (1,1.5) {} ;
    		\node[knopb] (p61) at (1,1.5) {} ;
    		\node[knop] (p7) at (2,1) {} ;
    		\node[knopb] (p71) at (2,1) {} ;
    		\node[knopp] (pk2) at (2,0.5) {} ;
    		\node[knopp] (pkk2) at (2,0) {} ;
    		\node[knopp] (pkk1) at (0,0) {} ;
    		\node[knop] (p9) at (1,0) {} ;
    		\node[knop] (p0) at (3,0) {} ;
    		\draw[gray, dashed] (p61.center)--(p3.center);
    		\draw[gray, dashed] (p9.center)--(p6.center);
    		\draw[gray, dashed] (p1.center)--(p7.center);
    		\draw[gray, dashed] (p2.center)--(p7.center);
    		\begin{scope}[on background layer]
    			\path [fill=bblue,draw=bblue,thin] (pkk1.center) to (pkk2.center) to (p71.center) to (p61.center) to (p3.center) to (pkk1.center);
    			\path [fill=blue,draw=blue,dashed] (p5.center) to  (p4.center);
    			\path [fill=red,draw=red,dashed] (p0.center) to  (p3.center);
    			\path [fill=blue,draw=blue,thin] (p71.center) to  (p61.center);
    			\path [draw=red,thin,pattern={Lines[angle=25,distance={5pt}]},pattern color=red] (pkk1.center) to (pkk2.center) to (pk2.center) to (p31.center) to (pkk1.center);
    		\end{scope}
    		\path [fill=bblue,draw=bblue,thin] (3.4,2.2) rectangle (3.7,2.3);
    		\path [draw=red,thin,pattern={Lines[angle=25,distance={2pt}]},pattern color=red] (3.4,2) rectangle (3.7,2.1);
    		\draw[red,dashed] (0,0)--(0,1.5);
    		\draw[red,dashed] (0,0)--(2,0);
    		\node[knopp] (p31) at (0,1.5) {} ;
    		\node at (-0.4,1) {\footnotesize$L_t^\infty L_x^3$};
    		\node at (-0.4,1.5) {\footnotesize$L_t^\infty L_x^2$};
    		\node at (1.16,1.7) {\footnotesize$L_t^4L_x^{2}$};
    		\node at (1.25,-0.25) {\footnotesize$L_t^4L_x^{\infty}$};
    		\node at (2.3,1.1) {\footnotesize$L_t^2L_x^{3}$};
    		\node at (2.1,-0.25) {\footnotesize$L_t^2L_x^{\infty}$};
    		\node at (2.3,0.6) {\footnotesize$L_t^2L_x^{6}$};
    		\node at (4.21,2.25) {\scriptsize$\xi=0$~(\cite{weak-Q-tensor})};
    		\node at (4.03,2.05) {\scriptsize$\xi\ne 0$};
    	\end{tikzpicture}
    	\caption{Different subclasses of the space $L^q(0,T;L^p(\R^3))$.
    		According to \eqref{eq1.8} and the interpolation argument as in the proof of Theorem~\ref{thm1.2},
    		the criterion for weak-strong uniqueness of system \eqref{eq1.1} with $\xi=0$ is given by the blue region.
    		For the case of $\xi\ne 0$, the condition obtained in Theorem~\ref{thm1.2} is defined by \eqref{w-k-xi}, as indicated by the red shaded region. 
    	}
    	\label{fig1}
    \end{figure}

	\begin{remark}
		As shown in Figure \ref{fig1}, uniqueness criteria for the case $\xi\ne0$ is much stronger than that of $\xi=0$, and is a sub-regime of the corotational case.
		In particular, in \cite{weak-Q-tensor,weak-strong-3d}, the condition \eqref{eq1.8} is used to achieve the estimates for some trilinear terms. For example, one may expect that
		\[
		\left| \int_0^T \left(\Phi\Delta\Phi:\nabla\psi \right) ds \right|\lesssim
		\|\Phi\|_{L_t^\infty(H^1)}^{\frac{2p-3}{p}}\|\Delta \Phi\|_{L_t^{\frac{2p}{2p-3}}(L^p)}\|\Phi\|_{L_t^2(H^2)}^{\frac{3-p}{p}}\|\nabla\psi\|_{L_t^2(L^2)}.
		\]
		However, when $\xi\ne 0$,
		we cannot handle
		the new quartic terms arising from $Q{\rm Tr}(Q\Delta Q)$ and $Q{\rm Tr}(Q\nabla u)$ under the same regularity assumptions. Indeed, one has an estimate of the following type:
		\[
		\left| \int_0^T \left(\Phi^2\Delta\Phi:\nabla\psi \right) ds \right|\lesssim
		\|\Phi\|_{L_t^\infty(H^1)}^{\frac{5p-6}{2p}}\|\Delta \Phi\|_{L_t^{\frac{4p}{3p-6}}(L^p)}\|\Phi\|_{L_t^2(H^2)}^{\frac{6-p}{2p}}\|\nabla\psi\|_{L_t^2(L^2)}.
		\]
		This fact motivated us to impose the new condition \eqref{eq:1.5}, and it is the most important novelty of this work.
	\end{remark}
The second part of this paper focuses on 
the existence of solutions with higher regularity to system \eqref{eq1.1} for arbirtray $\xi\in\R$. Because it is natural to ask whether one might get a similar result to that of \cite{Qtensor11} in $\R^2$.
The main difficulty in three dimensional case is due to the fact that
we cannot apply the sharp logarithmic Sobolev embedding of $H^{1+\epsilon}$ in $L^\infty$ to obtain the estimates on $\|Q(t,\cdot)\|_{W^{1,\infty}}$ and $\|u(t,\cdot)\|_{L^{\infty}}$, which allow us to close the higher regularity estimates of system \eqref{eq1.1}. Later on,
inspired by the work by Danchin \cite{Danchin06} for the density-dependent incompressible Navier-Stokes equations,
the authors in \cite{xi-all,Q-unique-3d}
utilize the maximal regularities of Stokes and parabolic operators to prove the existence of strong solutions for system \eqref{eq1.1}. We also remark that, this method has been successfully used to study the Ericksen-Leslie models of liquid crystals; see \cite{EL-cmp,EL-strong-jde,IEL-strong} and references therein.

Motivated by \cite{active-limit,ALC-decay}, we design an approximation system and apply the refined commutator estimates to prove the existence and uniqueness of local solution $(Q,u)$ to system \eqref{eq1.1} in $H^{s+1}\times H^s$ with $s\geq 2$ for the case $\xi\in\R$.
Note that the initial data considered here is regular enough so that one can solve the aforementioned difficulties on estimates by the standard embeddings.
%
As for the global regularity, we need to further assume that $a>0$. In this case, the term $-aQ$ in the equation of $Q$ can be regarded as a damping term, which enables us to achieve an energy inequality of the form
\[
\dfrac{d}{dt}\mathcal{E}(t)+ \mathcal{D}(t)
\leq C \sum_{m=1}^4\mathcal{E}^{\frac{m}{2}}(t) \mathcal{D}(t).
\]
where $\mathcal{E}(t)$ and $\mathcal{D}(t)$ are the kinetic energy and dissipation functions, respectively.
Then, by the standard continuation argument, we get our third main result regarding
the existence and uniqueness of global-in-time solution to system \eqref{eq1.1} with small initial data for arbirtray $\xi\in\R$.
\begin{theorem}\label{thm3.1}
	Let $\xi\in\R$, $a>0$ and $(Q_0,u_0)\in H^{s+1}\times H^s$ for any integer $s\geq 2$. If there exists a small real number $\delta_0$ such that $\|Q_0\|_{H^{s+1}}^2 + \|u_0\|_{H^s}^2\leq \delta_0 $, then the Cauchy problem of system \eqref{eq1.1} admits a unique global solution $(Q,u)$ satisfying
	\begin{equation}
		Q\in L^\infty\left(  [0,+\infty);H^{s+1} \right)\cap L^2\left(  [0,+\infty);H^{s+2} \right)~\text{and}~u\in L^\infty\left( [0,+\infty);H^s \right)\cap L^2\left(  [0,+\infty);H^{s+1} \right).
	\end{equation}
	Moreover, $(Q,u)$ satisfies the following energy inequality
	\begin{equation}
		\|u(t)\|_{H^s}^2 + \|Q(t)\|_{H^{s+1}}^2 
		+\int_0^t \|\nabla u(\tau)\|_{H^{s}}^2 + \|  Q(\tau)\|_{H^{s+1}}^2 + \|\Delta
		Q(\tau)\|_{H^{s}}^2 d\tau\leq C,
	\end{equation}
	for all $t\geq 0$,
	where $C$ is a positive constant depending on $\delta_0$, $s$, $a$, $b$, $c$, $\mu$, $|\xi|$, $L$ and $\Gamma$.
\end{theorem}
\begin{remark}
	A few remarks are in order:
	\begin{enumerate}[(i).]
		\item
		In the work \cite{ALC-decay}, the damping term $-aQ$ plays an important role in performing energy estimates. As a matter of fact‌, we use the dissipation effect of $Q$ to overcome the difficulty caused by the linear active term.
		But in this paper, non-zero parameter $\xi$ makes the study of system \eqref{eq1.1} more complicated. To overcome this complication, we have to refine the argument from \cite{ALC-decay}; see Section~\ref{sec5} below for the details. Moreover, the global regularity remains unclear for the case of $a\leq 0$.
		\item 
		Similar to \cite{ALC-decay}, one can see that the assumption on Sobolev index $s$ in Theorem~\ref{thm3.1} is weaker than that of \cite{active-limit} where integer $s> \frac{d}{2}+1$ with $d=2,3$.
		This improvement is due to the refined commutator estimates proposed in our recent work \cite{ALC-decay}, which will help to achieve the desired a priori estimate of the solutions in $H^s$-framework; see the proof of Proposition~\ref{prop-a1} in Appendix~\ref{ap1} below.
		\item 
		When comparing our result to that of \cite{Qtensor11}, it must be pointed out that
		the smallness assumption on the physical parameter $\xi$ has been removed by having good initial data, i.e., $(Q_0,u_0)\in H^{s+1}\times H^s$ with integer $s\geq 2$. Furthermore, we would like to mention the recent work by Hieber, Hussein and Wrona \cite{xi-all}, where the existence theory of this model in a bounded domain $\mathcal{O}\subset \R^3$ with smooth boundary was also established without any restriction on $\xi$. The associated initial data can be regarded as the interpolation space as follows:
		\begin{align*}
			Q_0\in B_{2,p}^{3-\frac{2}{p}}=\left( H^1,H^3 \right)_{1-\frac{1}{p},p}\,\,{and}\,\,
			u_0\in B_{2,p}^{2-\frac{2}{p}}\cap L_\sigma^2\,\,{with}\,\,B_{2,p}^{2-\frac{2}{p}}=\left( L^2,H^2 \right)_{1-\frac{1}{p},p},
		\end{align*}
		where $p>4$. Note that $(Q_0,u_0)\hookrightarrow B_{2,p}^{3-\frac{2}{p}}\times B_{2,p}^{2-\frac{2}{p}} \hookrightarrow C^1\times C^0$, the common point with the initial data herein is that both belong to the class $W^{1,\infty}\times L^\infty$.
		\item Note that one has the following interpolation inequality
		\[
		\int_0^T\|\nabla u\|_{L^p}^q dt\lesssim \|\nabla u\|_{L_t^\infty (L^2)}^{q-2} \int_0^T\|\nabla u\|_{H^1}^2 dt\quad\text{\rm with}\quad \dfrac{2}{q}+\dfrac{3}{p}=\dfrac{3}{2}
		\]
		for $\nabla u \in L^\infty(0,T;L^2(\R^3))  \cap L^2(0,T;H^1(\R^3))$.
		Thus our result also implies the existence of strong solutions in the sense of Definition~\ref{def1.2}, which obviously satisfies the energy equality.
	\end{enumerate}
\end{remark}

The rest of this paper is organized as follows. In Section~\ref{Sec.2}, we collect the notations and derive some useful lemmas.
In Section~\ref{Sec.3}, we establish the existence of energy equality for the Leray-Hopf type weak solutions to system \eqref{eq1.1} under the new regularity assumption \eqref{eq:1.5}. Thereafter, Section~\ref{sec4} is devoted to proving Theorem~\ref{thm1.2}.
In Section~\ref{sec5}, we present the a priori estimates for \eqref{eq1.1}, and Theorem~\ref{thm3.1} follows by the standard continuity method.
For the sake of completeness, the local well-posedness of system \eqref{eq1.1} is included in Appendix~\ref{ap1}.
Finally, we recall some technical inequalities and provide some important preliminary estimates in Appendix~\ref{ap2}.

\section{Preliminaries}\label{Sec.2}
In this section, we first introduce some notations used throughout the paper.
For the sake of brevity, we will write $X\lesssim Y$ to indicate $X\leq CY$ with some generic positive constant $C$. Similarly, the notation $X\sim Y$ means that there exists a positive constant $C$ such that
$C^{-1}Y\leq X\leq CY$. For $1\leq p\leq \infty$, $L^p(\R^3)$ denotes the usual Lebesgue space of all real-valued $L^p$ integrable functions, and endowed with norm $\|\cdot\|_{L^p}$.
Meanwhile, $\|\cdot\|_{H^s}$ stands for the norm on the customary Sobolev spaces $H^s(\R^3)$.
As usual, we do not distinguish between scalar, vector and tensor valued functions.
Then, the space $L_\sigma^p(\R^3)$ is defined by
\[
L_\sigma^p(\R^3):=\left\lbrace u\in L^p(\R^3): {\rm div}~u=0\right\rbrace
\]
and $S_0^3\subset \mathbb{M}^{3\times 3}$ denotes the space of $Q$-tensor in $3$-dimension, i.e.,
\[
S_0^3:=\left\lbrace  Q\in\mathbb{M}^{3\times 3}: Q_{\alpha\beta}=Q_{\beta\alpha},\quad {\rm Tr}(Q)=0,\quad \alpha,\beta=1,2,3  \right\rbrace,
\]
where $\mathbb{M}^{3\times 3}$ refers to the space of all real, $3\times 3$ matrix-valued functions.
For $u\in L^q(0,T;L^p(\R^3))$, we write
\begin{numcases}{\|u\|_{L^q(0,T;L^p(\R^3))}:=}
	\left( \int_0^t\|u(t)\|_{L^p}^q dt \right)^{\frac{1}{q}},&\text{if $1\leq q<\infty$,}\nonumber\\
	{\rm ess}\sup_{t\in[0,T]}\|u(t)\|_{L^p},&\text{if $ q=\infty$.}\nonumber
\end{numcases}
Moreover, the norm of $L^q(0,T;L^p(\R^3))$ is simplified by $\|\cdot\|_{L_t^q(L^p)}$.

For any multi-index $\alpha=\left(\alpha_1,\alpha_2,\cdots,\alpha_d \right)\in \mathbb{N}^d$ with $|\alpha|=\sum_{j=1}^d \alpha_j$, the partial derivative of order $\alpha$ is defined by
\[
\partial^{\alpha}f(x) = \dfrac{\partial^{|\alpha|} f(x) }{\partial_{x_1}^{\alpha_1}\partial_{x_2}^{\alpha_2}\cdots \partial_{x_d}^{\alpha_d}}.
\]
In general, $(\cdot,\cdot)$ denotes the inner product for real vector-valued functions in $L^2$, while for two real matrix-valued functions $A, B\in L^2(\R^3;S_0^3)$, their inner product is given by
\[
(A:B):=\int_{\R^3}{\rm tr}(AB)dx.
\]
Moreover, we use the Frobenius norm of the matrix-valued function $Q\in \mathcal{S}_0^3$, i.e.,
\[
|Q|:=\sqrt{{\rm tr}(Q^2)}=\sqrt{Q_{\alpha\beta}Q_{\alpha\beta}}.
\]
On the other hand, we also introduce the following commutators of $A$ and $B$ for any functions $f, g$:
\[
\begin{split}
	[A, B]f &= ABf -BAf,\\
	[A, B]_{\rm div}f &= AB\cdot\nabla f -B\cdot\nabla (A f),\\
	[A, (f,g)]_{\pm} &= A(fg)-(Af)g \pm [A,g]f.
\end{split}
\]

Let $\eta\in C_c^\infty(\R)$ be a real-valued, non-negative even function such that
\[
\int_{\R}\eta (\tau)d\tau= 1\quad\text{and}\quad \supp f\subset [-1,1].
\]
Then, for any space–time function $f$, we define its (time) mollification function $f_\varepsilon$
by
\[
f_\varepsilon(s,x)= \int_0^t \eta_\varepsilon(s-\tau)f(\tau,x)d\tau,
\]
where $\eta_\varepsilon(s)=\varepsilon^{-1}\eta(s/\varepsilon)$ and $0<\varepsilon<t$. Note that, as in \cite{Galdi-book,serrin1963}, we have the following relevant properties:
\begin{enumerate}[(i).]
	\item If $f\in L^q(0,T;L^p)$ with $1\leq q<\infty$, then $f_\varepsilon \in C^\infty(0,T;L^p)$ and $f_\varepsilon\rightarrow f$ in $L_t^q(L^p)$ as $\varepsilon\rightarrow 0$.
	\item If $\{ f^i \}_{i=1}^{\infty}$ converges to $f$ in $L_t^q(L^p)$, then $(f^i)_\varepsilon \rightarrow f_\varepsilon$ in $L_t^q(L^p)$ as $i\rightarrow +\infty$.
\end{enumerate}

Next, we state the following property of Leray-Hopf type weak solutions to system \eqref{eq1.1}, which is well-known in the existing literature for Navier-Stokes equations. For the proof, one can refer to \cite{Galdi-book,serrin1963,weak-strong-3d}.

\begin{prop}\label{prop2.1}
	Let $\xi\in\R$ and $(Q,u)$ be a Leray-Hopf type weak solution to system \eqref{eq1.1} with initial data $(Q_0,u_0)\in H^1\times L_\sigma^2$. Then $(Q,u)$ can be redefined on a set of zero Lebesgue measure in such a way that $(Q(t),u(t))\in H^1\times L^2$ for all $t\geq 0$ and satisfies the identities:
	\begin{equation}\label{eq-a1}
		\begin{split}
			\int_0^t ( Q:&\partial_t\phi )+\Gamma L (\Delta Q:\phi) +(Q:u\cdot\nabla \phi) -2\xi \left( \left[Q +\dfrac{\mathbb{I}_3}{3}\right] {\rm Tr}(Q\nabla u):\phi \right) ds\\
			&+ \int_0^t\left( (\xi  D+\Omega)\left[Q +\dfrac{\mathbb{I}_3}{3}\right] :\phi \right) + \left( \left[Q +\dfrac{\mathbb{I}_3}{3}\right] (\xi D-\Omega):\phi \right)ds\\
			=& \Gamma\int_0^t \left( aQ-b\left[Q^2-\dfrac{{\rm Tr}(Q^2)}{3}\mathbb{I}_3\right]+cQ{\rm Tr}(Q^2) :\phi \right)ds + (Q(t):\phi(t)) -(Q_0(x):\phi(0,x))
		\end{split}
	\end{equation}
	and
	\begin{equation}\label{eq-a2}
		\begin{split}
			&\int_0^t \left( u,\partial_t\psi \right)+(u,u\cdot\nabla \psi)-\mu (\nabla u,\nabla\psi)ds +L\int_0^t \left( \nabla Q \odot\nabla Q - Q\Delta Q+\Delta QQ :\nabla\psi \right)ds \\
			=& -\xi\int_0^t\left( H\left[Q +\dfrac{\mathbb{I}_3}{3}\right]+  \left[Q +\dfrac{\mathbb{I}_3}{3}\right]H: \nabla\psi \right)ds  + 2\xi\int_0^t\left(  \left[Q +\dfrac{\mathbb{I}_3}{3}\right]{\rm Tr}(QH): \nabla\psi \right)ds \\
			& +(u(t),\psi(t)) -(u_0(x),\psi(0,x)).
		\end{split}
	\end{equation}
	for all $t\geq 0$, and all $\phi\in C_c^\infty([0,\infty)\times \R^3;\mathcal{S}_0^3)$, $\psi\in C_c^\infty([0,\infty)\times \R^3;\R^3)$ with $\nabla\cdot \psi=0$.
\end{prop}
In what follows, we always assume that the Leray-Hopf type weak solution of system \eqref{eq1.1} satisfies \eqref{eq-a1} and \eqref{eq-a2} for all $t\geq 0$ (in the sense of Proposition~\ref{prop2.1}). Moreover, one can also observe that, it follows from the above result that $u$, $Q$, $\nabla Q$ are continuous in the weak topology of $L^2$ (see \cite[Corollary~$3.2$]{shinbrot1974}).

Finally, we close this section by presenting the following three Lemmas that used to deal with some nonlinear terms in \eqref{eq1.1}, which are crucial to our proof of Theorem~\ref{thm1.2}.
\begin{lemma}\label{lem2.2}
	Let $\phi\in L^\infty(0,T;L^2(\R^3))\cap L^2(0,T;H^1(\R^3))$, $\varphi\in L^\infty(0,T;H^1(\R^3))\cap L^2(0,T;H^2(\R^3))$, $\nabla\psi\in L^2(0,T;L^2(\R^3))$. If $\phi$ and $\varphi$ also satisfies
	\[
	(\nabla \phi, \Delta \varphi)\in L^q(0,T;L^p(\R^3))\quad\text{with}\quad \dfrac{2}{q}+\dfrac{3}{p}=\dfrac{3}{2},\,\, 2\leq p\leq 6.
	\]
	Then, we have the estimates
	\begin{equation}
		\left| \int_{0}^{T}(\phi,\phi \cdot \nabla \psi) dt \right|\lesssim\|\phi\|_{L_t^\infty(L^2)}^{\frac{5p-6}{4p}} \|\phi\|_{L_t^2(H^1)}^{\frac{p+6}{4p}} \|\nabla \phi\|_{L_t^q(L^p)}^{\frac{1}{2}} \|\nabla \psi\|_{L_t^2(L^2)}
	\end{equation}
	and
	\begin{equation}
			\left| \int_0^T ( \varphi\Delta \varphi-\Delta \varphi\varphi :\nabla \psi )dt \right| \lesssim  \|\varphi\|_{L_t^\infty(H^1)}^{\frac{3p-6}{2p}} \|\Delta \varphi\|_{L_t^q(L^p)} \|\varphi\|_{L_t^2(H^2)}^{\frac{6-p}{2p}} \|\nabla \psi\|_{L_t^2(L^2)}.
	\end{equation}
\end{lemma}
\begin{proof}
	For any $2\leq p\leq 6$, it is easy to check that
	\begin{align}
		\|\phi\|_{L^\frac{12p}{6+p}}\lesssim& \|\nabla \phi\|_{L^p}^{\frac{1}{2}}\|\phi\|_{L^2}^{\frac{1}{2}},\\
		\|\phi\|_{L^\frac{12p}{5p-6}}\lesssim& \| \phi\|_{L^2}^{\frac{3p-6}{4p}}\|\phi\|_{L^6}^{\frac{p+6}{4p}},
	\end{align}
	where we used the Gagliardo-Nirenberg inequality and $L^p$ interpolation.
	
	These together with H\"{o}lder inequality yield
	\[
		\begin{split}
			\left| \int_{0}^{T}(\phi,\phi \cdot \nabla \psi) dt \right|\leq & \int_0^T\|\phi\|_{L^\frac{12p}{5p-6}}\|\phi\|_{L^\frac{12p}{6+p}}\|\nabla \psi\|_{L^2} dt\\
			\lesssim & \int_0^T\|\nabla \phi\|_{L^p}^{\frac{1}{2}}\|\phi\|_{L^2}^{\frac{1}{2}}\| \phi\|_{L^2}^{\frac{3p-6}{4p}}\|\phi\|_{L^6}^{\frac{p+6}{4p}} \|\nabla \psi\|_{L^2} dt\\
			\lesssim & \left(\int_0^T \left[\|\nabla \phi\|_{L^p} \|\phi\|_{L^2}^{\frac{5p-6}{2p}}\right] (\|\nabla \phi\|_{L^2}^2)^{\frac{p+6}{4p}} dt \right)^{\frac{1}{2}}\|\nabla \psi\|_{L_t^2(L^2)}\\
			\lesssim &  \left(\int_0^T \|\nabla \phi\|_{L^p}^{\frac{4p}{3p-6}} \|\phi\|_{L^2}^{\frac{10p-12}{3p-6}} dt\right)^{\frac{3p-6}{8p}} \cdot \left(\int_0^T\|  \phi\|_{H^1}^2 dt \right)^{\frac{p+6}{8p}} \|\nabla \psi\|_{L_t^2(L^2)} \\
			\lesssim & \|\phi\|_{L_t^\infty(L^2)}^{\frac{5p-6}{4p}} \|\phi\|_{L_t^2(H^1)}^{\frac{p+6}{4p}} \left(\int_0^T \|\nabla \phi\|_{L^p}^{\frac{4p}{3p-6}} dt\right)^{\frac{3p-6}{8p}} \|\nabla \psi\|_{L_t^2(L^2)}\\
			\lesssim & \|\phi\|_{L_t^\infty(L^2)}^{\frac{5p-6}{4p}} \|\phi\|_{L_t^2(H^1)}^{\frac{p+6}{4p}} \|\nabla \phi\|_{L_t^q(L^p)}^{\frac{1}{2}} \|\nabla \psi\|_{L_t^2(L^2)}.
		\end{split}
	\]
	Similarly, by the use of $L^p$ interpolation, one can see that
	\begin{align}
		\|\varphi\|_{L^{\frac{2p}{p-2}}}\lesssim& \|\varphi\|_{L^3}^{\frac{3p-6}{6+p}} \|\varphi\|_{L^{\frac{4p}{p-2}}}^{\frac{12-2p}{6+p}},\\
		\|\varphi\|_{W^{1,\frac{12p}{7p-6}}}\lesssim& \| \varphi\|_{H^1}^{\frac{5p-6}{4p}}\|\varphi\|_{W^{1,6}}^{\frac{6-p}{4p}}
	\end{align}
	hold for all $2\leq p\leq 6$.
	Then, we have
	\begin{align*}
		\left| \int_0^T ( \varphi\Delta \varphi-\Delta \varphi\varphi :\nabla \psi )dt \right| \leq & \int_0^T \|\Delta \varphi\|_{L^p}\| \varphi\|_{L^{\frac{2p}{p-2}}}\|\nabla \psi\|_{L^2}  dt\\
		\lesssim & \int_0^T \|\Delta \varphi\|_{L^p} \|\varphi\|_{L^3}^{\frac{3p-6}{6+p}} \|\varphi\|_{L^{\frac{4p}{p-2}}}^{\frac{12-2p}{6+p}} \|\nabla \psi\|_{L^2}  dt\\
		\lesssim & \int_0^T \|\Delta \varphi\|_{L^p}\|\varphi\|_{H^1}^{\frac{3p-6}{6+p}}\| \varphi\|_{W^{1,\frac{12p}{7p-6}}}^{\frac{12-2p}{6+p}}\|\nabla \psi\|_{L^2}  dt\\
		\lesssim & \int_0^T \|\Delta \varphi\|_{L^p}\|\varphi\|_{H^1}^{\frac{3p-6}{6+p}}\| \| \varphi\|_{H^1}^{\frac{(5p-6)(6-p)}{2p(6+p)}}\|\varphi\|_{W^{1,6}}^{\frac{(6-p)^2}{2p(6+p)}}\|\nabla \psi\|_{L^2}  dt\\
		\lesssim & \int_0^T \|\Delta \varphi\|_{L^p}\| \varphi\|_{H^1}^{\frac{3p-6}{2p}} \| \varphi\|_{H^2}^{\frac{6-p}{2p}}\|\nabla \psi\|_{L^2}  dt\\
		\lesssim & \left(\int_0^T \|\Delta \varphi\|_{L^p}^2(\| \varphi\|_{H^1}^2)^{\frac{3p-6}{2p}} (\| \varphi\|_{H^2}^2)^{\frac{6-p}{2p}} dt \right)^{\frac{1}{2}}\|\nabla \psi\|_{L_t^2(L^2)} \\
		\lesssim & \left(\int_0^T \|\Delta \varphi\|_{L^p}^{\frac{4p}{3p-6}}\|\varphi\|_{H^1}^2 dt\right)^{\frac{3p-6}{4p}} \left( \int_0^T\|\varphi\|_{H^2}^2 dt \right)^{\frac{6-p}{4p}}\|\nabla \psi\|_{L_t^2(L^2)} \\
		\lesssim &  \|\varphi\|_{L_t^\infty(H^1)}^{\frac{3p-6}{2p}} \|\Delta \varphi\|_{L_t^q(L^p)} \|\varphi\|_{L_t^2(H^2)}^{\frac{6-p}{2p}} \|\nabla \psi\|_{L_t^2(L^2)}.
	\end{align*}
	where we have used the fact that
	\[
	\frac{(6-p)^2}{2p(6+p)} + \frac{6-p}{6+p} =\frac{6-p}{2p},\quad \frac{3p-6}{6+p}+ \frac{(5p-6)(6-p)}{2p(6+p)} - \frac{6-p}{6+p}=\frac{3p-6}{2p},
	\]
	and the third inequality is due to the Sobolev embededing $W^{1,\frac{12p}{7p-6}}\hookrightarrow L^{\frac{4p}{p-2}}$.
	This completes the proof.
\end{proof}
\begin{lemma}\label{lem2.3}
	Let $2\leq p\leq 6$ and $\phi\in L^\infty(0,T;L^2(\R^3))\cap L^2(0,T;H^1(\R^3))$, $\varphi\in L^\infty(0,T;H^1(\R^3))\cap L^2(0,T;H^2(\R^3))$. If $\nabla \psi\in L^q(0,T;L^p(\R^3))$ with
	$\frac{2}{q}+\frac{3}{p}=\frac{3}{2}$,
	then
	\begin{equation}
		\left| \int_{0}^{T}(\phi,\phi \cdot \nabla \psi) dt \right|\lesssim\|\phi\|_{L_t^\infty(L^2)}^{\frac{3p-6}{2p}} \|\phi\|_{L_t^2(H^1)}^{\frac{p+6}{2p}} \|\nabla \psi\|_{L_t^q(L^p)}
	\end{equation}
	and
	\begin{equation}
		\left| \int_0^T ( \varphi\Delta \varphi-\Delta \varphi\varphi :\nabla \psi )dt \right| \lesssim  \|\varphi\|_{L_t^\infty(H^1)}^{\frac{3p-6}{2p}}  \|\varphi\|_{L_t^2(H^2)}^{\frac{6+p}{2p}} \|\nabla \psi\|_{L_t^q(L^p)}.
	\end{equation}
\end{lemma}
\begin{proof}
	Similarly to Lemma~\ref{lem2.2}, we use H\"{o}lder inequality, $L^p$ interpolation, Gagliardo-Nirenberg inequality and Sobolev embedding again to estimate
		\begin{align*}
			\left| \int_{0}^{T}(\phi,\phi \cdot \nabla \psi) dt \right|\leq & \int_0^T\|\phi\|_{L^\frac{12p}{7p-6}}\|\phi\|_{L^\frac{12p}{5p-6}}\|\nabla\psi\|_{L^p} dt\\
			\lesssim & \int_0^T \|\phi\|_{L^2}^{\frac{5p-6}{4p}}\|\phi\|_{L^6}^{\frac{6-p}{4p}} \|\phi\|_{L^2}^{\frac{3p-6}{4p}}\|\phi\|_{L^6}^{\frac{p+6}{4p}} \|\nabla \psi\|_{L^p} dt\\
			\lesssim & \int_0^T \|\phi\|_{L^2}^{\frac{3p-6}{2p}}\|\phi\|_{H^1}^{\frac{p+6}{2p}}\|\nabla \psi\|_{L^p} dt\\
			\lesssim & \|\phi\|_{L_t^\infty(L^2)}^{\frac{3p-6}{2p}} \|\phi\|_{L_t^2(H^1)}^{\frac{p+6}{2p}} \|\nabla \psi\|_{L_t^q(L^p)}
		\end{align*}
		and
	\begin{align*}
		\left| \int_0^T ( \varphi\Delta \varphi-\Delta \varphi\varphi :\nabla \psi )dt \right| \leq & \int_0^T \|\Delta \varphi\|_{L^2}\| \varphi\|_{L^{\frac{2p}{p-2}}}\|\nabla \psi\|_{L^p}  dt\\
		\lesssim & \int_0^T \|\Delta \varphi\|_{L^2} \|\varphi\|_{L^3}^{\frac{3p-6}{6+p}} \|\varphi\|_{L^{\frac{4p}{p-2}}}^{\frac{12-2p}{6+p}} \|\nabla \psi\|_{L^p}  dt\\
		\lesssim & \int_0^T \|\Delta \varphi\|_{L^2}\|\varphi\|_{H^1}^{\frac{3p-6}{6+p}}\| \varphi\|_{W^{1,\frac{12p}{7p-6}}}^{\frac{12-2p}{6+p}}\|\nabla \psi\|_{L^p}  dt\\
		\lesssim & \int_0^T \|\Delta \varphi\|_{L^2}\| \varphi\|_{H^1}^{\frac{3p-6}{2p}} \| \varphi\|_{H^2}^{\frac{6-p}{2p}}\|\nabla \psi\|_{L^p}  dt\\
		\lesssim & \left(\int_0^T \|\nabla \psi\|_{L^p}^2(\| \varphi\|_{H^1}^2)^{\frac{3p-6}{2p}} (\| \varphi\|_{H^2}^2)^{\frac{6-p}{2p}} dt \right)^{\frac{1}{2}}\|\Delta \varphi\|_{L_t^2(L^2)} \\
		\lesssim & \left(\int_0^T \|\nabla \psi\|_{L^p}^{\frac{4p}{3p-6}}\|\varphi \|_{H^1}^2 dt \right)^{\frac{3p-6}{4p}} \left( \int_0^T\|\varphi\|_{H^2}^2 dt \right)^{\frac{6-p}{4p}}\|\Delta \varphi\|_{L_t^2(L^2)} \\
		\lesssim &  \|\varphi\|_{L_t^\infty(H^1)}^{\frac{3p-6}{2p}}  \|\varphi\|_{L_t^2(H^2)}^{\frac{6+p}{2p}} \|\nabla \psi\|_{L_t^q(L^p)}.
	\end{align*}
	The proof is complete.
\end{proof}
\begin{lemma}\label{lem2.4}
	Let $2\leq p\leq 6$ and $\Phi\in  L^\infty(0,T;H^1(\R^3))\cap L^2(0,T;H^2(\R^3))$, $\Theta\in  L^2(0,T;L^2(\R^3))$, $\psi\in  L^\infty(0,T;L^2(\R^3))$, $\phi\in  L^\infty(0,T;L^2(\R^3))\cap L^2(0,T;H^1(\R^3))$. If $(\nabla \psi,\Psi )\in L^q(0,T;L^p(\R^3))$ with
	$\frac{2}{q}+\frac{3}{p}=\frac{3}{2}$, we then have
	\begin{equation}
			\begin{split}
				\left| \int_{0}^{T} \left(\psi \cdot \nabla \Phi+ \Phi \nabla \psi: \Theta\right) dt  \right| \lesssim& \| \Theta \|_{L_t^{2}(L^{2})} \Big(\| \psi \|_{L_t^{\infty}(L^2)}^{\frac{1}{2}} \| \Phi \|_{L_t^{\infty}(H^{1})}^{\frac{3p-6}{4p}} \| \Phi \|_{L_t^{2}(H^{2})}^{\frac{p+6}{4p}} \|\nabla\psi\|_{L_t^q(L^p)}^{\frac{1}{2}}\\
				& + \| \Phi \|_{L_t^{\infty}(H^{1})}^{\frac{3p-6}{2p}} \| \Phi \|_{L_t^{2}(H^{2})}^{\frac{6-p}{2p}}\|\nabla\psi\|_{L_t^q(L^p)}  \Big)
			\end{split}
	\end{equation}
	and
	\begin{equation}
			\begin{split}
			&\left| \int_{0}^{T} \left(\phi \cdot \nabla \Phi+ \Phi \nabla \phi: \Psi\right) dt  \right| \\					
			\lesssim&\|  \Psi \|_{L_t^{q}(L^{p})} \left(\| \phi \|_{L_t^{\infty}(L^{2})}^{\frac{3p-6}{4p}}  \| \Phi \|_{L_t^{\infty}(H^{1})}^{\frac{3p-6}{4p}}\left(\| \phi \|_{L_t^{2}(H^{1})}^{\frac{p+6}{2p}}+ \| \Phi \|_{L_t^{2}(H^{2})}^{\frac{p+6}{2p}}\right) + \| \Phi \|_{L_t^{\infty}(H^{1})}^{\frac{3p-6}{2p}} \| \Phi \|_{L_t^{2}(H^{2})}^{\frac{6-p}{2p}} \| \nabla \phi \|_{L_t^{2}(L^{2})}\right).
			\end{split}
	\end{equation}
\end{lemma}
\begin{proof}
	Applying the same procedure as in Lemma~\ref{lem2.2} and Lemma~\ref{lem2.3}, we eventually obtain that
		\begin{align*}
			&\left| \int_{0}^{T} \left(\psi \cdot \nabla \Phi+ \Phi \nabla \psi: \Theta\right) dt  \right| \\
			\lesssim &  \int_0^T  \left( \|\psi\|_{L^{\frac{12p}{p+6}}}\|\nabla \Phi\|_{L^{\frac{12p}{5p-6}}} + \|\nabla\psi\|_{L^p}\| \Phi\|_{L^3}^{\frac{3p-6}{p+6}} \|\Phi\|_{L^{\frac{4p}{p-2}}}^{\frac{12-2p}{p+6}} \right) \|\Theta\|_{L^2}  dt \\	
			\lesssim &\int_0^T  \left( \|\nabla \psi\|_{L^p}^{\frac{1}{2}} \| \psi\|_{L^2}^{\frac{1}{2}} \|\nabla \Phi\|_{L^2}^{\frac{3p-6}{4p}}\|\nabla \Phi\|_{L^6}^{\frac{p+6}{4p}} + \|\nabla\psi\|_{L^p}\|\Phi\|_{H^1}^{\frac{3p-6}{p+6}} \|\Phi\|_{W^{1,\frac{12p}{7p-6}}}^{\frac{12-2p}{p+6}}\right) \|\Theta\|_{L^2}  dt \\	
			\lesssim& \int_0^T  \left( \|\nabla \psi\|_{L^p}^{\frac{1}{2}} \| \psi\|_{L^2}^{\frac{1}{2}}\|\nabla \Phi\|_{L^2}^{\frac{3p-6}{4p}}\|\nabla \Phi\|_{H^1}^{\frac{p+6}{4p}} + \|\nabla\psi\|_{L^p}\|\Phi\|_{H^1}^{\frac{3p-6}{p+6}}\|\Phi\|_{H^1}^{\frac{(5p-6)(6-p)}{2p(p+6)}} \|\Phi\|_{W^{1,6}}^{\frac{(6-p)^2}{2p(p+6)}} \right) \|\Theta\|_{L^2}  dt \\
			\lesssim& \| \Theta \|_{L_t^{2}(L^{2})} \left( \int_0^T  \|\nabla \psi\|_{L^p} \| \psi\|_{L^2}\|\nabla \Phi\|_{L^2}^{\frac{3p-6}{2p}}(\|\nabla \Phi\|_{H^1}^2)^{\frac{p+6}{4p}} + \|\nabla\psi\|_{L^p}^2(\| \Phi\|_{H^1}^2)^{\frac{3p-6}{2p}}(\| \Phi\|_{H^2}^2)^\frac{6-p}{2p}  dt \right)^\frac{1}{2} \\
			\lesssim& \| \Theta \|_{L_t^{2}(L^{2})} \left(\| \psi \|_{L_t^{\infty}(L^2)}^{\frac{1}{2}} \| \Phi \|_{L_t^{\infty}(H^{1})}^{\frac{3p-6}{4p}} \| \Phi \|_{L_t^{2}(H^{2})}^{\frac{p+6}{4p}} \|\nabla\psi\|_{L_t^q(L^p)}^{\frac{1}{2}} + \| \Phi \|_{L_t^{\infty}(H^{1})}^{\frac{3p-6}{2p}} \| \Phi \|_{L_t^{2}(H^{2})}^{\frac{6-p}{2p}}\|\nabla\psi\|_{L_t^q(L^p)}  \right),\\
		&\left| \int_{0}^{T} \left(\phi \cdot \nabla \Phi+ \Phi \nabla \phi: \Psi\right) dt  \right| \\
		\lesssim &  \int_0^T  \left( \|\phi\|_{L^{\frac{12p}{7p-6}}}\|\nabla \Phi\|_{L^{\frac{12p}{5p-6}}} + \|\nabla\phi\|_{L^2}\| \Phi\|_{L^3}^{\frac{3p-6}{p+6}} \|\Phi\|_{L^{\frac{4p}{p-2}}}^{\frac{12-2p}{p+6}} \right) \|\Psi\|_{L^p}  dt \\	
		\lesssim &\int_0^T  \left( \|\phi\|_{L^2}^{\frac{5p-6}{4p}} \| \phi\|_{L^6}^{\frac{6-p}{4p}} \|\nabla \Phi\|_{L^2}^{\frac{3p-6}{4p}}\|\nabla \Phi\|_{L^6}^{\frac{p+6}{4p}} + \|\nabla\phi\|_{L^2}\|\Phi\|_{H^1}^{\frac{3p-6}{p+6}} \|\Phi\|_{W^{1,\frac{12p}{7p-6}}}^{\frac{12-2p}{p+6}}\right)\|\Psi\|_{L^p}  dt \\	
		\lesssim& \int_0^T  \left(  ( \|\phi\|_{L^2}\|\nabla \Phi\|_{L^2})^{\frac{3p-6}{4p}}(\|\phi\|_{H^1}\|\nabla \Phi\|_{H^1})^{\frac{p+6}{4p}} + \|\nabla\phi\|_{L^2}\|\Phi\|_{H^1}^{\frac{3p-6}{p+6}}\|\Phi\|_{H^1}^{\frac{(5p-6)(6-p)}{2p(p+6)}} \|\Phi\|_{W^{1,6}}^{\frac{(6-p)^2}{2p(p+6)}} \right)\|\Psi\|_{L^p}  dt \\
		\lesssim& 
		\left(\int_0^T \| \Psi\|_{L^p}^{\frac{4p}{3p-6}}\|\phi\|_{L^2}\|\nabla \Phi\|_{L^2}  dt\right)^{\frac{3p-6}{4p}}\left(\int_0^T \|\phi\|_{H^1}\|\nabla \Phi\|_{H^1}  dt\right)^{\frac{p+6}{4p}}\\
		& +
		\| \nabla \phi \|_{L_t^{2}(L^{2})} \left( \int_0^T  \|\Psi\|_{L^p}^2(\| \Phi\|_{H^1}^2)^{\frac{3p-6}{2p}}(\| \Phi\|_{H^2}^2)^\frac{6-p}{2p} dt \right)^\frac{1}{2} \\
		\lesssim& \| \phi \|_{L_t^{\infty}(L^{2})}^{\frac{3p-6}{4p}}  \| \Phi \|_{L_t^{\infty}(H^{1})}^{\frac{3p-6}{4p}} \left(\int_0^T \| \Psi\|_{L^p}^{\frac{4p}{3p-6}} dt\right)^{\frac{3p-6}{4p}}\left( \| \phi \|_{L_t^{2}(H^{1})}^{\frac{p+6}{2p}}+ \| \Phi \|_{L_t^{2}(H^{2})}^{\frac{p+6}{2p}}\right)\\
		& +  \| \nabla \phi \|_{L_t^{2}(L^{2})}  \| \Phi \|_{L_t^{\infty}(H^{1})}^{\frac{3p-6}{2p}} \| \Phi \|_{L_t^{2}(H^{2})}^{\frac{6-p}{2p}} \left(\int_0^T \| \Psi\|_{L^p}^{\frac{4p}{3p-6}}dt\right)^{\frac{3p-6}{4p}}  \\					
		\lesssim& \|  \Psi \|_{L_t^{q}(L^{p})} \left(\| \phi \|_{L_t^{\infty}(L^{2})}^{\frac{3p-6}{4p}}  \| \Phi \|_{L_t^{\infty}(H^{1})}^{\frac{3p-6}{4p}}\left(\| \phi \|_{L_t^{2}(H^{1})}^{\frac{p+6}{2p}}+ \| \Phi \|_{L_t^{2}(H^{2})}^{\frac{p+6}{2p}}\right) + \| \Phi \|_{L_t^{\infty}(H^{1})}^{\frac{3p-6}{2p}} \| \Phi \|_{L_t^{2}(H^{2})}^{\frac{6-p}{2p}} \| \nabla \phi \|_{L_t^{2}(L^{2})}\right).
	\end{align*}
	The proof is complete.
\end{proof}
\begin{remark}
	We emphasize that these trilinear terms can be bounded under the condition \eqref{eq1.8}, which yields a much bigger set of indices for the $L_t^qL_x^p$ spaces; see \cite{weak-Q-tensor,weak-strong-3d}. But in the case of $\xi\ne 0$, the condition \eqref{eq:1.5} is a proper choice for the main terms in \eqref{eq1.1}, as one can see later in Section~\ref{sec4}. Since \eqref{eq:1.5} allows $2\leq p\leq 6$, this motivates the above refined analysis on trilinear terms.
\end{remark}

\section{Energy equality of weak solutions}\label{Sec.3}
Before proceeding with the proof of Theorem~\ref{thm1.1}, let us point out that we will prove the energy equality for a class of weak solutions that do not
fit into the Leray-Hopf type weak solution of \eqref{eq1.1}. In other words, the solutions only satisfy the conditions (i), (ii) as stated in Definition~\ref{def1}, which have been studied in \cite{Qtensor11}.

\noindent\textbf{Proof of Theorem~\ref{thm1.1}.}
The proof follows by a similar argument as done in \cite{weak-strong-3d},
which is mainly based on the use of special test functions. In our case, the main point is to deal with some new nonlinear terms due to the appearance of  $\xi$.

\noindent\text{(1).}
First of all, by standard density results, one can extract the approximating sequences $\{ ( Q^i,u^i) \}$ and $\{ ( R^i,v^i) \}$ such that
\begin{align*}
	\{ u^i \}\subset C_c^\infty ([0,T],C_{c,\sigma}^\infty) &\longrightarrow u \quad \text{in}\quad  L^2(0,T;H^1),\\
	\{ \nabla u^i \}\subset C_c^\infty ([0,T],C_c^\infty) &\longrightarrow \nabla u \quad \text{in}\quad  L^q(0,T;L^p),\\
	\{ Q^i \}\subset C_c^\infty ([0,T],C_c^\infty) &\longrightarrow Q \quad \text{in}\quad  L^2(0,T;H^2),\\
	\{ \Delta Q^i \}\subset C_c^\infty ([0,T],C_c^\infty) &\longrightarrow \Delta Q \quad \text{in}\quad  L^q(0,T;L^p),\\
	\{ v^i \}\subset C_c^\infty ([0,T],C_{c,\sigma}^\infty) &\longrightarrow v \quad \text{in}\quad  L^2(0,T;H^1),\\
	\{ R^i \}\subset C_c^\infty ([0,T],C_c^\infty) &\longrightarrow R \quad \text{in}\quad  L^2(0,T;H^2).
\end{align*}
Then, we take $v_{\varepsilon}^{i}$ to be a test function for $u$ in \eqref{eq-a2}.
In this way we get the following identity for all fixed $0<t\leq T$,
\begin{align}\label{eq-a3}
	&\int_0^t \left( u,\partial_tv_{\varepsilon}^{i} \right)+(u,u\cdot\nabla v_{\varepsilon}^{i})-\mu (\nabla u,\nabla v_{\varepsilon}^{i})ds +L\int_0^t \left( \nabla Q \odot\nabla Q - Q\Delta Q+\Delta QQ :\nabla v_{\varepsilon}^{i} \right)ds \nonumber\\
	=& \underbrace{-\xi\int_0^t\left( H[Q]\left[Q +\dfrac{\mathbb{I}_3}{3}\right]+  \left[Q +\dfrac{\mathbb{I}_3}{3}\right]H[Q]: \nabla v_{\varepsilon}^{i} \right)ds  + 2\xi\int_0^t\left(  \left[Q +\dfrac{\mathbb{I}_3}{3}\right]{\rm Tr}(QH[Q]): \nabla v_{\varepsilon}^{i} \right)ds}_{\Theta_{\xi}} \nonumber\\
	& +(u(t),v_{\varepsilon}^{i}(t)) -(u_0(x),v_{\varepsilon}^{i}(0,x)).
\end{align}
According to Lemma~\ref{lem2.2},
it is easy to see that
	\begin{equation}
		\begin{split}
		\left| \int_{0}^{t}(u,u \cdot \nabla (v_{\varepsilon}^{i} - v_{\varepsilon})) ds \right|\lesssim\|u\|_{L_t^\infty(L^2)}^{\frac{5p-6}{4p}} \|u\|_{L_t^2(H^1)}^{\frac{p+6}{4p}} \|\nabla u\|_{L_t^q(L^p)}^{\frac{1}{2}} \|\nabla (v_{\varepsilon}^{i} - v_{\varepsilon})\|_{L_t^2(L^2)},
		\end{split}
	\end{equation}
	\begin{equation}\label{eq3.3}
		\begin{split}
			\left| \int_{0}^{t} (\nabla Q \odot \nabla Q:\nabla (v_{\varepsilon}^{i} - v_{\varepsilon})) ds \right|
			\lesssim\|Q\|_{L_t^\infty(H^1)}^{\frac{5p-6}{4p}} \|Q\|_{L_t^2(H^2)}^{\frac{p+6}{4p}} \|\Delta Q\|_{L_t^q(L^p)}^{\frac{1}{2}} \|\nabla (v_{\varepsilon}^{i} - v_{\varepsilon})\|_{L_t^2(L^2)},
		\end{split}
	\end{equation}
	\begin{equation}\label{eq3.4}
		\left| \int_0^t ( Q\Delta Q-\Delta QQ :\nabla (v_{\varepsilon}^{i} - v_{\varepsilon}) )ds \right| \lesssim   \|Q\|_{L_t^\infty(H^1)}^{\frac{3p-6}{2p}} \|\Delta Q\|_{L_t^q(L^p)} \|Q\|_{L_t^2(H^2)}^{\frac{6-p}{2p}} \|\nabla (v_{\varepsilon}^{i} - v_{\varepsilon})\|_{L_t^2(L^2)}.
	\end{equation}
	For the nonlinear term $\Theta_{\xi}$, we divide it into two parts,
	\[
	\Theta_{\xi} \sim  \underbrace{\int_0^t \left( QH[Q]+H[Q]Q + H[Q] + {\rm Tr}(QH[Q])\mathbb{I}_3 ,\nabla v_{\varepsilon}^{i}\right) ds}_{\Theta_1} + \underbrace{\int_0^t \left( Q {\rm Tr}(QH[Q]) ,\nabla v_{\varepsilon}^{i}\right) ds}_{\Theta_2}.
	\]
	Indeed, from the proof of \cite[Theorem~1.2]{weak-strong-3d}, one has
	\[
		\begin{aligned}
			&\left| \int_{0}^{t} \left(Q \left(-a Q + b\left[Q^{2} - \frac{{\rm Tr}(Q^{2})}{3} \mathbb{I}_3\right] - c Q {\rm Tr}(Q^{2})\right):\nabla (v_{\varepsilon}^{i} - v_{\varepsilon})\right) ds \right|
			\\
			\lesssim&  \sum_{k=1}^{3} \Vert Q \Vert_{L_t^{\infty}(H^{1})}^{k}
			\Vert Q \Vert_{L_t^{2}(H^{2})}
			\Vert \nabla (v_{\varepsilon}^{i} - v_{\varepsilon}) \Vert_{L_t^{2}(L^{2})}.
		\end{aligned}
	\]
	This together with \eqref{eq3.4} yields
	\begin{equation}\label{eq3.6}
		\begin{aligned}
			|\Theta_1|& \lesssim \left( \sum_{k=0}^{3} \Vert Q \Vert_{L_t^{\infty}(H^{1})}^{k}
			\Vert Q \Vert_{L_t^{2}(H^{2})} + \|Q\|_{L_t^\infty(H^1)}^{\frac{3p-6}{2p}} \|\Delta Q\|_{L_t^q(L^p)} \|Q\|_{L_t^2(H^2)}^{\frac{6-p}{2p}}\right) \|\nabla (v_{\varepsilon}^{i} - v_{\varepsilon})\|_{L_t^2(L^2)}.
		\end{aligned}
	\end{equation}
	Next, we show the convergence of nonlinear term $\Theta_2$, which possesses the specific nonlinear structure in the case of $\xi\ne 0$. That is,
	\begin{equation}\label{eq3.7}
		\begin{split}
			&\left| \int_0^t\left(Q {\rm Tr}\left( -aQ^2+bQ\left[Q^2-\dfrac{{\rm Tr}(Q^2)}{3}\cI_3\right]-cQ^2{\rm Tr}(Q^2)\right): \nabla (v_\epsilon^i- v_\epsilon) \right)ds \right|\\
			\lesssim&\int_0^t\left( \|Q\|_{L^6}^3+\|Q\|_{L^8}^4 + \|Q\|_{L^{10}}^5 \right)\|\nabla (v_\epsilon^i- v_\epsilon)\|_{L^2}ds\\
			\lesssim&\left(\int_0^t \|Q\|_{H^1}^6+ \|D^2Q\|_{L^2} \|Q\|_{L^6}^{7} + + \|D^2Q\|_{L^2}^2 \|Q\|_{L^6}^{8} ds \right)^{\frac{1}{2}}\|\nabla (v_\epsilon^i- v_\epsilon)\|_{L_t^2(L^2)}\\
			\lesssim&\left( \sum_{k=2}^4\|Q\|_{L_t^\infty(H^1)}^k  \right)\|Q\|_{L_t^2(H^2)} \|\nabla (v_\epsilon^i- v_\epsilon)\|_{L_t^2(L^2)}
		\end{split}
	\end{equation}
	and
	\begin{align}\label{eq3.8}
		&\left| \int_{0}^{t} (Q{\rm Tr}(Q\Delta Q):\nabla (v_{\varepsilon}^{i} - v_{\varepsilon})) ds \right|\nonumber\\
		\lesssim& \int_{0}^{t} \| Q \|_{L^{\frac{4p}{p-2}}}^{2} \| \Delta Q \|_{L^{p}} \| \nabla (v_{\varepsilon}^{i} - v_{\varepsilon}) \|_{L^{2}} ds\nonumber\\
		\lesssim &\int_{0}^{t} \| Q \|_{W^{1,\frac{12p}{7p-6}}}^{2} \| \Delta Q \|_{L^{p}} \| \nabla (v_{\varepsilon}^{i} - v_{\varepsilon}) \|_{L^{2}} ds\nonumber\\
		\lesssim& \int_{0}^{t} \| Q \|_{H^1}^{\frac{5p-6}{2p}} \| Q \|_{W^{1,6}}^{\frac{6-p}{2p}} \| \Delta Q \|_{L^{p}} \| \nabla (v_{\varepsilon}^{i} - v_{\varepsilon}) \|_{L^{2}} ds\\
		\lesssim& \left(\int_{0}^{t}\left[  \| \Delta Q \|_{L^{p}}^{2}
		(\| Q \|_{H^{1}}^{2})^{\frac{5p-6}{2p}}\right]
		(\| Q \|_{H^{2}}^{2})^{\frac{6-p}{2p}}
		ds\right)^{\frac{1}{2}}
		\|\nabla (v_{\varepsilon}^{i} - v_{\varepsilon})\|_{L_t^2(L^2)} \nonumber\\
		\lesssim&  \left(\int_{0}^{t}  \| \Delta Q \|_{L^{p}}^{\frac{4p}{3p-6}}
		(\| Q \|_{H^{1}}^{2})^{\frac{5p-6}{3p-6}}  ds\right)^{\frac{3p-6}{4p}}
		\|Q\|_{L_t^2(H^2)}^{\frac{6-p}{2p}}
		\|\nabla (v_{\varepsilon}^{i} - v_{\varepsilon})\|_{L_t^2(L^2)}\nonumber\\
		\lesssim& \|\Delta Q\|_{L_t^q(L^p)} \|Q\|_{L_t^\infty(H^1)}^{\frac{5p-6}{2p}} \|Q\|_{L_t^2(H^2)}^{\frac{6-p}{2p}}
		\|\nabla (v_{\varepsilon}^{i} - v_{\varepsilon})\|_{L_t^2(L^2)},\nonumber
	\end{align}
	where we have used H\"{o}lder inequality, Gagliardo-Nirenberg inequality, $L^p$ interpolation and the Sobolev embedding $W^{1,\frac{12p}{7p-6}}\hookrightarrow L^{\frac{4p}{p-2}}$.
	Moreover, one should note that the constraint $2\leq p\leq 6$ must be imposed so that \eqref{eq3.8} holds.
	Then, we conclude from \eqref{eq3.6}, \eqref{eq3.7} and \eqref{eq3.8} that
	\begin{equation}
		\begin{split}
			|\Theta_{\xi}| \lesssim & \|\nabla (v_{\varepsilon}^{i} - v_{\varepsilon})\|_{L_t^2(L^2)} \Bigg( \sum_{k=0}^{4} \Vert Q \Vert_{L_t^{\infty}(H^{1})}^{k}
			\Vert Q \Vert_{L_t^{2}(H^{2})}\\
			& + \left(\|Q\|_{L_t^\infty(H^1)}^{\frac{3p-6}{2p}} + \|Q\|_{L_t^\infty(H^1)}^{\frac{5p-6}{2p}}\right) \|\Delta Q\|_{L_t^q(L^p)} \|Q\|_{L_t^2(H^2)}^{\frac{6-p}{2p}}\Bigg).
		\end{split}
	\end{equation}
	Now, let $i \rightarrow +\infty$ in \eqref{eq-a3}, we get
	\begin{align}\label{eq-b1}
		&\int_0^t \left( u,\partial_tv_{\varepsilon} \right)+(u,u\cdot\nabla v_{\varepsilon})-\mu (\nabla u,\nabla v_{\varepsilon})ds +L\int_0^t \left( \nabla Q \odot\nabla Q - Q\Delta Q+\Delta QQ :\nabla v_{\varepsilon}\right)ds \nonumber\\
		=& -\xi\int_0^t\left( H[Q]\left[Q +\dfrac{\mathbb{I}_3}{3}\right]+  \left[Q +\dfrac{\mathbb{I}_3}{3}\right]H[Q]: \nabla v_{\varepsilon} \right)ds  + 2\xi\int_0^t\left(  \left[Q +\dfrac{\mathbb{I}_3}{3}\right]{\rm Tr}(QH[Q]): \nabla v_{\varepsilon} \right)ds \nonumber\\
		& +(u(t),v_{\varepsilon}(t)) -(u_0(x),v_{\varepsilon}(0,x)).
	\end{align}
	On the other hand, by replacing $(u,Q)$ in \eqref{eq-a2} with $(v,R)$ and letting $\psi = u_{\varepsilon}^{i}$, one has
	\begin{align}\label{eq-a4}
		&\int_0^t \left( v,\partial_tu_{\varepsilon}^i \right)+(v,v\cdot\nabla u_{\varepsilon}^i)-\mu (\nabla v,\nabla u_{\varepsilon}^i)ds +L\int_0^t \left( \nabla R \odot\nabla R - R\Delta R+\Delta RR :\nabla u_{\varepsilon}^i\right)ds \nonumber\\
		=& -\xi\int_0^t\left( H[R]\left[R +\dfrac{\mathbb{I}_3}{3}\right]+  \left[R +\dfrac{\mathbb{I}_3}{3}\right]H[R]: \nabla u_{\varepsilon}^i \right)ds  + 2\xi\int_0^t\left(  \left[R +\dfrac{\mathbb{I}_3}{3}\right]{\rm Tr}(RH[R]): \nabla u_{\varepsilon}^i \right)ds \nonumber\\
		& +(v(t),u_{\varepsilon}^i(t)) -(v_0(x),u_{\varepsilon}^i(0,x)).
	\end{align}
	Then, it follows from Lemma~\ref{lem2.3} that
	\begin{equation}
		\left| \int_{0}^{t}(v,v \cdot  \nabla (u_\varepsilon^i- u_\varepsilon)) ds \right|\lesssim\|v\|_{L_t^\infty(L^2)}^{\frac{3p-6}{2p}} \|v\|_{L_t^2(H^1)}^{\frac{p+6}{2p}} \|\nabla (u_\varepsilon^i- u_\varepsilon)\|_{L_t^q(L^p)},
	\end{equation}
	\begin{equation}\label{eq3.12}
		\left| \int_{0}^{t}(\nabla R \odot \nabla R - R \Delta R+ \Delta R R:  \nabla (u_\varepsilon^i- u_\varepsilon)) ds \right|\lesssim\|R\|_{L_t^\infty(H^1)}^{\frac{3p-6}{2p}} \|R\|_{L_t^2(H^2)}^{\frac{p+6}{2p}} \|\nabla (u_\varepsilon^i- u_\varepsilon)\|_{L_t^q(L^p)}.
	\end{equation}
	Notice that most of the remaining terms in \eqref{eq-a4} are similar to that of \eqref{eq-a3}. So we focus on some terms that are not easy to deal with. That is,
	\begin{align}\label{eq3.14}
		\left| \int_{0}^{t} (R{\rm Tr}(R\Delta R):\nabla (u_{\varepsilon}^{i} - u_{\varepsilon})) ds \right|
		& \leq \int_{0}^{t} \| R \|_{L^{\frac{4p}{p-2}}}^{2} \| \Delta R \|_{L^{2}} \| \nabla (u_{\varepsilon}^{i} - u_{\varepsilon}) \|_{L^{p}} ds\nonumber\\
		& \lesssim \int_{0}^{t} \| R \|_{W^{1,\frac{12p}{7p-6}}}^{2} \| \Delta R \|_{L^{2}} \| \nabla (u_{\varepsilon}^{i} - u_{\varepsilon}) \|_{L^{p}} ds\nonumber\\
		& \lesssim\int_{0}^{t} \| R \|_{H^1}^{\frac{5p-6}{2p}} \| R \|_{W^{1,6}}^{\frac{6-p}{2p}} \| \Delta R\|_{L^{2}} \| \nabla (u_{\varepsilon}^{i} - u_{\varepsilon}) \|_{L^{p}} ds\\
		& \lesssim\left(\int_{0}^{t}\left[  \| \nabla (u_{\varepsilon}^{i} - u_{\varepsilon}) \|_{L^{p}}^{2}
		(\| R \|_{H^{1}}^{2})^{\frac{5p-6}{2p}}\right]
		(\|R \|_{H^{2}}^{2})^{\frac{6-p}{2p}}
		ds\right)^{\frac{1}{2}}
		\|\Delta R \|_{L_t^2(L^2)} \nonumber\\
		&\lesssim \left(\int_{0}^{t}  \| \nabla (u_{\varepsilon}^{i} - u_{\varepsilon}) \|_{L^{p}}^{\frac{4p}{3p-6}}
		(\| R \|_{H^{1}}^{2})^{\frac{5p-6}{3p-6}}  ds\right)^{\frac{3p-6}{4p}}
		\|R\|_{L_t^2(H^2)}^{\frac{6-p}{2p}}
		\|\Delta R\|_{L_t^2(L^2)}\nonumber\\
		& \lesssim \|R\|_{L_t^\infty(H^1)}^{\frac{5p-6}{2p}} \|R\|_{L_t^2(H^2)}^{\frac{p+6}{2p}}
		\|\nabla (u_{\varepsilon}^{i} - u_{\varepsilon})\|_{L_t^q(L^p)}.\nonumber
	\end{align}
	Similarly to \eqref{eq3.8} the above estimate implies $2\leq p\leq 6$.
	Then, one can send $i$ to infinity in \eqref{eq-a4} and get
	\begin{align}\label{eq-b2}
		&\int_0^t \left( v,\partial_tu_{\varepsilon} \right)+(v,v\cdot\nabla u_{\varepsilon})-\mu (\nabla v,\nabla u_{\varepsilon})ds +L\int_0^t \left( \nabla R \odot\nabla R - R\Delta R+\Delta RR :\nabla u_{\varepsilon}\right)ds \nonumber\\
		=& -\xi\int_0^t\left( H[R]\left[R +\dfrac{\mathbb{I}_3}{3}\right]+  \left[R +\dfrac{\mathbb{I}_3}{3}\right]H[R]: \nabla u_{\varepsilon} \right)ds  + 2\xi\int_0^t\left(  \left[R +\dfrac{\mathbb{I}_3}{3}\right]{\rm Tr}(RH[R]): \nabla u_{\varepsilon} \right)ds \nonumber\\
		& +(v(t),u_{\varepsilon}(t)) -(v_0(x),u_{\varepsilon}(0,x)).
	\end{align}
\noindent\text{(2).}
In this step, we first sum up \eqref{eq-b1} and \eqref{eq-b2} to have
\begin{align}\label{eq3.16}
	&\int_0^t \left( v,\partial_tu_{\varepsilon} \right) + \left( u,\partial_tv_{\varepsilon} \right) +(v,v\cdot\nabla u_{\varepsilon}) +(u,u\cdot\nabla v_{\varepsilon})-\mu (\nabla u,\nabla v_{\varepsilon})-\mu (\nabla v,\nabla u_{\varepsilon})ds \nonumber\\
	&+L\int_0^t \left( \nabla R \odot\nabla R - R\Delta R+\Delta RR :\nabla u_{\varepsilon}\right)ds 
	+L\int_0^t \left( \nabla Q \odot\nabla Q - Q\Delta Q+\Delta QQ :\nabla v_{\varepsilon}\right)ds \nonumber\\
	=& -\xi\int_0^t\left( H[R]\left[R +\dfrac{\mathbb{I}_3}{3}\right]+  \left[R +\dfrac{\mathbb{I}_3}{3}\right]H[R]: \nabla u_{\varepsilon} \right)ds  + 2\xi\int_0^t\left(  \left[R +\dfrac{\mathbb{I}_3}{3}\right]{\rm Tr}(RH[R]): \nabla u_{\varepsilon} \right)ds \nonumber\\
	& -\xi\int_0^t\left( H[Q]\left[Q +\dfrac{\mathbb{I}_3}{3}\right]+  \left[Q +\dfrac{\mathbb{I}_3}{3}\right]H[Q]: \nabla v_{\varepsilon} \right)ds  + 2\xi\int_0^t\left(  \left[Q +\dfrac{\mathbb{I}_3}{3}\right]{\rm Tr}(QH[Q]): \nabla v_{\varepsilon} \right)ds \nonumber\\
	&  +(v(t),u_{\varepsilon}(t))  +(u(t),v_{\varepsilon}(t))-(v_0(x),u_{\varepsilon}(0,x)) -(u_0(x),v_{\varepsilon}(0,x)).
\end{align}
Actually, we wish now to prove \eqref{eq:1.6} by taking the limit in \eqref{eq3.16} as $\varepsilon$ goes to zero.
Before proceeding any further, we recall that
\begin{align}
	&\int_{0}^{t} (u,\partial_{t} v_{\varepsilon})+ (v,\partial_{t} u_{\varepsilon})ds = 0,\\
	&\underset{\varepsilon \rightarrow 0}{\lim}\int_{0}^{t} (\nabla u, \nabla v_{\varepsilon})+ (\nabla v,\nabla u_{\varepsilon})ds = 2\int_{0}^{t} (\nabla u, \nabla v)ds,\\
	&\underset{\varepsilon \rightarrow 0}{\lim}
	(u_{0},v_{\varepsilon} (0,x))= \underset{\varepsilon \rightarrow 0}{\lim} (v_{0},u_{\varepsilon} (0,x))=\frac{1}{2}\|u_{0}\|_{L^2}^2,\\
	&\underset{\varepsilon \rightarrow 0}{\lim}(u(t),v_{\varepsilon}(t))= \underset{\varepsilon \rightarrow 0}{\lim} (v(t),u_{\varepsilon}(t))
	= \frac{1}{2}(u(t),v(t)).
\end{align}
These facts are well-known in \cite{Galdi-book,serrin1963,shinbrot1974}, which basically come from the choice of mollifier and the weak $L^2$ continuity of $u$. Moreover, by using Lemma~\ref{lem2.2} and Lemma~\ref{lem2.3} again, we have
\[
\begin{split}
	&\left| \int_{0}^{t}(u,u \cdot \nabla (v_{\varepsilon} - v)) ds \right|\lesssim\|u\|_{L_t^\infty(L^2)}^{\frac{5p-6}{4p}} \|u\|_{L_t^2(H^1)}^{\frac{p+6}{4p}} \|\nabla u\|_{L_t^q(L^p)}^{\frac{1}{2}} \|\nabla (v_{\varepsilon} - v)\|_{L_t^2(L^2)}\\
	&\left| \int_{0}^{t}(v,v \cdot  \nabla (u_\varepsilon- u)) ds \right|\lesssim\|v\|_{L_t^\infty(L^2)}^{\frac{3p-6}{2p}} \|v\|_{L_t^2(H^1)}^{\frac{p+6}{2p}} \|\nabla (u_\varepsilon- u)\|_{L_t^q(L^p)},
\end{split}
\]
and thus
\[
\lim_{\varepsilon\rightarrow 0}\int_0^t (v,v\cdot\nabla u_{\varepsilon}) +(u,u\cdot\nabla v_{\varepsilon})ds= \int_0^t (v,v\cdot\nabla u) +(u,u\cdot\nabla v)ds.
\]
Again, a similar argument as \eqref{eq3.3}-\eqref{eq3.8}, \eqref{eq3.12} and \eqref{eq3.14}, we also have
\begin{align*}
	&\lim_{\varepsilon\rightarrow 0} \int_0^t \left( \nabla R \odot\nabla R - R\Delta R+\Delta RR :\nabla u_{\varepsilon}\right)ds =\int_0^t \left( \nabla R \odot\nabla R - R\Delta R+\Delta RR :\nabla u\right)ds,\\
	& \lim_{\varepsilon\rightarrow 0}\int_0^t\left( H[R]\left[R +\dfrac{\mathbb{I}_3}{3}\right]+  \left[R +\dfrac{\mathbb{I}_3}{3}\right]H[R] -2\left[R +\dfrac{\mathbb{I}_3}{3}\right]{\rm Tr}(RH[R]): \nabla u_{\varepsilon} \right)ds  \\
	&\qquad\qquad\qquad=\int_0^t\left( H[R]\left[R +\dfrac{\mathbb{I}_3}{3}\right]+  \left[R +\dfrac{\mathbb{I}_3}{3}\right]H[R] -2\left[R +\dfrac{\mathbb{I}_3}{3}\right]{\rm Tr}(RH[R]): \nabla u \right)ds,  \\
	&\lim_{\varepsilon\rightarrow 0} \int_0^t \left( \nabla Q \odot\nabla Q - Q\Delta Q+\Delta QQ :\nabla v_{\varepsilon}\right)ds = \int_0^t \left( \nabla Q \odot\nabla Q - Q\Delta Q+\Delta QQ :\nabla v\right)ds,\\
	& \lim_{\varepsilon\rightarrow 0}\int_0^t\left( H[Q]\left[Q +\dfrac{\mathbb{I}_3}{3}\right]+  \left[Q +\dfrac{\mathbb{I}_3}{3}\right]H[Q]- 2\left[Q +\dfrac{\mathbb{I}_3}{3}\right]{\rm Tr}(QH[Q]): \nabla v_{\varepsilon} \right)ds  \\
	&\qquad\qquad\qquad=\int_0^t\left( H[Q]\left[Q +\dfrac{\mathbb{I}_3}{3}\right]+  \left[Q +\dfrac{\mathbb{I}_3}{3}\right]H[Q]- 2\left[Q +\dfrac{\mathbb{I}_3}{3}\right]{\rm Tr}(QH[Q]): \nabla v \right)ds.
\end{align*}
With the above results,
together with $\nabla\cdot u=\nabla\cdot v=0$,
one can obtain the desired \eqref{eq:1.6} by taking the limit in \eqref{eq3.16} as $\varepsilon\rightarrow 0$.

\noindent\text{(3).}
Now, we turn to prove \eqref{eq:1.7} by the same procedure as for \eqref{eq:1.6}.
More precisely, one can set $\phi = R_{\varepsilon}^{i} - L \Delta R_{\varepsilon}^{i}$ in \eqref{eq-a1} and integrate by parts to get
\begin{align}\label{eq-d1}
	\int_0^t& (Q:\partial_{t} R_{\varepsilon}^{i})
	+L (\nabla Q:\partial_{t} \nabla R_{\varepsilon}^{i}) ds -(a+1)\Gamma L\int_{0}^{t}  (\nabla Q:\nabla R_{\varepsilon}^{i})ds
	- \Gamma L^{2} \int_0^t(\Delta Q:\Delta R_{\varepsilon}^{i})
	ds \nonumber\\
	&+ \int_{0}^{t} (Q:u \cdot \nabla R_{\varepsilon}^{i}) +L (u \cdot \nabla Q:\Delta R_{\varepsilon}^{i})ds +\dfrac{2\xi}{3}\int_0^t ( D:R_{\varepsilon}^{i} - L \Delta R_{\varepsilon}^{i})ds \nonumber\\
	&+ \underbrace{\int_0^t \left(\xi ( DQ+QD ) + ( \Omega Q-Q\Omega)  -2\xi   \left[Q +\dfrac{\mathbb{I}_3}{3}\right] {\rm Tr}(Q\nabla u):R_{\varepsilon}^{i} - L \Delta R_{\varepsilon}^{i} \right) ds}_{\Pi_1}\\
	=&  a\Gamma  \int_{0}^{t} (Q:R_{\varepsilon}^{i} )ds -\Gamma\int_0^t \left( b\left[Q^2-\dfrac{{\rm Tr}(Q^2)}{3}\mathbb{I}_3\right]-cQ{\rm Tr}(Q^2) :R_{\varepsilon}^{i} - L \Delta R_{\varepsilon}^{i} \right)ds\nonumber\\
	& +(Q(t):R_{\varepsilon}^{i}(t)) + L (\nabla Q(t):\nabla R_{\varepsilon}^{i}(t)) - (Q_{0}:R_{\varepsilon}^{i}(0,x)) - L (\nabla Q_{0}:\nabla R_{\varepsilon}^{i}(0,x)).\nonumber
\end{align}
Obviously, the new nonlinear term $\Pi_1$ comes from the nonzero parameter $\xi$, and it is equivalent to the sum of the following four terms
\[
\int_0^t  (Q\nabla u : R_{\varepsilon}^{i})ds, \int_0^t  (Q\nabla u :\Delta R_{\varepsilon}^{i}) ds, 
\int_0^t  (Q{\rm Tr}(Q\nabla u) : R_{\varepsilon}^{i})  ds,
\int_0^t  (Q{\rm Tr}(Q\nabla u) : \Delta R_{\varepsilon}^{i}) ds.
\]
Similarly, we study the previous equality by taking the limit for $i\rightarrow +\infty$, with $\varepsilon$ fixed. For this purpose we need to have some estimate of the nonlinear terms in \eqref{eq-d1} and these are done as follows:
\begin{equation}\label{eq3.21}
	\begin{split}
		&\left| \int_{0}^{t} (Q:u \cdot \nabla (R_{\varepsilon}^{i} - R_{\varepsilon})) + (Q\nabla u : R_{\varepsilon}^{i}- R_{\varepsilon}) + (Q {\rm Tr}(Q \nabla u):R_{\varepsilon}^{i} - R_{\varepsilon}) ds \right|\\
		\lesssim&  \int_{0}^{t} \| Q \|_{L^{3}} \left(\| u \|_{L^{6}} \| \nabla (R_{\varepsilon}^{i} - R_{\varepsilon}) \|_{L^{2}} + \| \nabla u \|_{L^{2}} \| (R_{\varepsilon}^{i} - R_{\varepsilon} \|_{L^{6}}\right) + \|Q\|_{L^6}^2\|\nabla u\|_{L^2}\|R_{\varepsilon}^{i} - R_{\varepsilon}\|_{L^6} ds\\
		\lesssim &  \int_{0}^{t} \left( \| Q \|_{H^1} + \|Q\|_{H^1}^2\right) \| \nabla u \|_{L^2} \| \nabla (R_{\varepsilon}^{i} - R_{\varepsilon}) \|_{L^{2}} ds\\
		\lesssim &  \sum_{n=1}^2\| Q \|_{L_t^{\infty}(H^{1})}^n\|\nabla u \|_{L_t^{2}(L^{2})}  \| \nabla (R_{\varepsilon}^{i} - R_{\varepsilon}) \|_{L_t^{2}(L^{2})},
	\end{split}
\end{equation}
\begin{equation}\label{eq3.22}
	\begin{split}
		\left| \int_{0}^{t} \left(u \cdot \nabla Q+ Q \nabla u: \Delta (R_{\varepsilon}^{i} - R_{\varepsilon})\right) ds  \right| \lesssim& \|\Delta (R_{\varepsilon}^{i} - R_{\varepsilon}) \|_{L_t^{2}(L^{2})} \Big(  \| Q \|_{L_t^{\infty}(H^{1})}^{\frac{3p-6}{2p}} \| Q \|_{L_t^{2}(H^{2})}^{\frac{6-p}{2p}}\|\nabla u\|_{L_t^q(L^p)} \\
		& + \| u \|_{L_t^{\infty}(L^2)}^{\frac{1}{2}} \| Q \|_{L_t^{\infty}(H^{1})}^{\frac{3p-6}{4p}} \| Q \|_{L_t^{2}(H^{2})}^{\frac{p+6}{4p}} \|\nabla u\|_{L_t^q(L^p)}^{\frac{1}{2}}
		 \Big),
	\end{split}
\end{equation}
\begin{equation}\label{eq3.23}
\left| \int_{0}^{t} \left(  b\left[Q^{2} - \frac{{\rm Tr}(Q^{2})}{3} \mathbb{I}_3\right] - c Q {\rm Tr}(Q^{2}) : R_{\varepsilon}^{i} - L \Delta R_{\varepsilon}^{i}\right) ds \right| \lesssim  \sum_{k=1}^{2} \Vert Q \Vert_{L_t^{\infty}(H^{1})}^{k}
	\Vert Q \Vert_{L_t^{2}(H^{2})}
	\Vert R_{\varepsilon}^{i} - \Delta R_{\varepsilon}^{i} \Vert_{L_t^{2}(L^{2})},
\end{equation}
\begin{align}\label{eq3.24}
	&\left| \int_{0}^{t} (Q {\rm Tr}(Q \nabla u):\Delta (R_{\varepsilon}^{i} - R_{\varepsilon})) ds \right|\nonumber\\
	\leq&  \int_{0}^{t} \| Q \|_{L^{\frac{4p}{p-2}}}^{2}\| \nabla u\|_{L^{p}} \| \Delta (R_{\varepsilon}^{i} - R_{\varepsilon}) \|_{L^{2}}  ds\nonumber\\
	\lesssim& \int_{0}^{t} \| Q \|_{W^{1,\frac{12p}{7p-6}}}^{2} \| \nabla u\|_{L^{p}} \| \Delta (R_{\varepsilon}^{i} - R_{\varepsilon}) \|_{L^{2}} ds\nonumber\\
	\lesssim& \int_{0}^{t} \| Q \|_{H^1}^{\frac{5p-6}{2p}} \| Q \|_{W^{1,6}}^{\frac{6-p}{2p}} \| \nabla u\|_{L^{p}} \| \Delta (R_{\varepsilon}^{i} - R_{\varepsilon}) \|_{L^{2}} ds\\
	\lesssim&  \left(\int_{0}^{t}\left[  \| \nabla u\|_{L^{p}}^2
	(\| Q \|_{H^{1}}^{2})^{\frac{5p-6}{2p}}\right]
	(\|Q \|_{H^{2}}^{2})^{\frac{6-p}{2p}}
	ds\right)^{\frac{1}{2}} \| \Delta (R_{\varepsilon}^{i} - R_{\varepsilon}) \|_{L_t^2(L^2)} \nonumber\\
	\lesssim& \left(\int_{0}^{t}  \| \nabla u \|_{L^{p}}^{\frac{4p}{3p-6}} ds\right)^{\frac{3p-6}{4p}}
	\| Q \|_{L_t^{\infty}(H^{1})}^{\frac{5p-6}{2p}}
	\| Q \|_{L_t^{2}(H^{2})}^{\frac{6-p}{2p}}
	\|\Delta (R_{\varepsilon}^{i} - R_{\varepsilon})\|_{L_t^{2}(L^{2})}\nonumber\\
	\lesssim&\| \nabla u \|_{L_t^{q}(L^{p})}
	\| Q \|_{L_t^{\infty}(H^{1})}^{\frac{5p-6}{2p}}
	\| Q \|_{L_t^{2}(H^{2})}^{\frac{6-p}{2p}}
	\|\Delta (R_{\varepsilon}^{i} - R_{\varepsilon})\|_{L_t^{2}(L^{2})}.\nonumber
\end{align}
Here we remark that \eqref{eq3.22} follows directly from Lemma~\ref{lem2.4}, and \eqref{eq3.24} is again only valid for $2\leq p\leq 6$.

From the above estimates, letting $i \rightarrow +\infty$ in \eqref{eq-d1}, we then get
\begin{equation}\label{eq-d2}
	\begin{split}
		\int_0^t& (Q:\partial_{t} R_{\varepsilon})
		+L (\nabla Q:\partial_{t} \nabla R_{\varepsilon}) ds -(a+1)\Gamma L\int_{0}^{t}  (\nabla Q:\nabla R_{\varepsilon})ds
		- \Gamma L^{2} \int_0^t(\Delta Q:\Delta R_{\varepsilon})
		ds \\
		&+ \int_{0}^{t} (Q:u \cdot \nabla R_{\varepsilon}) +L (u \cdot \nabla Q:\Delta R_{\varepsilon}) +\dfrac{2\xi}{3}\int_0^t ( D:R_{\varepsilon} - L \Delta R_{\varepsilon})ds \\
		&+ \int_0^t \left(\xi ( DQ+QD ) + ( \Omega Q-Q\Omega)  -2\xi   \left[Q +\dfrac{\mathbb{I}_3}{3}\right] {\rm Tr}(Q\nabla u):R_{\varepsilon} - L \Delta R_{\varepsilon} \right) ds\\
		=&  a\Gamma  \int_{0}^{t} (Q:R_{\varepsilon} )ds -\Gamma\int_0^t \left( b\left[Q^2-\dfrac{{\rm Tr}(Q^2)}{3}\mathbb{I}_3\right]-cQ{\rm Tr}(Q^2) :R_{\varepsilon}  - L \Delta R_{\varepsilon}  \right)ds\\
		& +(Q(t):R_{\varepsilon} (t)) + L (\nabla Q(t):\nabla R_{\varepsilon} (t)) - (Q_{0}:R_{\varepsilon} (0,x)) - L (\nabla Q_{0}:\nabla R_{\varepsilon} (0,x)).
	\end{split}
\end{equation}
Analogously, by substituting $(u,Q)$ with $(v,R)$ in \eqref{eq-a1} and letting $\phi = Q_{\varepsilon}^{i} - L  \Delta Q_{\varepsilon}^{i}$, one can similarly find
\begin{align}\label{eq-d3}
	\int_0^t& (R:\partial_{t} Q_{\varepsilon}^{i})
	+L (\nabla R:\partial_{t} \nabla Q_{\varepsilon}^{i}) ds -(a+1)\Gamma L\int_{0}^{t}  (\nabla R:\nabla Q_{\varepsilon}^{i})ds
	- \Gamma L^{2} \int_0^t(\Delta R:\Delta Q_{\varepsilon}^{i})
	ds \nonumber\\
	&+ \int_{0}^{t} (R:v \cdot \nabla Q_{\varepsilon}^{i}) +L (v \cdot \nabla R:\Delta Q_{\varepsilon}^{i}) +\dfrac{2\xi}{3}\int_0^t ( \widebar{D}:Q_{\varepsilon}^{i} - L \Delta Q_{\varepsilon}^{i})ds \nonumber\\
	&+ \int_0^t \left(\xi ( \widebar{D}R+R\widebar{D} ) + ( \widebar{\Omega} R-R\widebar{\Omega})  -2\xi   \left[R +\dfrac{\mathbb{I}_3}{3}\right] {\rm Tr}(R\nabla v):Q_{\varepsilon}^{i} - L \Delta Q_{\varepsilon}^{i} \right) ds\\
	=&  a\Gamma  \int_{0}^{t} (R:Q_{\varepsilon}^{i} )ds -\Gamma\int_0^t \left( b\left[R^2-\dfrac{{\rm Tr}(R^2)}{3}\mathbb{I}_3\right]-cR{\rm Tr}(R^2) :Q_{\varepsilon}^{i} - L \Delta Q_{\varepsilon}^{i} \right)ds\nonumber\\
	& +(R(t):Q_{\varepsilon}^{i}(t)) + L (\nabla R(t):\nabla Q_{\varepsilon}^{i}(t)) - (R_{0}:Q_{\varepsilon}^{i}(0,x)) - L (\nabla R_{0}:\nabla Q_{\varepsilon}^{i}(0,x)),\nonumber
\end{align}
%
where
\[
\widebar{D}:=\frac{1}{2}\left( \nabla v + \nabla v^{T} \right),\quad \widebar{\Omega}:=\frac{1}{2}\left( \nabla v - \nabla v^{T} \right).
\]
Since the estimates of certain terms in the above equation are similar to \eqref{eq3.21}, \eqref{eq3.22}, we just show some tricky ones in the following:
\begin{align}\label{eq3.27}
	\left| \int_{0}^{t} (R {\rm Tr}(R \nabla v):\Delta (Q_{\varepsilon}^{i} - Q_{\varepsilon})) ds \right| \leq &  \int_{0}^{t} \| R \|_{L^{\frac{4p}{p-2}}}^{2}\| \nabla v\|_{L^{2}} \| \Delta (Q_{\varepsilon}^{i} - Q_{\varepsilon}) \|_{L^{p}}  ds\nonumber\\
	\lesssim&  \int_{0}^{t} \| R \|_{W^{1,\frac{12p}{7p-6}}}^{2} \| \nabla v\|_{L^{2}} \| \Delta (Q_{\varepsilon}^{i} - Q_{\varepsilon}) \|_{L^{p}} ds\nonumber\\
	\lesssim& \int_{0}^{t} \| R \|_{H^1}^{\frac{5p-6}{2p}} \| R \|_{W^{1,6}}^{\frac{6-p}{2p}} \| \nabla v\|_{L^{2}} \| \Delta (Q_{\varepsilon}^{i} - Q_{\varepsilon}) \|_{L^{p}} ds\\
	\lesssim& \left(\int_{0}^{t}\left[  \| \Delta (Q_{\varepsilon}^{i} - Q_{\varepsilon})\|_{L^{p}}^2
	(\| R \|_{H^{1}}^{2})^{\frac{5p-6}{2p}}\right]
	(\|R \|_{H^{2}}^{2})^{\frac{6-p}{2p}}
	ds\right)^{\frac{1}{2}} \| \nabla v \|_{L_t^2(L^2)} \nonumber\\
	\lesssim& \left(\int_{0}^{t}  \| \Delta (Q_{\varepsilon}^{i} - Q_{\varepsilon}) \|_{L^{p}}^{\frac{4p}{3p-6}} ds\right)^{\frac{3p-6}{4p}}
	\| R \|_{L_t^{\infty}(H^{1})}^{\frac{5p-6}{2p}}
	\| R \|_{L_t^{2}(H^{2})}^{\frac{6-p}{2p}}
	\| \nabla v \|_{L_t^{2}(L^{2})}\nonumber\\
	\lesssim& \| \nabla v \|_{L_t^{2}(L^{2})}
	\| R \|_{L_t^{\infty}(H^{1})}^{\frac{5p-6}{2p}}
	\| R \|_{L_t^{2}(H^{2})}^{\frac{6-p}{2p}}
	\|\Delta (Q_{\varepsilon}^{i} - Q_{\varepsilon})\|_{L_t^{q}(L^{p})},\nonumber
\end{align}
\begin{equation}\label{eq3.28}
	\begin{split}
		\left| \int_{0}^{T} \left(v\cdot \nabla R+ R \nabla v: \Delta (Q_{\varepsilon}^{i} - Q_{\varepsilon})\right) dt  \right|\lesssim&\|\Delta (Q_{\varepsilon}^{i} - Q_{\varepsilon})\|_{L_t^{q}(L^{p})} \Bigg[\| v \|_{L_t^{\infty}(L^{2})}^{\frac{3p-6}{4p}}  \| R \|_{L_t^{\infty}(H^{1})}^{\frac{3p-6}{4p}}\left(\| v \|_{L_t^{2}(H^{1})}^{\frac{p+6}{2p}}+ \| R \|_{L_t^{2}(H^{2})}^{\frac{p+6}{2p}}\right) \\
		&+ \| R \|_{L_t^{\infty}(H^{1})}^{\frac{3p-6}{2p}} \| R \|_{L_t^{2}(H^{2})}^{\frac{6-p}{2p}} \| \nabla v \|_{L_t^{2}(L^{2})}\Bigg].
	\end{split}
\end{equation}
One should note that \eqref{eq3.28} is owing to Lemma~\ref{lem2.4}, and the estimate \eqref{eq3.27} only holds for $2\leq p\leq 6$.

Passing to the limit as $i \rightarrow +\infty$ in \eqref{eq-d3} we then get
\begin{equation}\label{eq-d4}
	\begin{split}
		\int_0^t& (R:\partial_{t} Q_{\varepsilon})
		+L (\nabla R:\partial_{t} \nabla Q_{\varepsilon}) ds -(a+1)\Gamma L\int_{0}^{t}  (\nabla R:\nabla Q_{\varepsilon})ds
		- \Gamma L^{2} \int_0^t(\Delta R:\Delta Q_{\varepsilon})
		ds \\
		&+ \int_{0}^{t} (R:v \cdot \nabla Q_{\varepsilon}) +L (v \cdot \nabla R:\Delta Q_{\varepsilon}) +\dfrac{2\xi}{3}\int_0^t ( \widebar{D}:Q_{\varepsilon} - L \Delta Q_{\varepsilon})ds \\
		&+ \int_0^t \left(\xi ( \widebar{D}R+R\widebar{D} ) + ( \widebar{\Omega} R-R\widebar{\Omega})  -2\xi   \left[R +\dfrac{\mathbb{I}_3}{3}\right] {\rm Tr}(R\nabla v):Q_{\varepsilon} - L \Delta Q_{\varepsilon} \right) ds\\
		=&  a\Gamma  \int_{0}^{t} (R:Q_{\varepsilon} )ds -\Gamma\int_0^t \left( b\left[R^2-\dfrac{{\rm Tr}(R^2)}{3}\mathbb{I}_3\right]-cR{\rm Tr}(R^2) :Q_{\varepsilon} - L \Delta Q_{\varepsilon} \right)ds\\
		& +(R(t):Q_{\varepsilon}(t)) + L (\nabla R(t):\nabla Q_{\varepsilon}(t)) - (R_{0}:Q_{\varepsilon}(0,x)) - L (\nabla R_{0}:\nabla Q_{\varepsilon}(0,x)).
	\end{split}
\end{equation}

\noindent\text{(4).}
Since the summation of \eqref{eq-d2} and \eqref{eq-d4} gives
\begin{equation}\label{eq3.30}
	\begin{split}
		\int_0^t& (Q:\partial_{t} R_{\varepsilon})
		 + (R:\partial_{t} Q_{\varepsilon})+L (\nabla Q:\partial_{t} \nabla R_{\varepsilon})
		+L (\nabla R:\partial_{t} \nabla Q_{\varepsilon}) ds \\
		& -(a+1)\Gamma L\int_{0}^{t}  (\nabla Q:\nabla R_{\varepsilon})+  (\nabla R:\nabla Q_{\varepsilon})ds
		- \Gamma L^{2} \int_0^t(\Delta Q:\Delta R_{\varepsilon})+(\Delta R:\Delta Q_{\varepsilon})
		ds \\
		&+ \int_{0}^{t} (R:v \cdot \nabla Q_{\varepsilon}) +(Q:u \cdot \nabla R_{\varepsilon}) +L (u \cdot \nabla Q:\Delta R_{\varepsilon})+L (v \cdot \nabla R:\Delta Q_{\varepsilon})ds\\
		& +\dfrac{2\xi}{3}\int_0^t ( \widebar{D}:Q_{\varepsilon} - L \Delta Q_{\varepsilon})+ ( D:R_{\varepsilon} - L \Delta R_{\varepsilon})ds - a\Gamma  \int_{0}^{t} (R:Q_{\varepsilon} )+ (Q:R_{\varepsilon} )ds\\
		&+ \int_0^t \left(\xi ( \widebar{D}R+R\widebar{D} ) + ( \widebar{\Omega} R-R\widebar{\Omega})  -2\xi   \left[R +\dfrac{\mathbb{I}_3}{3}\right] {\rm Tr}(R\nabla v):Q_{\varepsilon} - L \Delta Q_{\varepsilon} \right) ds\\
		&+ \int_0^t \left(\xi ( DQ+QD ) + ( \Omega Q-Q\Omega)  -2\xi   \left[Q +\dfrac{\mathbb{I}_3}{3}\right] {\rm Tr}(Q\nabla u):R_{\varepsilon} - L \Delta R_{\varepsilon} \right) ds\\
		=&   -\Gamma\int_0^t \left( b\left[R^2-\dfrac{{\rm Tr}(R^2)}{3}\mathbb{I}_3\right]-cR{\rm Tr}(R^2) :Q_{\varepsilon} - L \Delta Q_{\varepsilon} \right)ds\\
		& -\Gamma\int_0^t \left( b\left[Q^2-\dfrac{{\rm Tr}(Q^2)}{3}\mathbb{I}_3\right]-cQ{\rm Tr}(Q^2) :R_{\varepsilon}  - L \Delta R_{\varepsilon}  \right)ds\\
		& +(R(t):Q_{\varepsilon}(t)) + L (\nabla R(t):\nabla Q_{\varepsilon}(t)) - (R_{0}:Q_{\varepsilon}(0,x)) - L (\nabla R_{0}:\nabla Q_{\varepsilon}(0,x))\\
		& +(Q(t):R_{\varepsilon} (t)) + L (\nabla Q(t):\nabla R_{\varepsilon} (t)) - (Q_{0}:R_{\varepsilon} (0,x)) - L (\nabla Q_{0}:\nabla R_{\varepsilon} (0,x)).
	\end{split}
\end{equation}
Similarly to the proof of \eqref{eq:1.6}, by passing to the $\varepsilon\rightarrow 0$ limit in the previous equality, one can derive \eqref{eq:1.7}.
For simplicity, we omit the details here and leave it to the interested readers.
We also remark that, in order to obtain the convergence of every term in \eqref{eq3.30} as $\varepsilon$ goes to zero, we need to achieve some estimates for the nonlinear terms, which could be treated in the same way as \eqref{eq3.21}-\eqref{eq3.24} \eqref{eq3.27} and \eqref{eq3.28}.

\noindent\text{(5).}
Finally, substituting $(R,v)$ by $(Q,u)$ in \eqref{eq:1.6} and \eqref{eq:1.7},
and adding them together to get
%
\begin{align}
	&\|u\|_{L^2}^2 + \|Q\|_{L^2}^2+ L \|\nabla Q\|_{L^2}^2\nonumber\\
	=& \|u_0\|_{L^2}^2 +\|Q_0\|_{L^2}^2 + L \|\nabla Q_0\|_{L^2}^2 -2\mu \int_{0}^{t} \|\nabla u\|_{L^2}^2 ds- 2a\Gamma  \int_{0}^{t} \|Q\|_{L^2}^2ds\nonumber\\
	&  - 2(a+1)\Gamma L\int_{0}^{t} \|\nabla Q\|_{L^2}^2ds - 2\Gamma L^{2} \int_0^t \|\Delta Q\|_{L^2}^2
	ds \nonumber\\
	& + \underbrace{2\int_{0}^{t} (u,u \cdot \nabla u) +(Q:u \cdot \nabla Q)ds }_{\mathcal{K}_1} + \underbrace{2L\int_0^t \left( \nabla Q \odot\nabla Q  :\nabla u\right) + (u \cdot \nabla Q:\Delta Q)ds }_{\mathcal{K}_2} \nonumber\\
	&  +\underbrace{2L\int_0^t \left( \Delta QQ- Q\Delta Q:\nabla u\right) + \left( Q\Omega-\Omega Q:\Delta Q\right) ds}_{\mathcal{K}_3} + \underbrace{ 2\int_0^t \left( \Omega Q-Q\Omega  :Q  \right) ds}_{\mathcal{K}_4}\nonumber\\
	&  +\underbrace{ \dfrac{4\xi}{3}\int_0^t ( D:Q - L \Delta Q)+ (L\Delta Q-aQ:\nabla u)ds}_{\mathcal{K}_5}\\
	& +\underbrace{2\xi\int_0^t\left( L\Delta QQ+  LQ\Delta Q -2aQ^2: \nabla u \right) + \left( DQ+QD :Q - L \Delta Q \right) ds }_{\mathcal{K}_6}   \nonumber\\
	& \underbrace{- 4\xi\int_0^t\left(  Q{\rm Tr}(QH): \nabla u \right) + \left(  Q {\rm Tr}(Q\nabla u):Q - L \Delta Q \right)ds }_{\mathcal{K}_7}\nonumber\\
	&  +2\Gamma\int_0^t \left( b\left[Q^2-\dfrac{{\rm Tr}(Q^2)}{3}\mathbb{I}_3\right]-cQ{\rm Tr}(Q^2) :Q - L \Delta Q \right)ds\nonumber\\
	&  +\dfrac{4\xi}{3}\int_0^t \left( b\left[Q^2-\dfrac{{\rm Tr}(Q^2)}{3}\mathbb{I}_3\right]-cQ{\rm Tr}(Q^2) :\nabla u \right)ds\nonumber\\
	&  +4\xi\int_0^t \left( Q\left\lbrace b\left[Q^2-\dfrac{{\rm Tr}(Q^2)}{3}\mathbb{I}_3\right]-cQ{\rm Tr}(Q^2)\right\rbrace :\nabla u \right)ds.\nonumber
\end{align}
	Obviously, it follows from Lemma~\ref{lem.cancel} that $\mathcal{K}_i=0$ ($1\leq i\leq 4$). In addition, it is easy to check that
	\begin{align}
		&\mathcal{K}_5= \dfrac{4(1-a)\xi}{3}\int_0^t ( D:Q)ds,\nonumber\\
		&\mathcal{K}_6= 4(1-a)\xi \int_0^t ( Q^2:\nabla u)ds ,\\
		&\mathcal{K}_7= 4(a-1)\xi\int_0^t \left(  Q {\rm Tr}(Q\nabla u):Q\right)ds- 4\xi\int_0^t \left( b\left[Q^2-\dfrac{{\rm Tr}(Q^2)}{3}\mathbb{I}_3\right]-cQ{\rm Tr}(Q^2) :Q {\rm Tr}(Q\nabla u) \right)ds.\nonumber
	\end{align}
    From all these facts, we can get the desired energy equality \eqref{energy_eq} by a straightforward computation.
    This completes the proof of Theorem~\ref{thm1.1}.

\section{Weak-strong uniqueness}\label{sec4}
In this section, we will show the weak-strong uniqueness for the Leray-Hopf type weak solution of system \eqref{eq1.1} for $\xi\ne 0$. In order to do so, we study a system in terms of the difference of $(Q,u)$ and $(R,v)$, which will come from the energy inequality \eqref{eq1.4} and equalities in Theorem~\ref{thm1.1}. Similarly to the proof of \cite[Theorem~1.5]{weak-strong-3d}, the key step is to achieve some estimates for the nonlinear terms, especially the new one in the case of $\xi\ne 0$.
Then, we can obtian the desired result by using Gr\"{o}nwall's inequality.

\noindent\textbf{Proof of Theorem~\ref{thm1.2}.}
To begin with, we will show the weak-strong uniqueness holds for the following condition  instead:
\begin{equation}\label{eq-4.1}
	\left(\Delta Q,\nabla u\right)\in L^q\left(0,T;L^p\right)\quad{\rm and}\quad \dfrac{2}{q}+\dfrac{3}{p}=\frac{3}{2},\,\,2\leq p \leq 6.
\end{equation}
Let us denote that
\[
G := R - Q,\quad \omega := v - u\quad{\rm and}\quad \widetilde{\Omega}:=\widebar{\Omega}-\Omega.
\]
Since a straightforward computation gives
	\begin{equation}\label{eq4.1}
		\begin{split}
			&\|\omega(t)\|_{L^2}^2+ \|G(t)\|_{L^2}^2+ L\|\nabla G(t)\|_{L^2}^2\\
			=& \|v(t)\|_{L^2}^2+ \|R(t)\|_{L^2}^2  +L\|\nabla R(t)\|_{L^2}^2 + \|u(t)\|_{L^2}^2+ \| Q(t)\|_{L^2}^2 +L\|\nabla Q(t)\|_{L^2}^2 \\
			&-2(v(t),u(t)) - 2(R(t),Q(t)) -2L(\nabla R(t),\nabla Q(t)).
		\end{split}
	\end{equation}
We then substitute \eqref{eq1.4}, \eqref{eq:1.6}, \eqref{eq:1.7}, \eqref{energy_eq} into \eqref{eq4.1} to have
	\begin{align}\label{eq4.2}
			&\|\omega(t)\|_{L^2}^2+ L\|G(t)\|_{L^2}^2+ \|\nabla G(t)\|_{L^2}^2\notag\\
			\leq & -2\mu \int_{0}^{t} \|\nabla \omega\|_{L^2}^2 ds -2a\Gamma  \int_{0}^{t} \|G\|_{L^2}^2ds -2(a+1)\Gamma L\int_{0}^{t} \|\nabla G\|_{L^2}^2ds -2\Gamma L^{2} \int_0^t \|\Delta G\|_{L^2}^2
			ds \notag\\
			& \underbrace{+2 \int_{0}^{t} (u,\omega \cdot \nabla v) ds}_{\mathcal{T}_1} \underbrace{-2L\int_0^t \left( \nabla R \odot\nabla R :\nabla u\right) + \left( \nabla Q \odot\nabla Q :\nabla v\right)ds}_{\mathcal{T}_2}\notag\\
			&\underbrace{-2 \int_{0}^{t} (Q:u \cdot \nabla R)+ (R:v \cdot \nabla Q) +L (v \cdot \nabla R:\Delta Q) +L (u \cdot \nabla Q:\Delta R) ds}_{\mathcal{T}_3} \notag\\
			& \underbrace{+2 \Gamma\int_0^t \left( \mathcal{M}[R]- \mathcal{M}[Q] :G - L \Delta G \right)ds + 4(1-a)\xi \int_0^t  \left( RG+GQ:\nabla \omega\right) ds}_{\mathcal{T}_4}\notag\\
			& \underbrace{+2L\int_0^t \left(  R\Delta R-\Delta RR :\nabla u\right) + \left( Q\Delta Q-\Delta QQ :\nabla v\right)ds}_{\mathcal{T}_5} \underbrace{+\dfrac{4\xi}{3}(1-a)\int_0^t ( G:\nabla \omega) ds}_{\mathcal{T}_6} \\
			&\underbrace{-2 \int_0^t \left(  \Omega Q-Q\Omega :R - L \Delta R\right) + \left( \widebar{\Omega} R-R\widebar{\Omega}:Q - L \Delta Q \right) ds}_{\mathcal{T}_7}\notag\\
			&\underbrace{-2 \xi\int_0^t \left(  DQ+QD -2RD  :R - L \Delta R\right)+\left( \widebar{D}R+R\widebar{D}- 2Q\widebar{D}  :Q - L \Delta Q \right) ds}_{\mathcal{T}_8}\notag\\
			& \underbrace{ +4\xi\int_0^t (G{\rm Tr}(R\nabla \omega) + Q{\rm Tr}(G\nabla \omega) :Q-L\Delta Q) - (Q{\rm Tr}(G\nabla u) + G{\rm Tr}(R\nabla u) :G-L\Delta G)  ds}_{\mathcal{T}_9}\notag\\
			&  \underbrace{+4\xi\int_0^t \left( \left[R + \dfrac{\mathbb{I}_3}{3}\right] \mathcal{M}[R]- \left[Q + \dfrac{\mathbb{I}_3}{3}\right] \mathcal{M}[Q] :\nabla \omega \right) ds}_{\mathcal{T}_{10}}\notag\\
			& \underbrace{+ 4(1-a)\xi \int_0^t \left(  Q {\rm Tr}(Q\nabla \omega):Q\right) -\left(  R {\rm Tr}(R\nabla \omega):R\right)ds}_{\mathcal{T}_{11}}\notag\\
			& \underbrace{+4\xi\int_0^t  \left(\mathcal{M}[Q]:Q {\rm Tr}(Q\nabla \omega) \right)- \left(\mathcal{M}[R]:R {\rm Tr}(R\nabla \omega) \right)ds}_{\mathcal{T}_{12}},\notag
	\end{align}
	 where
	\[
	\mathcal{M}[Q]:= b\left[Q^2-\dfrac{{\rm Tr}(Q^2)}{3}\mathbb{I}_3\right]-cQ{\rm Tr}(Q^2).
	\]
	It is easy to check that
	\[
	\int_0^t \left( u,\omega\cdot\nabla v \right)ds= \int_0^t \left( u,\omega\cdot\nabla \omega \right)ds.
	\]
	Then, similarly to Lemma~\ref{lem2.2}, $\mathcal{T}_1$ and $\mathcal{T}_6$ are estimated as
	\begin{equation}
		\begin{split}
			\left|\mathcal{T}_1+ \mathcal{T}_6\right|\leq &C\int_0^t\int_{\R^3}|u||\omega||\nabla\omega|+ |G||\nabla \omega| dxds\\
			\leq & C\int_0^t\|u\|_{L^{\frac{12p}{6+p}}}\|\omega\|_{L^{\frac{12p}{5p-6}}}\|\nabla \omega\|_{L^2} + \|G\|_{L^2} \|\nabla\omega\|_{L^2}ds\\
			\leq & \varepsilon\int_0^t\|\nabla\omega\|_{L^2}^2ds + C(\varepsilon)\int_0^t\|G\|_{L^2}^2ds + \int_0^t\|\nabla u\|_{L^p}^{\frac{1}{2}}\| u\|_{L^2}^{\frac{1}{2}} \|\omega\|_{L^2}^{\frac{3p-6}{4p}}\|\omega\|_{L^6}^{\frac{p+6}{4p}}\|\nabla\omega\|_{L^2}ds \\
			\leq& \varepsilon\int_0^t\|\nabla\omega\|_{L^2}^2ds + C(\varepsilon)\int_0^t\|G\|_{L^2}^2ds+ \int_0^t\|\nabla u\|_{L^p}^{\frac{1}{2}}\| u\|_{L^2}^{\frac{1}{2}} \|\omega\|_{L^2}^{\frac{3p-6}{4p}}\|\nabla\omega\|_{L^2}^{\frac{5p+6}{4p}}ds \\
			\leq& \varepsilon\int_0^t\|\nabla\omega\|_{L^2}^2ds + C(\varepsilon)\int_0^t\|G\|_{L^2}^2  +\left(\|\nabla u\|_{L^p}^{\frac{1}{2}}\| u\|_{L^2}^{\frac{1}{2}} \|\omega\|_{L^2}^{\frac{3p-6}{4p}} \right)^{\frac{8p}{3p-6}}ds \\
			\leq& \varepsilon\int_0^t\|\nabla\omega\|_{L^2}^2ds + C(\varepsilon) \int_0^t \underbrace{(1 +\| u\|_{L^2}^{\frac{4p}{3p-6}}\|\nabla u\|_{L^p}^{\frac{4p}{3p-6}} )}_{\mathcal{N}_1}\left(\|G\|_{L^2}^2+\|\omega\|_{L^2}^2\right)ds,
		\end{split}
	\end{equation}
	where we have used H\"{o}lder inequality, $L^p$ interpolation, Gagliardo-Nirenberg inequality and Young's inequality.
	
Again, by using the identities shown in Lemma~\ref{lemB4}, we also have
	\begin{align}
			\left|\mathcal{T}_2+ \mathcal{T}_3 \right|=& \left|2\int_0^t (\Delta G,u\cdot\nabla G) -(\Delta Q,\omega\cdot\nabla G)-(G:\omega\cdot\nabla Q)ds\right|\nonumber\\
			\leq &C\int_0^t\|u\|_{L^{\frac{12p}{p+6}}}\|\nabla G\|_{L^{\frac{12p}{5p-6}}}\|\Delta G\|_{L^2}+ \|\Delta Q\|_{L^p}\|\nabla G\|_{L^{\frac{12p}{7p-6}}}\|\omega\|_{L^{\frac{12p}{5p-6}}}+ \|G\|_{L^2}\|\omega\|_{L^6}\|\nabla Q\|_{L^3} ds\nonumber\\
			\leq &C\int_0^t  \|\nabla\omega\|_{L^2}\|G\|_{L^2}\|\nabla Q\|_{H^1} +\|\nabla u\|_{L^p}^{\frac{1}{2}} \| u\|_{L^2}^{\frac{1}{2}} \|\nabla G\|_{L^2}^{\frac{3p-6}{4p}}\|\nabla G\|_{L^6}^{\frac{p+6}{4p}}\|\Delta G\|_{L^2}ds\nonumber\\
			&+ C\int_0^t \|\Delta Q\|_{L^p}\|\nabla G\|_{L^2}^{\frac{5p-6}{4p}}\|\nabla G\|_{L^6}^{\frac{6-p}{4p}}\| \omega\|_{L^2}^{\frac{3p-6}{4p}} \| \omega\|_{L^6}^{\frac{p+6}{4p}} ds\nonumber\\
			\leq & \varepsilon\int_0^t\|\nabla\omega\|_{L^2}^2 ds + C(\varepsilon)\int_0^t\| Q\|_{H^2}^2\|G\|_{L^2}^2 ds + C\int_0^t \|\nabla u\|_{L^p}^{\frac{1}{2}} \| u\|_{L^2}^{\frac{1}{2}} \|\nabla G\|_{L^2}^{\frac{3p-6}{4p}} \|\Delta G\|_{L^2}^{\frac{5p+6}{4p}}ds\nonumber\\
			&+ C\int_0^t \|\Delta Q\|_{L^p}\|\nabla G\|_{L^2}^{\frac{3p-6}{4p}}\| \omega\|_{L^2}^{\frac{3p-6}{4p}}\|\nabla G\|_{H^1}^{\frac{6+p}{4p}} \|\nabla \omega\|_{L^2}^{\frac{p+6}{4p}} ds\nonumber\\
			\leq & \varepsilon\int_0^t\|\nabla\omega\|_{L^2}^2+\|\Delta G\|_{L^2}^2 ds + C(\varepsilon)\int_0^t\| Q\|_{H^2}^2\|G\|_{L^2}^2 +\|\nabla u\|_{L^p}^{\frac{4p}{3p-6}} \| u\|_{L^2}^{\frac{4p}{3p-6}} \|\nabla G\|_{L^2}^2 ds \\
			& +\varepsilon\int_0^t\|\nabla\omega\|_{L^2}\|\nabla G\|_{H^1} ds + C(\varepsilon)\int_0^t \|\Delta Q\|_{L^p}^{\frac{4p}{3p-6}}\|\nabla G\|_{L^2} \| \omega\|_{L^2} ds\nonumber\\
			\leq & \varepsilon\int_0^t\|\nabla\omega\|_{L^2}^2+\|\nabla G\|_{H^1}^2 ds + C(\varepsilon)\int_0^t \left(\| Q\|_{H^2}^2 +\|\nabla u\|_{L^p}^{\frac{4p}{3p-6}} \| u\|_{L^2}^{\frac{4p}{3p-6}}\right) \|G\|_{H^1}^2 ds\nonumber\\
			& + C(\varepsilon)\int_0^t \|\Delta Q\|_{L^p}^{\frac{4p}{3p-6}}\left(\|\nabla G\|_{L^2}^2+ \| \omega\|_{L^2}^2 \right) ds\nonumber\\
			\leq & \varepsilon\int_0^t\|\nabla\omega\|_{L^2}^2 + \|\Delta G\|_{L^2}^2 ds \nonumber\\
			&+ C(\epsilon)\int_0^t \underbrace{\left(1+\| Q\|_{H^2}^2  +\| u\|_{L^2}^{\frac{4p}{3p-6}}\|\nabla u\|_{L^p}^{\frac{4p}{3p-6}}   +\|\Delta Q\|_{L^p}^{\frac{4p}{3p-6}}\right)}_{\mathcal{N}_2} \left(\| G\|_{H^1}^2+ \| \omega\|_{L^2}^2 \right) ds.\nonumber
	\end{align}
	Similarly, the remainder terms in the right hand side of \eqref{eq4.2} are estimated as follows:
	\begin{align}\label{eq4.5}
		\left|\mathcal{T}_5+ \mathcal{T}_7\right| =& \left|2\int_0^t\left(\widetilde{\Omega} Q-Q\widetilde{\Omega}: G\right) +  (G\Delta G-\Delta GG,\nabla u) - (G\Delta Q-\Delta QG,\nabla \omega)ds\right|\nonumber\\
		\leq &C\int_0^t \| Q\|_{L^3}\|\nabla \omega\|_{L^2}\|G\|_{L^6}+ \|\nabla u\|_{L^p}\| G\|_{L^{\frac{2p}{p-2}}}\|\Delta G\|_{L^2}+ \|\Delta Q\|_{L^p}\| G\|_{L^{\frac{2p}{p-2}}}\|\nabla \omega\|_{L^2}ds \nonumber\\
		\leq & \varepsilon\int_0^t\|\nabla\omega\|_{L^2}^2 ds + C(\varepsilon)\int_0^t\| Q\|_{H^1}^2\|\nabla G\|_{L^2}^2 ds\nonumber\\
		& +C\int_0^t \left(\|\nabla u\|_{L^p}+\|\Delta Q\|_{L^p}\right)\| G\|_{L^{\frac{4p}{p-2}}}^{\frac{12-2p}{6+p}} \| G\|_{L^3}^{\frac{3p-6}{6+p}} \left(\|\Delta G\|_{L^2}+\|\nabla \omega\|_{L^2}\right)ds\nonumber\\
		\leq & \varepsilon\int_0^t\|\nabla\omega\|_{L^2}^2+ \|\Delta G\|_{L^2}^2 ds + C(\varepsilon)\int_0^t\| Q\|_{H^1}^2\|\nabla G\|_{L^2}^2 ds \nonumber\\
		& + C(\varepsilon)\int_0^t \left(\|\nabla u\|_{L^p}^2+\|\Delta Q\|_{L^p}^2\right)\| G\|_{W^{1,\frac{12p}{7p-6}}}^{\frac{24-4p}{6+p}} \| G\|_{H^1}^{\frac{6p-12}{6+p}} ds\nonumber\\
		\leq & \varepsilon\int_0^t\|\nabla\omega\|_{L^2}^2+ \|\Delta G\|_{L^2}^2 ds + C(\varepsilon)\int_0^t\| Q\|_{H^1}^2\|\nabla G\|_{L^2}^2 ds \nonumber\\
		& + C(\varepsilon)\int_0^t \left(\|\nabla u\|_{L^p}^2+\|\Delta Q\|_{L^p}^2\right)\| G\|_{H^1}^{\frac{(5p-6)(6-p)}{p(6+p)}}\| G\|_{W^{1,6}}^{\frac{(6-p)^2}{p(6+p)}} \| G\|_{H^1}^{\frac{6p-12}{6+p}}ds \\
		\leq & \varepsilon\int_0^t\|\nabla\omega\|_{L^2}^2+ \|\Delta G\|_{L^2}^2 ds + C(\varepsilon)\int_0^t\| Q\|_{H^1}^2\|\nabla G\|_{L^2}^2 ds \nonumber\\
		& + C(\varepsilon)\int_0^t \left(\|\nabla u\|_{L^p}^2+\|\Delta Q\|_{L^p}^2\right)\| G\|_{H^1}^{\frac{3p-6}{p}}\| G\|_{H^2}^{\frac{6-p}{p}} ds\nonumber\\
		\leq & \varepsilon\int_0^t\|\nabla\omega\|_{L^2}^2+ \|\Delta G\|_{L^2}^2 +\|G\|_{H^2}^2ds + C(\varepsilon)\int_0^t \left(\| Q\|_{H^1}^2 +\|\nabla u\|_{L^p}^{\frac{4p}{3p-6}}+\|\Delta Q\|_{L^p}^{\frac{4p}{3p-6}}\right)\| G\|_{H^1}^2 ds\nonumber\\
		\leq & \varepsilon\int_0^t\|\nabla\omega\|_{L^2}^2 + \|\Delta G\|_{L^2}^2 ds + C(\varepsilon)\int_0^t \underbrace{\left(1+ \| Q\|_{H^2}^2  + \|\nabla u\|_{L^p}^{\frac{4p}{3p-6}}  +\|\Delta Q\|_{L^p}^{\frac{4p}{3p-6}}\right)}_{\mathcal{N}_3} \| G\|_{H^1}^2 ds;\nonumber
	\end{align}
	\begin{equation}
		\begin{split}
			\left|\mathcal{T}_8\right| 
			\leq &C\int_0^t\int_{\R^3}|G||\nabla u||\Delta G| + |G|^2|\nabla u| + |G||\Delta Q||\nabla \omega| +|Q||G||\nabla\omega| dxds\\
			\lesssim & \int_0^t\underbrace{\left(\|\nabla u\|_{L^p}\|\Delta G\|_{L^2}+ \|\Delta Q\|_{L^p}\|\nabla \omega\|_{L^2}\right)\| G\|_{L^{\frac{2p}{p-2}}}+ \|\nabla \omega\|_{L^2}\|Q\|_{L^3}\|G\|_{L^6}}_{\text{The same as the ones in \eqref{eq4.5} }}+ \|\nabla u\|_{L^2}\|G\|_{L^3}\|G\|_{L^6}   ds\\
			\leq & \varepsilon\int_0^t\|\nabla\omega\|_{L^2}^2 + \|\Delta G\|_{L^2}^2 ds + C(\varepsilon)\int_0^t \underbrace{\left(1+ \|u\|_{H^1}^2 + \|Q\|_{H^2}^2 + \|\nabla u\|_{L^p}^{\frac{4p}{3p-6}}  +\|\Delta Q\|_{L^p}^{\frac{4p}{3p-6}}\right)}_{\mathcal{N}_4} \| G\|_{H^1}^2 ds;
		\end{split}
	\end{equation}
	\begin{align}
		|\mathcal{T}_9|\lesssim &  \int_0^t  \|G\|_{L^\frac{4p}{p-2}}\left( \|R\|_{L^\frac{4p}{p-2}}+  \|Q\|_{L^\frac{4p}{p-2}}\right)\|\nabla \omega\|_{L^2} \|\Delta Q\|_{L^p} + \|G\|_{L^6}\|\nabla\omega\|_{L^2}\left(\|R\|_{L^6}+ \|Q\|_{L^6}\right) \|Q\|_{L^6}ds \nonumber\\
		&+\int_0^t  \|G\|_{L^\frac{4p}{p-2}}\left( \|R\|_{L^\frac{4p}{p-2}}+  \|Q\|_{L^\frac{4p}{p-2}}\right)\|\nabla u\|_{L^p} \|G-L\Delta G\|_{L^2}  ds \nonumber\\
		\leq & \varepsilon\int_0^t\|\nabla\omega\|_{L^2}^2 +  \|\Delta G\|_{L^2}^2 ds + C(\varepsilon)\int_0^t \left(1+ \|Q\|_{H^1}^4 + \|R\|_{H^1}^4 \right) \| G\|_{H^1}^2 ds \nonumber\\
		&+C(\varepsilon) \int_0^t  \|G\|_{W^{1,\frac{12p}{7p-6}}}^2\left( \|R\|_{W^{1,\frac{12p}{7p-6}}}^2+  \|Q\|_{W^{1,\frac{12p}{7p-6}}}^2\right) \left(\|\nabla u\|_{L^p}^2 +\|\Delta Q\|_{L^p}^2\right) ds\nonumber\\
		\leq & \varepsilon\int_0^t\|\nabla\omega\|_{L^2}^2 +  \|\Delta G\|_{L^2}^2 ds + C(\varepsilon)\int_0^t \left(1+ \|Q\|_{H^1}^4 + \|R\|_{H^1}^4 \right) \| G\|_{H^1}^2 ds \nonumber\\
		&+ C(\varepsilon)\int_0^t  \|G\|_{H^1}^{\frac{5p-6}{2p}}\|G\|_{W^{1,6}}^{\frac{6-p}{2p}}\left( \|R\|_{H^1}^{\frac{5p-6}{2p}}\|R\|_{W^{1,6}}^{\frac{6-p}{2p}}+  \|Q\|_{H^1}^{\frac{5p-6}{2p}}\|Q\|_{W^{1,6}}^{\frac{6-p}{2p}}\right) \left(\|\nabla u\|_{L^p}^2 +\|\Delta Q\|_{L^p}^2\right)ds\nonumber\\
		\leq & \varepsilon\int_0^t\|\nabla\omega\|_{L^2}^2 +  \|\Delta G\|_{L^2}^2 ds + C(\varepsilon)\int_0^t \left(1+ \|Q\|_{H^1}^4 + \|R\|_{H^1}^4 \right) \| G\|_{H^1}^2 ds \\
		&+ C(\varepsilon)\int_0^t \left[\left( \|R\|_{H^1}^2(\|R\|_{H^2}^2)^{\frac{6-p}{5p-6}}+  \|Q\|_{H^1}^2(\|Q\|_{H^2}^2)^{\frac{6-p}{5p-6}}\right) \left(\|\nabla u\|_{L^p}^{\frac{8p}{5p-6}} +\|\Delta Q\|_{L^p}^{\frac{8p}{5p-6}}\right)  \right]  \|G\|_{H^1}^2ds\nonumber\\
		\leq & \varepsilon\int_0^t\|\nabla\omega\|_{L^2}^2 +  \|\Delta G\|_{L^2}^2 ds + C(\varepsilon)\int_0^t \underbrace{\left(1+ \|Q\|_{H^1}^4 + \|R\|_{H^1}^4 \right)}_{\mathcal{N}_5} \| G\|_{H^1}^2 ds \nonumber\\
		&+ C(\varepsilon)\int_0^t \underbrace{\left( \|R\|_{H^1}^{\frac{10p-12}{6-p}}\|R\|_{H^2}^2+   \|Q\|_{H^1}^{\frac{10p-12}{6-p}}\|Q\|_{H^2}^2 + \|\nabla u\|_{L^p}^{\frac{4p}{3p-6}} +\|\Delta Q\|_{L^p}^{\frac{4p}{3p-6}}\right) }_{\mathcal{N}_6} \|G\|_{H^1}^2ds;\nonumber
	\end{align}
	\begin{align}
			|\mathcal{T}_4+ \mathcal{T}_{10}|\lesssim & \int_0^t\int_{\R^3} (|R||G| + |Q||G|)( |\nabla\omega|+ |G| +|\Delta G|) dxds\nonumber\\
			&+  \int_0^t\int_{\R^3} (|R|^2|G| + |R||Q||G|+|G||Q|^2)(|\nabla\omega|+ |G| +|\Delta G|) dxds\nonumber\\
			& + \int_0^t\int_{\R^3} (|Q|^3|G| + |R|^2|Q||G|+|R||G||Q|^2+ |R|^3|G|)|\nabla\omega| dxds\nonumber\\
			\leq & \varepsilon\int_0^t\|\nabla \omega\|_{L^2}^2 + \|G\|_{L^2}^2 +\|\Delta G \|_{L^2}^2 ds +C(\varepsilon)\int_0^t (\|R\|_{L^4}^2 + \|Q\|_{L^4}^2) \|G\|_{L^4}^2 ds \nonumber\\
			& +C(\varepsilon)\int_0^t (\|R\|_{L^6}^4 + \|R\|_{L^6}^2\|Q\|_{L^6}^2 + \|Q\|_{L^6}^4) \|G\|_{L^6}^2 ds\nonumber\\
			& +C(\varepsilon)\int_0^t (\|R\|_{L^9}^6 + \|R\|_{L^9}^4\|Q\|_{L^9}^2 + \|R\|_{L^9}^2\|Q\|_{L^9}^4 + \|Q\|_{L^9}^6) \|G\|_{L^6}^2 ds\nonumber\\
			\leq & \varepsilon\int_0^t\|\nabla \omega\|_{L^2}^2 +\|\Delta G \|_{L^2}^2 ds + C(\varepsilon)\int_0^t \left(1+ \|Q\|_{H^1}^2 + \|R\|_{H^1}^2\right) \|G\|_{H^1}^2 ds \\
			& +C(\varepsilon)\int_0^t (\|\nabla R\|_{L^2}^4 +  \|\nabla Q\|_{L^2}^4 + \|Q\|_{L^9}^6 + \|R\|_{L^9}^6) \|\nabla G\|_{L^2}^2 ds\nonumber\\
			\leq & \varepsilon\int_0^t\|\nabla \omega\|_{L^2}^2 +\|\Delta G \|_{L^2}^2 ds + C(\varepsilon)\int_0^t \left(1+ \|Q\|_{H^1}^2 + \|R\|_{H^1}^2\right) \|G\|_{H^1}^2 ds \nonumber\\
			& +C(\varepsilon)\int_0^t \left( \|Q\|_{H^1}^2 + \|R\|_{H^1}^2\right)(\|Q\|_{H^2}^2 + \|R\|_{H^2}^2) \| G\|_{H^1}^2 ds\nonumber\\
			& +C(\varepsilon)\int_0^t (\|D^2 Q\|_{L^2}\| \nabla Q\|_{L^2}^5+ \|D^2 R\|_{L^2}\| \nabla R\|_{L^2}^5) \|\nabla G\|_{L^2}^2 ds\nonumber\\
			\leq & \varepsilon\int_0^t\|\nabla \omega\|_{L^2}^2 + \|\Delta G \|_{L^2}^2 ds +C(\varepsilon)\int_0^t \underbrace{\left(\sum_{l=0}^2\|Q\|_{H^1}^{2l} + \|R\|_{H^1}^{2l}\right)\left(1+ \|Q\|_{H^2}^2 + \|R\|_{H^2}^2\right)}_{\mathcal{N}_7} \| G\|_{H^1}^2 ds;\nonumber
	\end{align}
	\begin{align}
		|\mathcal{T}_{11}+ \mathcal{T}_{12}|\leq & \int_0^t\int_{\R^3}(|R|^3+ |R|^4+ |R|^2 |Q|+ |R|^3|Q|) |\nabla\omega||G| dxds\nonumber\\
		& + \int_0^t\int_{\R^3} (|Q|^3|G| + |Q|^2|R||G| + |R||Q|^3|G|+|R|^2|G||Q|^2+ |Q|^4|G|)|\nabla\omega| dxds\nonumber\\
		\leq & \varepsilon\int_0^t\|\nabla\omega\|_{L^2}^2  ds + C(\varepsilon)\int_0^t (\| R\|_{L^9}^6 +  \|R\|_{L^{12}}^8 + \|R\|_{L^9}^4\|Q\|_{L^9}^2 + \|R\|_{L^{12}}^6\|Q\|_{L^{12}}^2) \| G\|_{L^6}^2 ds\nonumber\\
		&+ C(\varepsilon)\int_0^t (\| Q\|_{L^9}^6  + \|R\|_{L^9}^2\|Q\|_{L^9}^4 +  \|R\|_{L^{12}}^2\|Q\|_{L^{12}}^6+ \|R\|_{L^{12}}^4\|Q\|_{L^{12}}^4 +  \|Q\|_{L^{12}}^8 ) \| G\|_{L^6}^2 ds\nonumber\\
		\leq & \varepsilon\int_0^t\|\nabla\omega\|_{L^2}^2  ds + C(\varepsilon)\int_0^t (\| R\|_{L^9}^6 + \| Q\|_{L^9}^6 +  \|R\|_{L^{12}}^8 + \|Q\|_{L^{12}}^8) \|\nabla G\|_{L^2}^2 ds\\
		\leq & \varepsilon\int_0^t\|\nabla\omega\|_{L^2}^2  ds + C(\varepsilon)\int_0^t (\| D^2 R\|_{L^2} \| \nabla R\|_{L^2}^5 + \| D^2 Q\|_{L^2} \| \nabla Q\|_{L^2}^5  ) \|\nabla G\|_{L^2}^2 ds\nonumber\\
		& + C(\varepsilon)\int_0^t ( \| D^2 R\|_{L^2}^2 \| \nabla R\|_{L^2}^6 + \| D^2 Q\|_{L^2}^2 \| \nabla Q\|_{L^2}^6 ) \|\nabla G\|_{L^2}^2 ds\nonumber\\
		\leq & \varepsilon\int_0^t\|\nabla \omega\|_{L^2}^2  ds +C(\varepsilon)\int_0^t \underbrace{( \|Q\|_{H^1}^4 + \|R\|_{H^1}^4 + \|Q\|_{H^1}^6 + \|R\|_{H^1}^6)(\|Q\|_{H^2}^2 + \|R\|_{H^2}^2)}_{\mathcal{N}_8} \| G\|_{H^1}^2 ds.\nonumber
	\end{align}
	Now, we introduce the following functional for $(G,\omega)$:
	\[
	\mathcal{Q}(t)=\|\omega(t)\|_{L^{2}}^{2}+\| G(t) \|_{L^{2}}^{2}+L \| \nabla G(t) \|_{L^{2}}^{2}.
	\]
	From all the above estimates,
	we can choose a sufficiently small number $\varepsilon$ and conclude that
	\begin{equation}
		\mathcal{Q}(t) \lesssim \int_{0}^{t} \mathcal{A}(s) \mathcal{Q}(s) ds,
	\end{equation}
	where $\mathcal{A}:= \sum_{i=1}^8 \mathcal{N}_i$.
	According to the regularity assumptions of solutions, one can see that $\mathcal{A}$ is integrable on $[0,T]$. Then, it follows from the Gr\"{o}nwall's inequality that $\mathcal{Q}(t) \equiv 0$ for all $t\in [0,T]$, and hence $(Q,u) \equiv (R,v)$. This implies that weak-strong uniqueness holds for the condition \eqref{eq-4.1}.
	
	Finally, we demonstrate the desired result in Theorem~\ref{thm1.2} by interpolation. Indeed, one should note that $(\Delta Q,\nabla u)\in L_t^2(L^2)$. Hence, if $(\Delta Q,\nabla u)\in L_t^q(L^p)$ with
	\begin{equation}\label{eq4.12}
		\frac{2}{q}+\frac{3}{p}<\frac{3}{2}\quad {\rm and}\quad 2\leq p,q \leq +\infty,
	\end{equation}
	we can find some $\theta\in [0,\frac{3}{5}]$ and apply the following inequality
	\[
	\|f\|_{L_t^r(L^s)}\leq \|f\|_{L_t^2(L^2)}^\theta\|f\|_{L_t^q(L^p)}^{1-\theta},\,\,\dfrac{1}{r}=\dfrac{\theta}{2} + \dfrac{1-\theta}{q}, \,\,\dfrac{1}{s}=\dfrac{\theta}{2} + \dfrac{1-\theta}{p},
	\]
	to conclude that $(\Delta Q,\nabla u)\in L_t^r(L^s)$ with $\dfrac{2}{r}+\dfrac{3}{s}=\dfrac{3}{2}$ and $s\in [2,6]$. Then, by applying the previous result for \eqref{eq-4.1}, one can see that weak-strong uniqueness also holds for the condition \eqref{eq4.12}, and this completes the proof of Theorem~\ref{thm1.2}.

\section{Global regularity of strong solution}\label{sec5}
The aim of this section is to investigate the global well-posedness of system \eqref{eq1.1} for small initial data with Sobolev regularity. To do this, we first need to show the existence and uniqueness of local-in-time solution to \eqref{eq1.1} for the case $\xi\in\R$. Then, together with the a priori estimate of the solution, global regularity will follows in a standard continuity argument.

Motivated by the work in \cite{ALC-decay},
where local well-posedness result for the model of active liquid crystals was established for large initial data belonging to $H^{s+1}\times H^s$ with $s\geq 2$,
we similarly construct the following approximation system to show the local existence and uniqueness result (because the structure of system \eqref{eq1.1} is similar with that of incompressible active liquid crystals, even though for the case of arbitrary $\xi\in\R$):
\begin{equation}\label{eq-ap}
	\begin{cases}
		\partial_t Q^{n+1}+ u^n\cdot\nabla Q^{n+1} + Q^n\Omega^{n+1} -\Omega^{n+1}Q^n=\Gamma H^{n+1} + \xi D^{n+1}\left( Q^n+\dfrac{1}{3}\mathbb{I}_3 \right) \\
		\hspace{1.2cm}+ \xi \left( Q^n+\dfrac{1}{3}\mathbb{I}_3 \right)D^{n+1}- 2\xi \left( Q^n+\dfrac{1}{3}\mathbb{I}_3 \right){\rm Tr}(Q^n\nabla u^{n+1}),\\
		\partial_t u^{n+1} + u^n \cdot\nabla u^{n+1}-\mu \Delta u^{n+1} =L\mathbb{P}\nabla\cdot \left(Q^n\Delta Q^{n+1}-\Delta Q^{n+1}Q^n - \nabla Q^n \odot\nabla Q^n\right)\\
		\hspace{1.2cm}-\xi \mathbb{P}\nabla\cdot \left[ \left( Q^n+\dfrac{1}{3}\mathbb{I}_3 \right)H^{n+1}+  H^{n+1}\left( Q^n+\dfrac{1}{3}\mathbb{I}_3 \right) - 2  \left( Q^n+\dfrac{1}{3}\mathbb{I}_3 \right){\rm Tr}(Q^nH^{n+1}) \right],\\
		\nabla\cdot u^{n+1} ={\rm Tr}\,Q^{n+1}= 0,
	\end{cases}
\end{equation}
with the initial data
\begin{equation}\label{eq:2.2}
	(Q^{n+1},u^{n+1})(t,x)\mid_{t=0}=(Q_0,u_0)(x).
\end{equation}
Here
\[
\begin{split}
D^{n+1}:=&\dfrac{\nabla u^{n+1} + (\nabla u^{n+1})^{T}}{2},\quad \Omega^{n+1}:=\dfrac{\nabla u^{n+1} - (\nabla u^{n+1})^{T}}{2},\\
H^{n+1}:= &L\Delta Q^{n+1}-aQ^{n+1}+M^{n+1},\quad
M^{n+1}:= b\left[(Q^n)^2-\dfrac{{\rm Tr}((Q^n)^2)}{3}\mathbb{I}_3\right]-cQ^{n+1}{\rm Tr}((Q^n)^2).
\end{split}
\]
As a result, we have the following local well-posedness result for system \eqref{eq1.1} with any large initial data in the Sobolev framework (see Appendix~\ref{ap1} for details).
\begin{theorem}[Local existence]\label{thm5.1}
	Let $\xi,a\in\R$, $s\geq 2$ be an integer and $E_0:=\|Q_0\|_{H^{s+1}}^2 + \|u_0\|_{H^s}^2$ be the initial energy. 
	Then there exists a time $T$ with
	\begin{equation}
		0<T\leq T^\ast:=\dfrac{\min\{E_0, \ln 2\}}{\sum_{l=0}^4 4^{l}\widetilde{C}E_0^{l}}
	\end{equation}
	such that the Cauchy problem of system \eqref{eq1.1} with initial data $(Q_0,u_0)\in H^{s+1}\times H^s$ admits a unique local solution
	\begin{equation}
		Q\in L^\infty\left( 0,T;H^{s+1} \right)\cap L^2\left( 0,T;H^{s+2} \right)\quad\text{and}\quad u\in L^\infty\left( 0,T;H^s \right)\cap L^2\left( 0,T;H^{s+1} \right),
	\end{equation}
	where $\widetilde{C}$ is a positive constant defined as in Appendix~\ref{ap1}.
	Moreover, $(Q,u)$ satisfies the following energy inequality
	\[
		\|u(t)\|_{H^s}^2 + \|Q(t)\|_{H^{s+1}}^2 
		+\int_0^t  \|\nabla u(\tau)\|_{H^{s}}^2 + \|Q(\tau)\|_{H^{s+1}}^2 +  \|\Delta Q(\tau)\|_{H^{s}}^2 d\tau \leq C_0,
	\]
	for all $t\in [0,T]$, where $C_0$ is a positive constant depending only on $E_0$, $s$, $|a|$, $b$, $c$, $\mu$, $|\xi|$, $L$ and $\Gamma$.
\end{theorem}
\begin{remark}
	In contrast to the work \cite{active-limit},
	we utilize a refined commutator estimates to explore the a priori estimate for the approximation solutions in $H^s$-framework, which is essential to reduce the assumption on Sobolev index $s$.
	Moreover, our result holds for all $a\in\R$.
\end{remark}

Based on the above result, the main goal of this section is to find some a priori estimate for the solutions to system \eqref{eq1.1}.
Then, the global well-posedness for small initial data can be justified by the standard continuation argument. In order to do that, one may suggest using the following energy functionals in a similar way to that of \cite{active-limit,ALC-decay}:
\[
\begin{split}
&\widetilde{\mathcal{E}}(t):= \|u(t)\|_{H^s}^2 + M\|Q(t)\|_{H^{s}}^2 +L\|\nabla Q(t)\|_{H^{s}}^2,\\
&\widetilde{\mathcal{D}}(t):= \mu\|\nabla u(t)\|_{H^{s}}^2 + aM\Gamma\|  Q(t)\|_{H^{s}}^2 + (a+M)\Gamma L\|\nabla Q(t)\|_{H^{s}}^2 + \Gamma L^2\|\Delta
Q(t)\|_{H^{s}}^2.
\end{split}
\]
Indeed, in \cite{ALC-decay} we have to deal with a quadratic term of the following type:
\begin{equation}\label{eq5.5}
	\left| \kappa\left( \partial^k Q:\partial^k\nabla u\right) \right| \leq \dfrac{\mu}{2}\|\nabla u\|_{H^s}^2 + \dfrac{2\kappa^2}{\mu} \|Q\|_{H^s}^2.
\end{equation}
Clearly, one can choose a sufficient large number $M$ such that these two terms on the right hand side of the above inequality can be absorbed by the dissipative energy functional $\widetilde{\mathcal{D}}(t)$. However, when $\xi\ne 0$, both of the equations of $(Q,u)$ contain linear terms.
This makes the analysis quite delicate.
Similarly, we will have the following quadratic term alike \eqref{eq5.5}:
\[
\dfrac{2\xi(M-a)}{3}\left( \partial^k Q:\partial^k\nabla u\right).
\]
Since the coefficient of this term also depends on $M$, the above strategy does not apply. Instead, we will choose $M=a$ such this term completely cancel out, which  differs from our previous study. Then, one can refine the argument from \cite{ALC-decay} and carry out a delicate analysis for the global regularity.
Nevertheless, we use the following energy functionals for the case of $\xi\ne0$:
\[
\begin{split}
	&\mathcal{E}(t)= \|u(t)\|_{H^s}^2 + a\|Q(t)\|_{H^{s}}^2 +L\|\nabla Q(t)\|_{H^{s}}^2,\\
	&\mathcal{D}(t)= \mu\|\nabla u(t)\|_{H^{s}}^2 + a^2\Gamma\|  Q(t)\|_{H^{s}}^2 + 2a\Gamma L\|\nabla Q(t)\|_{H^{s}}^2 + \Gamma L^2\|\Delta
	Q(t)\|_{H^{s}}^2.
\end{split}
\]

Now, we turn to present the a priori estimate for system \eqref{eq1.1} below.
\begin{prop}\label{prop3.1}
	Let $\xi\in\R$, $a>0$ and $s\geq 2$ be an integer. Then the local solution $(Q,u)$ established in Theorem~\ref{thm5.1} satisfies
	\begin{equation}\label{es-global}
		\dfrac{d}{dt}\mathcal{E}(t)+ \mathcal{D}(t)
		\leq C \sum_{m=1}^4\mathcal{E}^{\frac{m}{2}}(t) \mathcal{D}(t),\quad \forall t\in [0,T],
	\end{equation}
	where $C$ is some positive constant depending only on $s$, $|a|$, $b$, $c$, $\mu$, $|\xi|$, $L$ and $\Gamma$.
\end{prop}
\begin{proof}
Similarly to Proposition~\ref{prop-a1}, the proof consists of two steps.
\begin{enumerate}[(1).]
	\item The first step is to derive the basic $L^2$-energy estimate.
	Thus we take the summation of \eqref{eq1.1}$_1$ multiplied by $aQ-L\Delta Q$ and \eqref{eq1.1}$_2$ multiplied by $u$, take the trace, and then integrate by parts over $\R^3$ to get
	\begin{align}\label{eq5.7}
		\dfrac{1}{2}\dfrac{d}{dt}\Big( \|u\|_{L^2}^2 + &a\|Q\|_{L^2}^2 + L\|\nabla Q\|_{L^2}^2 \Big)
		+\mu\|\nabla u\|_{L^2}^2\nonumber\\
		&  + a^2\Gamma\| Q(t)\|_{L^2}^2 + 2a\Gamma L\|\nabla Q(t)\|_{L^2}^2 + \Gamma L^2\|\Delta
		Q(t)\|_{L^2}^2\nonumber\\
		=& \Gamma  \left(
		b\left[Q^2-\dfrac{{\rm Tr}(Q^2)}{3}\mathbb{I}_3\right]-cQ{\rm Tr}(Q^2) :aQ -L\Delta Q  \right)\nonumber\\
		&+ \dfrac{2\xi}{3}  \left(  b\left[Q^2-\dfrac{{\rm Tr}(Q^2)}{3}\mathbb{I}_3\right]-cQ{\rm Tr}(Q^2) :\nabla u  \right)\nonumber\\
		&- 2\xi  \left(  Q{\rm Tr}(Q\nabla u):  b\left[Q^2-\dfrac{{\rm Tr}(Q^2)}{3}\mathbb{I}_3\right]-cQ{\rm Tr}(Q^2)  \right)\\
		&+ 2\xi\left( Q \left\lbrace b\left[Q^2-\dfrac{{\rm Tr}(Q^2)}{3}\mathbb{I}_3\right]-cQ{\rm Tr}(Q^2) \right\rbrace:\nabla u  \right)\nonumber\\
		\lesssim& \left(\|Q\|_{L^2}+\|\Delta Q\|_{L^2} + \|\nabla u\|_{L^2}\right)\sum_{m=1}^4 \|Q\|_{L^\infty}^m\|Q\|_{L^2}\nonumber\\
		\lesssim& \sum_{m=1}^4 \mathcal{E}^{\frac{m}{2}} (t) \mathcal{D}(t).\nonumber
	\end{align}
	Note that we have used Lemma~\ref{lem.cancel} and the cyclic property of the trace to cancel some nonlinear terms.
	\item  For the higher-order energy estimates, let us apply $\partial^k$ ($1\leq |k|\leq s$) to system \eqref{eq1.1}, and sum up the second equation of resulting system multiplied by $\partial^ku $ and the first equation of resulting system multiplied by $a\partial^kQ -L\partial^k\Delta Q $, take the trace, and integrate by parts over $\R^3$ to have
	\begin{align*}
		&\dfrac{1}{2}\dfrac{d}{dt}\left( \|\partial^ku\|_{L^2}^2 +  a\|\partial^kQ\|_{L^2}^2 +  L\|\partial^k\nabla Q\|_{L^2}^2 \right)
		+\mu\|\partial^{k}\nabla u\|_{L^2}^2 \\
		&+ a^2\Gamma \|\partial^{k} Q\|_{L^2}^2+ 2aL\Gamma\|\partial^{k}\nabla Q\|_{L^2}^2 + \Gamma L^2\|\partial^k\Delta
		Q\|_{L^2}^2\\
		=& \left( [\partial^k, u]_{\rm div}Q :L\partial^k\Delta
		Q-a\partial^kQ  \right)  +L\left(u\cdot\nabla \partial^kQ:\partial^k\Delta
		Q  \right) -\left([\partial^k, u]_{\rm div}u:\partial^k u \right)\\
		& +a\left( \partial^k\Omega Q -  Q\partial^k\Omega:\partial^k Q\right)  +\left( [\partial^k,( \Omega, Q)]_{-}:a\partial^k Q-L\partial^k\Delta
		Q \right)   + L \left(  \partial^k\left( \nabla Q \odot\nabla Q\right) :\nabla\partial^{k} u  \right)\\
		& + a\xi\left( \partial^k DQ+Q\partial^k D:\partial^k Q  \right)  + L\left( [\partial^k,( \Delta Q, Q)]_{-}:\partial^{k}\nabla u   \right)  +\Gamma  \left(  \partial^k\mathcal{M}[Q]:a\partial^kQ -L\partial^k\Delta Q  \right)\\
		& + \dfrac{2\xi}{3}\left( \partial^k  \mathcal{M}[Q]:\partial^k\nabla u \right) -2 a\xi\left(  Q  {\rm Tr}(Q\partial^k\nabla u):\partial^k Q \right)-2 \xi\left( \partial^k  \left\lbrace Q{\rm Tr}\left(Q\mathcal{M}[Q]- aQ^2\right)\right\rbrace: \partial^k\nabla u \right)\\
		&+ \xi\left( \partial^k  \left\lbrace Q\mathcal{M}[Q] + \mathcal{M}[Q]Q-2aQ^2 \right\rbrace:\partial^k\nabla u \right) + \xi\left( [\partial^k,( D, Q)]_{+}:a\partial^k Q -L\Delta \partial^k Q \right)\\
		&+ \xi L\left( [\partial^k,( \Delta Q, Q)]_{+}:\partial^k\nabla u \right)-2 \xi L\left(  Q\partial^k\left\lbrace {\rm Tr}(Q\Delta Q)\right\rbrace -Q  {\rm Tr}(Q\partial^k\Delta Q):\partial^k\nabla u \right)\\
		&-2 \xi\left( \partial^k  \left\lbrace Q{\rm Tr}(Q\nabla u)\right\rbrace-    Q\partial^k\left\lbrace {\rm Tr}(Q\nabla u)\right\rbrace:a\partial^k Q -L\Delta \partial^k Q \right)\\
		&-2 \xi\left(  Q\partial^k\left\lbrace {\rm Tr}(Q\nabla u)\right\rbrace -Q  {\rm Tr}(Q\partial^k\nabla u):a\partial^k Q -L\Delta \partial^k Q \right)\\
		&-2 \xi L\left( \partial^k  \left\lbrace Q{\rm Tr}(Q\Delta Q)\right\rbrace-    Q\partial^k\left\lbrace {\rm Tr}(Q\Delta Q)\right\rbrace:\partial^k\nabla u \right)\\
		\overset{\bigtriangleup}{=}& \sum_{j=1}^{19} N_j.
	\end{align*}
	The same as \eqref{ap_a6}, some higher-order nonlinear terms have been canceled by Lemma~\ref{lem.cancel}. Then, by using H\"{o}lder inequality, Sobolev embedding inequality and the Moser-type calculus inequalities, one can estimate $N_j$ $(1\leq j\leq 19)$ as follows:
	\begin{align*}
		|N_1|\lesssim & \left( \|\nabla u\|_{L^\infty} \|\partial^{k-1} \nabla  Q\|_{L^2}+ \|\nabla  Q\|_{L^\infty}\|\partial^{k} u\|_{L^2}\right)\| \partial^kQ \|_{H^2}\\
		\lesssim & \left( \|\nabla u\|_{H^s} \| \nabla  Q\|_{H^{s-1}}+ \|\nabla  Q\|_{H^s}\|  u\|_{H^s}\right)\| \partial^kQ \|_{H^2}\\
		\leq & C\mathcal{E}^{\frac{1}{2}}(t)\mathcal{D}(t)\\
		|N_2|+ |N_4| + |N_7|\lesssim &  \| u\|_{L^\infty}\|\partial^k\nabla Q\|_{L^2}\|\partial^k\Delta
		Q\|_{L^2} + \| Q\|_{L^\infty}\|\partial^{k}\nabla u\|_{L^2}\|\partial^k Q  \|_{L^2}\\
		\leq & C\mathcal{E}^{\frac{1}{2}}(t)\mathcal{D}(t)\\
		|N_3|\lesssim & \left( \|\nabla u \|_{L^\infty}\|\partial^{k-1}\nabla u \|_{L^2}+\|\nabla u \|_{L^\infty} \|\partial^{k} u \|_{L^2}\right)  \|\partial^{k}  u \|_{L^2}\\
		\lesssim & \mathcal{D}^{\frac{1}{2}}(t) \left( \|\nabla u \|_{H^s}\| \nabla u \|_{H^{s-1}}+\|\nabla u\|_{H^s} \|  u\|_{H^s}\right)  \quad (|k|\geq 1)\\
		\leq & C\mathcal{E}^{\frac{1}{2}}(t)\mathcal{D}(t)\\
		|N_5|+ |N_{14}|\lesssim &  \left( \|\nabla Q\|_{L^\infty}\|\partial^{k} u \|_{L^2}+ \|\nabla u \|_{L^\infty}\|\partial^{k}Q \|_{L^2} \right)\| \partial^kQ \|_{H^2}\\
		\lesssim & \mathcal{D}^{\frac{1}{2}}(t)\left( \|\nabla Q \|_{H^s}\| u \|_{H^s}+ \|\nabla u \|_{H^s}\| Q \|_{H^s} \right)\\
		\leq & C\mathcal{E}^{\frac{1}{2}}(t)\mathcal{D}(t)\\
		|N_6|\lesssim & \| \nabla Q  \|_{L^\infty}  \|\partial^{k}\nabla Q  \|_{L^2} \|\partial^{k}\nabla u   \|_{L^2} \\
		\leq &  C\mathcal{E}^{\frac{1}{2}}(t)\mathcal{D}(t)\\
		|N_{8}| + |N_{15}|\lesssim &   \left(  \|\nabla  Q \|_{L^\infty}  \|\partial^{k-1}\Delta Q \|_{L^2} +  \| \Delta Q \|_{L^\infty}  \|\partial^{k}  Q \|_{L^2}  \right) \|\partial^{k}\nabla u  \|_{L^2}\\
		\lesssim &   \left(  \|\nabla  Q  \|_{H^s}  \|\Delta Q  \|_{H^{s-1}} +  \| \Delta Q \|_{H^s}  \| Q \|_{H^s}  \right) \|\nabla u \|_{H^s}\\
		\leq &  C\mathcal{E}^{\frac{1}{2}}(t)\mathcal{D}(t)\\
		\sum_{i=9}^{13}|N_i| \lesssim &  \|\nabla u\|_{H^s} \sum_{m=1}^4 \|Q\|_{L^\infty}^m\|\partial^kQ\|_{L^2}  +  \| \partial^k Q\|_{H^2}\sum_{m=1}^2 \|Q\|_{L^\infty}^m\|\partial^kQ\|_{L^2}\\
		\leq & C\sum_{m=1}^{4}\mathcal{E}^{\frac{m}{2}}(t)\mathcal{D}(t)\\
		\sum_{i=16}^{19}|N_i| \lesssim  & \|\nabla u\|_{H^s} \Big( \|\nabla Q\|_{L^\infty}\|\Delta Q\|_{L^\infty}\|\partial^{k-1}Q\|_{L^2} + \|\nabla Q\|_{L^\infty}\|Q\|_{L^\infty}\|\partial^{k-1}\Delta Q\|_{L^2} \\
		&+ \|Q\|_{L^\infty}\|\Delta Q\|_{L^\infty}\|\partial^kQ\|_{L^2}\Big) + \| \partial^k Q\|_{H^2}\Big( \|\nabla Q\|_{L^\infty}\|\nabla u\|_{L^\infty}\|\partial^{k-1}Q \|_{L^2} \\
		& + \|\nabla Q\|_{L^\infty}\|Q\|_{L^\infty}\|\partial^{k-1}\nabla u\|_{L^2} + \|Q\|_{L^\infty}\|\nabla u\|_{L^\infty}\|\partial^kQ\|_{L^2}  \Big) \\
		\leq  & C\mathcal{E}(t)\mathcal{D}(t).
	\end{align*}
	Therefore, we get
	\begin{equation}\label{eq5.8}
		\begin{split}
			\dfrac{1}{2}\dfrac{d}{dt}\Big( \|\partial^ku\|_{L^2}^2 +  a\|\partial^kQ\|_{L^2}^2 + & L\|\partial^k\nabla Q\|_{L^2}^2 \Big)
			+\mu\|\partial^{k}\nabla u\|_{L^2}^2 + a^2\Gamma \|\partial^{k} Q\|_{L^2}^2\\
			&+ 2aL\Gamma\|\partial^{k}\nabla Q\|_{L^2}^2 + \Gamma L^2\|\partial^k\Delta
			Q\|_{L^2}^2\\
			\leq& C \sum_{m=1}^{4}\mathcal{E}^{\frac{m}{2}}(t)  \mathcal{D}(t) .
		\end{split}
	\end{equation}
\end{enumerate}
Finally, summing up \eqref{eq5.7} and \eqref{eq5.8} for all $1\leq |k|\leq s$, we obtain \eqref{es-global} and this completes the proof.
\end{proof}

By the standard continuation argument \cite{nishida1978}, together with the a priori estimate \eqref{es-global}, we claim that the Cauchy problem of system \eqref{eq1.1} admits a unique global solution. This completes the proof of global well-posedness result in Theorem~\ref{thm3.1}.

%
%
%
%


\begin{appendices}
	\section{Local well-posedness}\label{ap1}
	In this appendix, we show the local existence and uniqueness of solutions to system \eqref{eq1.1} in $\R^3$ with arbitrary $\xi\in\R$, see Theorem~\ref{thm5.1}. Although the proof herein is obtained by a similar argument as in \cite{active-limit,ALC-decay}, we will present full details for what concerns some new nonlinearities arising from the non-zero physical coefficient $\xi$. Indeed, by deriving some a priori estimate for approximation system \eqref{eq-ap}, the desired result mainly follows by the standard compactness arguments. Thus, we  start by defining the following energy functionals for the solution $(Q^{n+1},u^{n+1})$ to system \eqref{eq-ap} with initial data $(Q_0,u_0)\in H^{s+1} \times H^{s}$:
	\begin{equation}\label{eq-f1}
	\begin{split}
	\mathbb{E}_n(t):=& \|u^n(t)\|_{H^s}^2 + \|Q^n(t)\|_{H^{s}}^2 + L\|\nabla Q^n(t)\|_{H^{s}}^2,\\
	\mathbb{D}_n(t):=& \mu\|\nabla u^n(t)\|_{H^{s}}^2 +(|a|+L+ |a|L)\Gamma\|Q^n(t)\|_{H^{s+1}}^2 + \Gamma L^2\|\Delta
	Q^n(t)\|_{H^{s}}^2.
	\end{split}
	\end{equation}
	
	Moreover, we state the following commutator estimates (For a proof, see \cite[Lemma~2.1]{ALC-decay}).
	\begin{lemma}\label{new-com}
		Let $\psi\in H^s(\R^3)$ and $\phi,\Phi\in H^{s+1}(\R^3)$ with $s\geq 2$. For all $|k|\leq s$, we have
		\begin{enumerate}[(i).]
			\item $\left\|\partial^k(\psi \nabla \phi) -\psi \partial^k\nabla \phi \right\|_{L^2}\leq C \|\psi\|_{H^s}\|\phi\|_{H^{s+1}}$,
			\item $\left\|\partial^k(\phi \nabla\psi ) -\phi \partial^k\nabla \psi \right\|_{L^2}\leq C \|\psi\|_{H^s}\|\phi\|_{H^{s+1}}$,
			\item $\left\|\partial^k(\phi \Delta \Phi) -\phi \partial^k\Delta \Phi \right\|_{L^2} \leq C \|\phi\|_{H^{s+1}}\|\Phi\|_{H^{s+1}}$.
		\end{enumerate}
	\end{lemma}

	Then, by exploiting an energy argument, one can derive the a priori estimate for system \eqref{eq-ap}, which is stated as follows:
	\begin{prop}\label{prop-a1}
		Let $s\geq 2$ be an integer and $(Q^{n+1},u^{n+1})(t,x)$ be the solution to Cauchy problem of system \eqref{eq-ap} with initial data $(Q_0,u_0)\in  H^{s+1}\times H^{s} $ in time interval $[0,T]$. Then, for all $t\in [0,T]$, one has
		\begin{equation}\label{en-loc}
			\begin{split}
				\dfrac{1}{2}\dfrac{d}{dt}\mathbb{E}_{n+1}(t)+ \mathbb{D}_{n+1}(t) \leq C \left[\mathbb{E}_{n+1}(t) + \sum_{l=0}^4 \mathbb{E}_n^{\frac{l}{2}}(t)\mathbb{E}_{n+1}^{\frac{1}{2}}(t)\mathbb{D}_{n+1}^{\frac{1}{2}}(t) +\sum_{l=2}^4\mathbb{E}_n^{\frac{l}{2}}(t)\mathbb{D}_{n+1}^{\frac{1}{2}}(t)\right],
			\end{split}
		\end{equation}
		where $C$ is a positive constant depending only on $E_0$, $s$, $|a|$, $b$, $c$, $\mu$, $|\xi|$, $L$ and $\Gamma$.
	\end{prop}
	\begin{proof}
		First, we present the basic $L^2$-energy estimate for system \eqref{eq-ap}. Indeed, one can take the summation of the first equation in \eqref{eq-ap} multiplied by $Q^{n+1}-L\Delta Q^{n+1}$ and the second equation in \eqref{eq-ap} multiplied by $u^{n+1}$, take the trace, and then integrate by parts over $\R^3$ to get
		\begin{align}\label{eq_a2}
			\dfrac{1}{2}\dfrac{d}{dt} &\Big(\|u^{n+1}\|_{L^2}^2 + \|Q^{n+1}\|_{L^2}^2 + L\|\nabla Q^{n+1}\|_{L^2}^2 \Big)\nonumber\\
			&+\mu\|\nabla u^{n+1}\|_{L^2}^2  + (|a|+L+ |a|L)\Gamma\|Q^{n+1}(t)\|_{H^{1}}^2  + \Gamma L^2\|\Delta
			Q^{n+1}\|_{L^2}^2\nonumber\\
			=&(\eta+L+aL)\Gamma\|Q^{n+1}(t)\|_{L^2}^2 + (\eta+a)\Gamma\|\nabla Q^{n+1}(t)\|_{L^2}^2 \underbrace{-\left(  u^n\cdot\nabla Q^{n+1}:Q^{n+1}  \right)  -\left(  u^n\cdot\nabla u^{n+1},u^{n+1}  \right)}_{\mathcal{Y}_1} \nonumber\\
			&+L\left(  u^n\cdot\nabla Q^{n+1}:\Delta Q^{n+1}  \right) + \left( \Omega^{n+1} Q^n - Q^n\Omega^{n+1}:  Q^{n+1}  \right)  + L\left(  \nabla Q^n \odot\nabla Q^n:\nabla u^{n+1}  \right)\nonumber\\
			&\underbrace{-L\left( \Omega^{n+1} Q^n - Q^n\Omega^{n+1}:\Delta Q^{n+1}  \right)}_{\mathcal{Y}_2}  \underbrace{-  L\left( Q^n\Delta Q^{n+1}-\Delta Q^{n+1}Q^n :\nabla u^{n+1}   \right)}_{\mathcal{Y}_3} \\
			& +\xi\left( \left[Q^n+\dfrac{\mathbb{I}_3}{3} \right] H^{n+1}+H^{n+1} \left[Q^n+\dfrac{\mathbb{I}_3}{3} \right] :\nabla u^{n+1}   \right) +\Gamma  \left(  M^{n+1}:Q^{n+1} -L\Delta Q^{n+1}  \right)\nonumber\\
			&  +\xi\left( D^{n+1} \left[Q^n+\dfrac{\mathbb{I}_3}{3} \right] + \left[Q^n+\dfrac{\mathbb{I}_3}{3} \right]D^{n+1} :Q^{n+1}-L\Delta Q^{n+1}   \right) \nonumber\\
			&  -2\xi\left( Q^n{\rm Tr}(Q^nH^{n+1}) :\nabla u^{n+1}   \right)- 2\xi\left( Q^n{\rm Tr}(Q^n\nabla u^{n+1}) :Q^{n+1}-L\Delta Q^{n+1}   \right)\nonumber.
		\end{align}
		where the constant $\eta$ is defined by
		\[
		\eta=
		\begin{cases}
			0, & \text{if $a\geq 0$,}\\
			2(1+L)|a|, & \text{if $a< 0$.}\\
		\end{cases}
		\]
			According to Lemma~\ref{lem.cancel}, one can see that $\mathcal{Y}_1=\mathcal{Y}_2+\mathcal{Y}_3=0$. Thus,
			\eqref{eq_a2} can be simplified as
			\begin{equation}\label{eq.a3}
				\begin{split}
					&\dfrac{1}{2}\dfrac{d}{dt}\Big( \|u^{n+1}\|_{L^2}^2 + \|Q^{n+1}\|_{L^2}^2 + L\|\nabla Q^{n+1}\|_{L^2}^2 \Big)\\
					&+\mu\|\nabla u^{n+1}\|_{L^2}^2  + (|a|+L+ |a|L)\Gamma\|Q^{n+1}(t)\|_{H^{1}}^2  + \Gamma L^2\|\Delta
					Q^{n+1}\|_{L^2}^2\\
					=&(\eta+L+aL)\Gamma\|Q^{n+1}(t)\|_{L^2}^2 + (\eta+a)\Gamma\|\nabla Q^{n+1}(t)\|_{L^2}^2+\dfrac{2\xi}{3}(1-a)\left( Q^{n+1} :\nabla u^{n+1} \right)\\
					&+L\left(  u^n\cdot\nabla Q^{n+1}:\Delta Q^{n+1}  \right)  + \left(  \Omega^{n+1} Q^n - Q^n\Omega^{n+1}: Q^{n+1}  \right) + L\left(  \nabla Q^n \odot\nabla Q^n:\nabla u^{n+1}  \right) \\
					&  +(1-a)\xi \left( Q^nQ^{n+1} + Q^{n+1}Q^n: \nabla u^{n+1} \right) + \dfrac{2\xi}{3}\left( M^{n+1}:\nabla u^{n+1} \right) \\
					& +\Gamma  \left(  M^{n+1}:Q^{n+1} -L\Delta Q^{n+1}  \right)  +\xi\left( M^{n+1} Q^n + Q^nM^{n+1} :\nabla u^{n+1} \right) \\
					&-2(1-a)\xi\left( Q^n{\rm Tr}(Q^n\nabla u^{n+1}) :Q^{n+1}  \right) -2\xi\left( Q^n{\rm Tr}(Q^n\nabla u^{n+1}): M^{n+1} \right)  \\
					\overset{\bigtriangleup}{=}&(\eta+L+aL)\Gamma\|Q^{n+1}(t)\|_{L^2}^2 + (\eta+a)\Gamma\|\nabla Q^{n+1}(t)\|_{L^2}^2 + \sum_{i=1}^{10} J_i.
				\end{split}
			\end{equation}
			where we used the cyclic property for the trace of three symmetric matrices.
			By applying H\"{o}lder inequality, Sobolev embedding inequality and Gagliardo-Nirenberg inequality, we immediately get
			\begin{align*}
				|J_1|+|J_4|\leq& \|\nabla Q^{n}\|_{L^4}^2\|\nabla u^{n+1}\|_{L^2} + \|Q^{n+1}\|_{L^2}\|\nabla u^{n+1}\|_{L^2}\\
				\leq& C\mathbb{E}_n(t) \mathbb{D}_{n+1}^{\frac{1}{2}}(t) + \mathbb{E}_{n+1}^{\frac{1}{2}}(t) \mathbb{D}_{n+1}^{\frac{1}{2}}(t)\\
				|J_2|+ |J_3| + |J_5|\leq&\|u^n\|_{L^6}\|\nabla Q^{n+1}\|_{L^3}\|\Delta Q^{n+1}\|_{L^2}+ \| \nabla u^{n+1}\|_{L^2}\| Q^{n}\|_{L^3}\| Q^{n+1}\|_{L^6}\\
				\leq& \|\nabla u^n\|_{L^2}\|\nabla Q^{n+1}\|_{H^1}\|\Delta Q^{n+1}\|_{L^2}+ \|\nabla u^{n+1}\|_{L^2}\| Q^{n}\|_{H^1}\| \nabla Q^{n+1}\|_{L^2}\\
				\leq& C\mathbb{E}_n^{\frac{1}{2}}(t) \mathbb{E}_{n+1}^{\frac{1}{2}}(t)\mathbb{D}_{n+1}^{\frac{1}{2}}(t),\\
				\sum_{i=6}^{10}|J_i|\leq& C  \left( \|Q^n\|_{L^4}^2+ \|Q^n\|_{L^6}^2 \|Q^{n+1}\|_{L^6} \right) \left( \|Q^{n+1}\|_{H^2} + \|\nabla u^{n+1}\|_{L^2} \right) \\
				&+ C \|\nabla u^{n+1}\|_{L^2}\left( \|Q^n\|_{L^\infty}+\|Q^n\|_{L^\infty}^2\right) \left(\|Q^n\|_{L^4}^2+ \|Q^n\|_{L^6}^2 \|Q^{n+1}\|_{L^6} \right)\\
				\leq& CD_{n+1}^{\frac{1}{2}}(t)\left( \|Q^n\|_{H^1}^2+ \|\nabla Q^n\|_{L^2}^2\|\nabla Q^{n+1}\|_{L^2}  \right)\left( 1+\|Q^n\|_{H^2}+ \|Q^n\|_{H^2}^2 \right)\\
				\leq& C\left( 1+\mathbb{E}_n^{\frac{1}{2}}(t)+\mathbb{E}_n(t) \right)\left( \mathbb{E}_n(t)+\mathbb{E}_n(t)\mathbb{E}_{n+1}^{\frac{1}{2}}(t) \right)\mathbb{D}_{n+1}^{\frac{1}{2}}(t).
			\end{align*}
			Substituting these estimates into \eqref{eq.a3}, one has
			\begin{equation}\label{eq:0-th}
				\begin{split}
					\dfrac{1}{2}\dfrac{d}{dt}\Big( \|u^{n+1}\|_{L^2}^2 +& \|Q^{n+1}\|_{L^2}^2 + L\|\nabla Q^{n+1}\|_{L^2}^2 \Big)\\
					&+\mu\|\nabla u^{n+1}\|_{L^2}^2  + (|a|+L+ |a|L)\Gamma\|Q^{n+1}(t)\|_{H^{1}}^2  + \Gamma L^2\|\Delta
					Q^{n+1}\|_{L^2}^2\\
					\leq& C\mathbb{E}_{n+1}(t) +  C\sum_{l=2}^4\mathbb{E}_n^{\frac{l}{2}}(t) \mathbb{D}_{n+1}^{\frac{1}{2}}(t)+  C\sum_{l=0}^4\mathbb{E}_n^{\frac{l}{2}}(t)  \mathbb{E}_{n+1}^{\frac{1}{2}}(t) \mathbb{D}_{n+1}^{\frac{1}{2}}(t).
				\end{split}
			\end{equation}
			Next, we show the higher order estimates for system \eqref{eq-ap}.
			Similarly,
			applying $\partial^k$ ($1\leq |k|\leq s$) to the \eqref{eq-ap}$_2$ and multiplying the resulting equation by $\partial^ku^{n+1}$, meanwhile,
			applying $\partial^k$ to the \eqref{eq-ap}$_1$, and multiplying the resulting equation by $\partial^kQ^{n+1}-L\partial^k\Delta Q^{n+1}$,
			taking the trace, and then summing up all the results to have
			\begin{align*}
					\dfrac{1}{2}\dfrac{d}{dt}&\Big( \|\partial^ku^{n+1}\|_{L^2}^2 + \|\partial^kQ^{n+1}\|_{L^2}^2 + L\|\partial^k\nabla Q^{n+1}\|_{L^2}^2 \Big)\\
					&+\mu\|\partial^{k}\nabla u^{n+1}\|_{L^2}^2 + (|a|+L+ |a|L)\Gamma \|\partial^{k} Q^{n+1}\|_{H^1}^2+ \Gamma L^2\|\partial^k\Delta
					Q^{n+1}\|_{L^2}^2\\
					=& (\eta+L+aL)\Gamma\|\partial^k Q^{n+1}\|_{L^2}^2 + (\eta + a)\Gamma\|\partial^k\nabla Q^{n+1}\|_{L^2}^2 +L\left( u^n\cdot\nabla\partial^k Q^{n+1}:\partial^k\Delta
					Q^{n+1}  \right)   \\
					& + \left( [\partial^k, u^n]_{\rm div} Q^{n+1}:L\partial^k\Delta
					Q^{n+1}-\partial^kQ^{n+1}  \right)  +\left(\partial^k(u^n\otimes u^{n+1}):\partial^k\nabla u^{n+1} \right)\\
					& +\left( \partial^k\Omega^{n+1} Q^n -  Q^n\partial^k\Omega^{n+1}:\partial^k Q^{n+1} \right)+\left([\partial^k,(\Omega^{n+1}, Q^n)]_{-}:\partial^k Q^{n+1}-L\partial^k\Delta
					Q^{n+1} \right)  \\
					& + L \left(  \partial^k\left( \nabla Q^n \odot\nabla Q^n\right):\partial^{k}\nabla u^{n+1}  \right)  +  L\left([\partial^k,( \Delta Q^{n+1},Q^n)]_{-} :\partial^{k}\nabla u^{n+1}   \right) \\
					& + \dfrac{2\xi}{3}(1-a)\left( \partial^k Q^{n+1} :\partial^k\nabla u^{n+1}  \right) +\Gamma  \left( \partial^k M^{n+1}:\partial^kQ^{n+1} -L\partial^k\Delta Q^{n+1}  \right)\\
					& + \dfrac{2\xi}{3}\left( \partial^k M^{n+1} :\partial^k\nabla u^{n+1}  \right) + \xi\left( \partial^k \left\lbrace Q^n(M^{n+1} - aQ^{n+1}) + (M^{n+1} - aQ^{n+1}) Q^n \right\rbrace :\partial^k\nabla u^{n+1}  \right)\\
					& + \xi\left(  \partial^kD^{n+1}Q^n + Q^n \partial^kD^{n+1} : \partial^k Q^{n+1}\right)  -2 \xi\left( \partial^k \left\lbrace Q^{n} {\rm Tr}(Q^nM^{n+1}-aQ^nQ^{n+1}) \right\rbrace :\partial^k\nabla u^{n+1}  \right)\\
					& + \xi L\left([\partial^k, (\Delta Q^{n+1},Q^n )]_{+}:\partial^k\nabla u^{n+1}  \right)	+ \xi\left([\partial^k, (D^{n+1},Q^n)]_{+} : \partial^k Q^{n+1} -L\partial^k\Delta Q^{n+1} \right)\\
					&-2 \xi L\left( [\partial^k , Q^{n}] {\rm Tr}(Q^n\Delta Q^{n+1})  :\partial^k\nabla u^{n+1}  \right)  -2 \xi\left( [\partial^k , Q^{n}] {\rm Tr}(Q^n\nabla u^{n+1}): \partial^k Q^{n+1} -L\partial^k\Delta Q^{n+1}  \right)\\
					&-2 \xi L\left( Q^{n}\left(\partial^k\left\lbrace  {\rm Tr}(Q^n\Delta Q^{n+1})\right\rbrace -{\rm Tr}(Q^n\partial^k\Delta Q^{n+1})\right) :\partial^k \nabla u^{n+1}  \right)\\
					& -2 \xi\left(  Q^{n}\left(\partial^k\left\lbrace {\rm Tr}(Q^n\nabla u^{n+1}) \right\rbrace-  {\rm Tr}(Q^n\partial^k\nabla u^{n+1}) \right): \partial^k Q^{n+1} -L\partial^k\Delta Q^{n+1}  \right)\\
					& - 2\xi \left(  Q^{n}{\rm Tr}(Q^n\partial^k\nabla u^{n+1}): \partial^k Q^{n+1} \right)\\
					\overset{\bigtriangleup}{=}& (\eta+L+aL)\Gamma\|\partial^k Q^{n+1}\|_{L^2}^2 + (\eta + a)\Gamma\|\partial^k\nabla Q^{n+1}\|_{L^2}^2+  \sum_{j=1}^{20} K_j,
			\end{align*}
			where we have used $\nabla \cdot u^{n+1}={\rm Tr}\,Q^{n+1}=0$ and the following fact:
			\begin{equation}\label{ap_a6}
			\begin{split}
				&\left( Q^n \partial^k\Omega^{n+1} - \partial^k\Omega^{n+1} Q^n:\partial^k\Delta Q^{n+1}\right) -\left( Q^n\partial^k \Delta Q^{n+1} - \partial^k \Delta Q^{n+1} Q^n: \partial^k\nabla u^{n+1} \right) =0.\\
			\end{split}
			\end{equation}
			Then, together with Lemma~\ref{new-com} and Lemma~\ref{lem-a1},
			one can estimate $K_j$ $(1\leq j\leq 20)$ as follows:
			\begin{align*}
				|K_1|+|K_4| +|K_{12}|\lesssim & \|u^n\|_{L^\infty}\|\partial^k \nabla Q^{n+1}\|_{L^2}\|\partial^k\Delta  Q^{n+1}\|_{L^2} + \|Q^n\|_{L^\infty}\|\partial^k\nabla  u^{n+1}\|_{L^2}\| \partial^k Q^{n+1}\|_{L^2}\\
				\lesssim& \mathbb{E}_n^{\frac{1}{2}}(t)  \mathbb{E}_{n+1}^{\frac{1}{2}}(t) \mathbb{D}_{n+1}^{\frac{1}{2}}(t),\\
				|K_2|+ |K_3|\lesssim & \| u^n\|_{H^s}\|  Q^{n+1}\|_{H^{s+1}} \| \partial^k Q^{n+1}\|_{H^2} +  \| u^{n}\|_{L^\infty}\|\partial^{k} u^{n+1}\|_{L^2} \|\partial^{k}\nabla  u^{n+1}\|_{L^2}\\
				& + \| u^{n+1}\|_{L^\infty}\|\partial^{k} u^{n}\|_{L^2} \|\partial^{k}\nabla  u^{n+1}\|_{L^2}\\
				\lesssim&  \mathbb{E}_n^{\frac{1}{2}}(t)\mathbb{E}_{n+1}^{\frac{1}{2}}(t)\mathbb{D}_{n+1}^{\frac{1}{2}}(t),\\
				|K_5|+ |K_{7}|+ |K_{14}|+ |K_{15}|\lesssim&  \|Q^n\|_{H^{s+1}}\| u^{n+1}\|_{H^s} \| \partial^k Q^{n+1}\|_{H^2}+ \| Q^n \|_{H^{s+1}}  \| Q^{n+1} \|_{H^{s+1}}  \|\partial^{k}\nabla u^{n+1} \|_{L^2} \\
				\lesssim&   \mathbb{E}_n^{\frac{1}{2}}(t)\mathbb{E}_{n+1}^{\frac{1}{2}}(t)\mathbb{D}_{n+1}^{\frac{1}{2}}(t),\\
				|K_6| + |K_8|\lesssim &   \|\partial^{k}\nabla Q^n \|_{L^2} \|\nabla Q^n \|_{L^\infty} \|\partial^{k}\nabla u^{n+1}  \|_{L^2}+ \|\partial^{k}Q^{n+1}  \|_{L^2}\|\partial^{k}\nabla u^{n+1}  \|_{L^2}\\
				\lesssim& \|  \nabla Q^n \|_{H^{s}} \|\nabla Q^n \|_{H^{2}} \|\nabla u^{n+1}  \|_{H^s}+ \mathbb{E}_{n+1}^{\frac{1}{2}}(t)\mathbb{D}_{n+1}^{\frac{1}{2}}(t)\\
				\lesssim & \mathbb{E}_n(t)\mathbb{D}_{n+1}^{\frac{1}{2}}(t)+ \mathbb{E}_{n+1}^{\frac{1}{2}}(t)\mathbb{D}_{n+1}^{\frac{1}{2}}(t),\\
				|K_9| + |K_{10}| \lesssim & \mathbb{D}_{n+1}^{\frac{1}{2}}(t)\left((1+ \| Q^{n+1}\|_{L^\infty})\| Q^n\|_{L^\infty}\|  \partial^k  Q^n  \|_{L^2} + \| Q^{n}\|_{L^\infty}^2\|  \partial^k  Q^{n+1}  \|_{L^2}\right)\\
				\lesssim & \mathbb{E}_n(t)\mathbb{D}_{n+1}^{\frac{1}{2}}(t)+ \mathbb{E}_n(t)\mathbb{E}_{n+1}^{\frac{1}{2}}(t)\mathbb{D}_{n+1}^{\frac{1}{2}}(t),\\
				|K_{11}| + |K_{13}|+ |K_{20}|\lesssim &  \|\partial^k\nabla u^{n+1}\|_{L^2} \Big(\|Q^n\|_{L^\infty}^2 \|\partial^kQ^n\|_{L^2} +\|Q^n\|_{L^\infty}^3 \|\partial^kQ^n\|_{L^2} \\
				&+\sum_{l=1}^4\|Q^n\|_{L^\infty}^l\|\partial^kQ^{n+1}\|_{L^2} + \sum_{l=0}^3\|Q^{n+1}\|_{L^\infty}\|Q^n\|_{L^\infty}^l\|\partial^kQ^{n}\|_{L^2} \Big)\\
				\lesssim  & \left(\mathbb{E}_n^{\frac{3}{2}}(t) +\mathbb{E}_n^2(t) \right)\mathbb{D}_{n+1}^{\frac{1}{2}}(t)+\sum_{l=1}^4 \mathbb{E}_n^{\frac{l}{2}}(t)\mathbb{E}_{n+1}^{\frac{1}{2}}(t)\mathbb{D}_{n+1}^{\frac{1}{2}}(t),\\
				\sum_{i=16}^{19}|K_i| \lesssim  & \mathbb{D}_{n+1}^{\frac{1}{2}}(t) \Big( \|Q^{n}\|_{H^{s+1}}^2\| Q^{n+1}\|_{H^{s+1}} +\|Q^{n}\|_{H^{s+1}}^2\| u^{n+1}\|_{H^{s}} \\
				& + \|Q^{n}\|_{L^\infty}\| Q^{n}\|_{H^{s+1}}\| u^{n+1}\|_{H^s}+ \|Q^{n}\|_{L^\infty}\| Q^{n}\|_{H^{s+1}}\| Q^{n+1}\|_{H^{s+1}}  \Big) \\
				\lesssim  &  \mathbb{E}_n(t)\mathbb{E}_{n+1}^{\frac{1}{2}}(t)\mathbb{D}_{n+1}^{\frac{1}{2}}(t).
			\end{align*}
			From all the above estimates, we then get
			\begin{align}\label{eq:k-th}
				\dfrac{1}{2}\dfrac{d}{dt}\Big( &\|\partial^ku^{n+1}\|_{L^2}^2 + \|\partial^kQ^{n+1}\|_{L^2}^2 + L\|\partial^k\nabla Q^{n+1}\|_{L^2}^2 \Big)\nonumber\\
				&+\mu\|\partial^{k}\nabla u^{n+1}\|_{L^2}^2 + (|a|+L+ |a|L)\Gamma \|\partial^{k} Q^{n+1}\|_{H^1}^2+ \Gamma L^2\|\partial^k\Delta
				Q^{n+1}\|_{L^2}^2\\
				\lesssim&   \mathbb{E}_{n+1}(t) + \sum_{l=0}^4 \mathbb{E}_n^{\frac{l}{2}}(t)\mathbb{E}_{n+1}^{\frac{1}{2}}(t)\mathbb{D}_{n+1}^{\frac{1}{2}}(t) +\sum_{l=2}^4\mathbb{E}_n^{\frac{l}{2}}(t)\mathbb{D}_{n+1}^{\frac{1}{2}}(t).\nonumber
			\end{align}
			Finally, summing up \eqref{eq:0-th} and \eqref{eq:k-th} for all $1\leq |k|\leq s$, we get \eqref{en-loc} and this completes the proof.
	\end{proof}
	Now, we are in a position to prove the local well-posedness of \eqref{eq1.1} for the case of $\xi\in \R$.
	
	\noindent\textbf{Proof of Theorem~\ref{thm5.1}.}
	
	\noindent\text{\em Step~1: Local existence.}
	First of all, we apply Young's inequality to get
	\[
		\begin{split}
			\sum_{l=0}^4 \mathbb{E}_n^{\frac{l}{2}}(t)\mathbb{E}_{n+1}^{\frac{1}{2}}(t)\mathbb{D}_{n+1}^{\frac{1}{2}}(t) \leq& 16C \sum_{l=0}^4\mathbb{E}_n^{l}(t)\mathbb{E}_{n+1}(t)+\dfrac{5}{16C} \mathbb{D}_{n+1}(t),\\
			 \sum_{l=2}^4\mathbb{E}_n^{\frac{l}{2}}(t)\mathbb{D}_{n+1}^{\frac{1}{2}}(t)\leq& 16C \sum_{l=2}^4\mathbb{E}_n^{l}(t)+\dfrac{3}{16C} \mathbb{D}_{n+1}(t),
		\end{split}
	\]
	where $C$ is a positive constant from Proposition~\ref{prop-a1}. Then, it follows from \eqref{en-loc} that
	\begin{equation}\label{ap-a6}
		\dfrac{d}{dt}\mathbb{E}_{n+1}(t)+ \mathbb{D}_{n+1}(t)
		\leq \widetilde{C}\sum_{l=0}^4\mathbb{E}_n^{l}(t)\mathbb{E}_{n+1}(t) + \widetilde{C}\sum_{l=2}^4\mathbb{E}_n^{l}(t),
	\end{equation}
	where $\widetilde{C}:= 2C(16C+1)$.
	
	Suppose now $
	(Q^0,u^0)(t,x)=(Q_0,u_0)(x)$ for all $t\in [0,T]$, one has $\mathbb{E}_{0}(t)\leq 4E_0$.
	With the induction assumption $\sup_{t\in [0,T]} \mathbb{E}_{n}(t)\leq 4E_0$ in mind, and applying the Gr\"{o}nwall's inequality to \eqref{ap-a6}, we then get
	\begin{equation}
		\begin{split}
			\mathbb{E}_{n+1}(t)+ \int_0^t\mathbb{D}_{n+1}(\tau)d\tau &\leq \left( \mathbb{E}_{n+1}(0) + \widetilde{C} \sum_{l=2}^4\int_0^t\mathbb{E}_n^{l}(\tau) d\tau \right)
			\exp\left\lbrace \widetilde{C}\sum_{l=0}^4\int_0^t\mathbb{E}_n^{l}(\tau) d\tau  \right\rbrace\\
			&\leq \left( E_0 + \sum_{l=2}^4 4^{l}\widetilde{C}E_0^{l}t  \right)
			\exp\left\lbrace \sum_{l=0}^4 4^{l}\widetilde{C}E_0^{l}t \right\rbrace\\
		\end{split}	
	\end{equation}
	for all $0\leq t\leq T$.
	Obviously, one can choose some small $t$ so that
	\[
		\sum_{l=2}^4 4^{l}\widetilde{C}E_0^{l}t \leq E_0 \quad {\rm and}\quad
		\exp\left\lbrace \sum_{l=0}^4 4^{l}\widetilde{C}E_0^{l}t \right\rbrace  \leq 2.
	\]
	Then, there exists a time $T$ with
	\begin{equation}\label{eq:Tmax}
		T:=\dfrac{\min\{E_0, \ln 2\}}{\sum_{l=0}^4 4^{l}\widetilde{C}E_0^{l}}
	\end{equation}
	such that
	\[
		\sup_{t\in [0,T]}	\mathbb{E}_{n+1}(t)+ \int_0^T\mathbb{D}_{n+1}(t)dt \leq 4E_0.
	\]
	Then the above induction argument implies that the following energy inequality holds for all integer $n\geq 0$,
	\[
		\mathbb{E}_{n}(t)+ \int_0^t\mathbb{D}_{n}(\tau)d\tau \leq 4E_0,\quad \forall t\in [0,T].
	\]
	Now, the desired local existence result in Theorem~\ref{thm5.1} follows from the standard compactness arguments. 
	
	\noindent\text{\em Step~2: Uniqueness.}
	In this step,
	we consider the difference of any two solutions $(Q_1,u_1)$ and $(Q_2,u_2)$ belonging to the class specified in Theorem~\ref{thm5.1}. Then, by exploiting an energy argument jointly with the Gr\"{o}nwall's inequality, one can deduce our constructed solution is unique. We omit the details here for simplicity; see also \cite{Qtensor11,ALC-decay}.

	\noindent\text{\em Step~3: Trace free.}
	In order to show traceless property of $Q$, we take the trace on both sides of the first equation in system \eqref{eq1.1} to find
	\[
		\begin{cases}
			\partial_t {\rm Tr}Q+ u \cdot\nabla{\rm Tr} Q  =\Gamma  \left(L\Delta {\rm Tr}Q -a{\rm Tr}Q -c{\rm Tr}Q {\rm Tr}(Q^2)\right)-2\xi{\rm Tr}(Q\nabla u){\rm Tr}Q,\\
			{\rm Tr}Q|_{t=0}= {\rm Tr}Q_0=0,
		\end{cases}
	\]
	where we have used the fact that $\Omega^T=-\Omega$, ${\rm Tr}D=0$ and ${\rm Tr}(Q\Omega)=0$. Since the above system is similar to that of
	$\xi=0$. Thus we can get the following estimate via a similar argument as in \cite{weak-ALC},
	\[
		\dfrac{d}{dt}\|{\rm Tr}Q\|_{L^2}^2 + \Gamma\|\nabla {\rm Tr}Q\|_{L^2}^2 \lesssim\|{\rm Tr}Q\|_{L^2}^2.
	\]
	Applying the Gr\"{o}nwall's inequality again, we complete the proof of Theorem~\ref{thm5.1}.
	
	\section{Some basic theories and lemmas}\label{ap2}
	In this section, we review the inequalities that are used extensively in this paper.
	First, the well-known Moser-type calculus inequalities are stated in the following Lemma \cite{ms-eq}.
	\begin{lemma}\label{lem-a1}
		If $f$ and $g\in H^s$, then for all $|k|\leq s$ we have
		\begin{align*}
			\|\partial_x^k\left(fg\right)\|_{L^2}\leq& C\left( \|f\|_{L^\infty}\|\partial_x^k g\|_{L^2}  +  \|g\|_{L^\infty}\|\partial_x^k f\|_{L^2}\right),\\
			\|[\partial_x^k,f] g\|_{L^2}\leq& C\left( \|\partial_x f\|_{L^\infty}\|\partial_x^{k-1}g\|_{L^2}  +  \|g\|_{L^\infty}\|\partial_x^k f\|_{L^2}\right).
		\end{align*}
	\end{lemma}
	The following useful lemma regarding the cancellations rules of system \eqref{eq1.1} can be found in \cite{weak-ALC,Qtensor12,weak-strong-3d} and references therein.
	\begin{lemma}\label{lem.cancel}
		Let $u\in L_\sigma^2(\R^d)$ and $Q, G\in H^2(\R^d)\cap S_0^d$, and let
		\[
		D=\frac{1}{2}\left( \nabla u + \nabla u^{T} \right)\quad{\rm and}\quad\Omega=\frac{1}{2}\left( \nabla u - \nabla u^{T} \right).
		\]
		Then we have
		\begin{enumerate}[(1).]
			\item $(u\cdot\nabla u,u)= (u \cdot\nabla Q:Q)=0;$
			\item $(Q \Omega-\Omega Q:Q)=0;$
			\item $(u\cdot\nabla Q:\Delta Q)-\left( \nabla\cdot (\nabla Q \odot \nabla Q): u \right)=0 ;$
			\item $(G \Omega - \Omega G:\Delta Q) -\left( G \Delta Q - \Delta Q G: \nabla u \right) =0;$
			\item $(D Q+Q D:G) -\left( G  Q + Q G: \nabla u \right) =0;$
			\item $( Q:\nabla u)-(D: Q)=0;$
			
		\end{enumerate}
	\end{lemma}
	Next, we introduce the following well-known interpolation inequality:
	\begin{lemma}[Gagliardo-Nirenberg inequality]\label{lem.GN}
		Let $1 \leq q, r \leq  \infty$ and $m$ is a natural number. Suppose that there is a real number $\alpha$ and a natural number $j<m$ are such that
		\[
		\frac{1}{p} = \frac{j}{d} + \left( \frac{1}{r} - \frac{m}{d} \right) \alpha + \frac{1 - \alpha}{q}
		\]
		and
		\[
		\frac{j}{m} \leq \alpha \leq 1.
		\]
		In particular, under the folowing two exceptional cases:
		\begin{enumerate}[(i).]
			\item  If $j = 0, mr < d$ and $q = \infty$, then it is necessary to make the additional assumption that either $u$ tends to zero at infinity or that $u$ lies in $L^s$ for some finite $s > 0$.
			\item If $1 < r <\infty$ and $m-j- \frac{d}{r}$ is a non-negative integer, then it is necessary to assume also that $\alpha\ne 1$.
		\end{enumerate}
		Then it holds that
		\begin{equation}
			\| \mathrm{D}^{j} u \|_{L^{p}} \leq C \| \mathrm{D}^{m} u \|_{L^{r}}^{\alpha} \| u \|_{L^{q}}^{1 - \alpha},
		\end{equation}
		where $C$ is a constant depending only on $d$, $m$, $j$, $q$, $r$ and $\alpha$.
	\end{lemma}
	To simplify the proof of Theorem~\ref{thm1.2}, we close this section by the following lemma (see also \cite{weak-strong-3d}).
	\begin{lemma}\label{lemB4}
		Let $(Q,u)$ and $(R,v)$ be two Leray-Hopf weak solutions in Theorem~\ref{thm1.2}, then we have
		\[
			\begin{split}
			\mathcal{T}_2 + \mathcal{T}_3 \sim&\int_0^t (\Delta G:u\cdot\nabla G) -(\Delta Q:\omega\cdot\nabla G)-(G:\omega\cdot\nabla Q)ds,\\
			\mathcal{T}_5 + \mathcal{T}_7 \sim&\int_0^t\left(\widetilde{\Omega} Q-Q\widetilde{\Omega}: G\right) +  (G\Delta G-\Delta GG:\nabla u) - (G\Delta Q-\Delta QG:\nabla \omega)ds,\\
			\mathcal{T}_8 \sim&\int_0^t\left(G\nabla u: G \right)+ \left(G\nabla u: \Delta G\right)+ (GQ:\nabla \omega)+ (G\Delta Q:\nabla \omega)ds.
			\end{split}
		\]
	\end{lemma}
\end{appendices}

\end{document}